	\documentclass{siamltex}

	\usepackage[T1]{fontenc}
	\usepackage[english]{babel}
	\usepackage[utf8]{inputenc}

	\usepackage{amsmath,latexsym,amsfonts}
	\usepackage{graphicx}
	\usepackage{amssymb}
	\usepackage{bm}
	\usepackage{bbm}
	\usepackage{amsbsy}
	\usepackage{nicefrac}

	\usepackage[usenames,dvipsnames,svgnames]{xcolor}

	\usepackage[colorinlistoftodos,textwidth=2cm,textsize=small,backgroundcolor=red!10]{todonotes}

	\definecolor{colorCS}{RGB}{191,207,239}
	\newcounter{CS}
	
	\definecolor{colorVK}{RGB}{207,239,191}
	\newcounter{VK}

	\definecolor{colorMR}{RGB}{186,180,231}
	\newcounter{MR}
	
	\usepackage{regexpatch}
	\makeatletter
	\xpatchcmd{\@todo}{\setkeys{todonotes}{#1}}{\setkeys{todonotes}{inline,#1}}{}{}
	\makeatother

	\usepackage{graphicx}

	\usepackage{caption}
	\usepackage{subcaption}
	\usepackage{mathtools}

	\usepackage[cal=cm,calscaled=1.00,scr=boondoxo,scrscaled=1.00]{mathalfa}

	\usepackage{bm}

	\usepackage{tikz}
	\usepackage{tabularx}
	\newcolumntype{R}{>{\raggedleft\arraybackslash}X}
	\usepackage{hyperref}
	\usepackage{algorithm}
	\usepackage{algorithmic}

	\newcommand{\beq}{\begin{equation}}
	\newcommand{\eeq}{\end{equation}}
	\newcommand{\beqas}{\begin{eqnarray*}}
	\newcommand{\eeqas}{\end{eqnarray*}}
	\newcommand{\bbC}{{\mathbb C}}

	\newcommand{\bbR}{{\mathbb R}}
	\newcommand{\bbN}{{\mathbb N}}
	\newcommand{\bbZ}{{\mathbb Z}}

	\newcommand{\R}{\bbR}
	\newcommand{\N}{\bbN}
	\newcommand{\Z}{\bbZ}
	\newcommand{\C}{\bbC}
	\newcommand{\Nz}{\N_0}

	\newcommand{\Restr}[1]{\raise-0.5ex\hbox{$\mid$}_{#1}}
	\renewcommand{\d}{\,{\rm d}}
	\newcommand{\Jac}{\mathscr{J}}
	\newcommand{\MT}{\mathsf{T}}

	\newtheorem{assumption}[theorem]{Assumption}
	\newtheorem{remark}[theorem]{Remark}
	
	\usepackage{mleftright}

	\newcommand{\DelimiterGroup}[4]{
		\ifcase#1\relax
			#2 #4 #3	
		\or
			\bigl#2 #4 \bigr#3	
		\or
			\Bigl#2 #4 \Bigr#3	
		\or
			\biggl#2 #4 \biggr#3	
		\or
			\Biggl#2 #4 \Biggr#3	
		\or
		\or
		\or
			\mleft#2 #4 \mright#3	
		\or
			\left#2 #4 \right#3	
		\else
		\fi
	}

	\newcommand{\Par}[2][0]{\DelimiterGroup{#1}{(}{)}{#2}}
	\newcommand{\SqBr}[2][0]{\DelimiterGroup{#1}{[}{]}{#2}}
	\newcommand{\CuBr}[2][0]{\DelimiterGroup{#1}{\{}{\}}{#2}}

	\newcommand{\Abs}[2][0]{\DelimiterGroup{#1}{\lvert}{\rvert}{#2}}
	\newcommand{\IndNorm}[2][0]{\DelimiterGroup{#1}{\lvert}{\rvert}{#2}}
	\newcommand{\Norm}[2][0]{\DelimiterGroup{#1}{\lVert}{\rVert}{#2}}
	\newcommand{\Seminorm}[2][0]{\DelimiterGroup{#1}{\lvert}{\rvert}{#2}}

	\newcommand{\Set}[2][0]{\DelimiterGroup{#1}{\{}{\}}{#2}}
	
	\newcommand{\Tuple}[2][0]{\DelimiterGroup{#1}{(}{)}{#2}}

	\newcommand{\Ceil}[2][0]{\DelimiterGroup{#1}{\lceil}{\rceil}{#2}}

	\newcommand{\IProd}[3][0]{\DelimiterGroup{#1}{\langle}{\rangle}{#2,#3}}

	\newcommand{\IntOO}[3][0]{\DelimiterGroup{#1}{(}{)}{#2,#3}}
	\newcommand{\IntCC}[3][0]{\DelimiterGroup{#1}{[}{]}{#2,#3}}
	\newcommand{\IntCO}[3][0]{\DelimiterGroup{#1}{[}{)}{#2,#3}}
	\newcommand{\IntOC}[3][0]{\DelimiterGroup{#1}{(}{]}{#2,#3}}

	\newcommand{\IU}{\mathsf{i}}

	\newcommand{\SetMinus}{\mathbin{\mathpalette\rsetminusaux\relax}}
	\newcommand{\rsetminusaux}[2]{\mspace{-4mu}
	    \raisebox{\rsmraise{#1}\depth}{\rotatebox[origin=c]{-20}{$#1\smallsetminus$}}
	  \mspace{-4mu}
	}
	\newcommand{\rsmraise}[1]{%
	  \ifx#1\displaystyle .8\else
	    \ifx#1\textstyle .8\else
	      \ifx#1\scriptstyle .6\else
		.45%
	      \fi
	    \fi
	  \fi}

	\newcommand{\QQ}{{\:}\!}

	\DeclareMathOperator*{\Span}{span}
	\DeclareMathOperator*{\Dim}{dim}

	\renewcommand{\Vec}[1]{
		\boldsymbol{#1}
	}
	\newcommand{\Ten}[1]{
		\boldsymbol{#1}
	}

	\title{Quantized Tensor FEM for Multiscale Problems:
	diffusion problems in two and three dimensions.
\thanks{
Work initiated while VK and IO visited the
Institute for Mathematical Research (FIM) of
ETH Z\"urich in 2017 (IO) and 2018 (VK),
and
performed in part while the authors participated
in the thematic programme
``Numerical Analysis of Complex PDE Models in the Sciences''
of the
Erwin Schr\"odinger International Institute for Mathematics and Physics (ESI),
Vienna, Austria, in June--August 2018.
}
}
\author{
V. Kazeev
\thanks{
	Faculty of Mathematics,
	University of Vienna,
	Oskar-Morgenstern-Platz 1,
	1090 Vienna, Austria.
	\texttt{vladimir.kazeev@univie.ac.at}.
}
\and
I. Oseledets
\thanks{
  Skolkovo Institute of Science and Technology, Novaya Ul.~100, 143025 Skolkovo, Moscow Region, Russia
  and
  Institute of Numerical Mathematics, Russian Academy of Sciences, Gubkina St.~8, 119333, Moscow, Russia.
	\texttt{i.oseledets@skoltech.ru}
}
\and
M. Rakhuba
\thanks{
	Seminar for Applied Mathematics, ETH Zurich, R\"{a}mistrasse~101, 8092 Zurich, Switzerland.
	\texttt{\{maksim.rakhuba,christoph.schwab\}@sam.math.ethz.ch}
}
\and
Ch. Schwab
\footnotemark[4]
}
\begin{document}
\maketitle

\vskip 1cm
\begin{abstract}
	Homogenization in terms of multiscale limits transforms
	a multiscale problem with $n+1$ asymptotically separated microscales
	posed on a physical domain $D \subset \R^d$
	into a one-scale problem posed on a product domain of
	dimension $(n+1)d$ by introducing $n$
	so-called ``fast variables''.
        This procedure allows to convert $n+1$ scales
	in $d$ physical dimensions into a
        single-scale structure in $(n+1)d$ dimensions.
	We prove here that both
        the original, physical multiscale problem and the corresponding
        high-dimensional, one-scale limiting problem
	can be efficiently treated numerically with the recently developed
	\emph{quantized tensor-train finite-element method} (QTT-FEM).

The method is based on restricting computation to
sequences of nested subspaces of low dimensions
(which are called tensor ranks)
within a vast but generic ``virtual'' (background) discretization space.
In the course of computation, these subspaces are computed iteratively
and data-adaptively at runtime, bypassing any ``offline precomputation''.
For the purpose of theoretical analysis,
such low-dimensional subspaces are constructed analytically
so as to bound the tensor ranks vs. error tolerance $\tau>0$.

	We consider a model linear elliptic multiscale problem
        in several physical dimensions
	and show, theoretically and experimentally, that both
        (i) the solution of the associated high-dimensional one-scale problem
	and
        (ii) the corresponding approximation to the solution of the multiscale problem
	admit efficient approximation by the QTT-FEM.
	These problems can therefore be numerically solved in a scale-robust fashion 
	by standard (low-order) PDE discretizations combined with
	state-of-the-art general-purpose solvers for tensor-structured linear systems.
	We prove scale-robust exponential convergence, i.e.,
        that QTT-FEM achieves accuracy $\tau$ with 
        the number of effective degrees of freedom scaling
	polynomially in $\log \tau$. 
\end{abstract}
\section{Introduction}
\label{sc:Intro}

The efficient numerical solution of mathematical models of
\emph{physical processes with multiple scales}
has undergone a rapid development during recent years.
Several classes of computational approaches have been put forward
which try, usually through selective and sparing access of the
microscopic structure of the problem, to correctly numerically
approximate the ``effective'', macroscopic or ``homogenized''
features of the solution.
In the context of Finite Element discretizations,
these methodologies are referred to as \emph{multiscale FEM}
(MsFEM).
In a broader context, such computational approaches for the numerical approximation
of multiscale differential equation models (ordinary or partial)
have been referred to as \emph{hierarchic multiscale methods} (HMM).
We refer to \cite{HouMsFEM,AbdulleActa} and the references therein for a comprehensive discussion.

In these approaches, the solution of the correct macroscopic, or ``upscaled''
mathematical model is numerically approximated by \emph{selective, localized
access to the microscopic information}.
This can be achieved by the mentioned methods
in (essentially optimal) \emph{numerical complexity that is independent of the
microscopic length scale of the problem}.
Additionally,
\emph{postprocessing techniques} allow for \emph{localized numerical recovery of the microscopic structure
of the physical solution},
at extra computational costs.

An alternative computational approach aims at the simultaneous numerical approximation
of the macroscopic, homogenized solution and at the numerical approximation of the
microscopic structure of the physical solution, \emph{throughout the physical domain},
at computational work which is independent of the physical length scale of data.
This is feasible, in general, under additional assumptions on the microstructure, such
as (locally) periodicity or ergodicity.
Under such assumptions, it is known that for linear, second order elliptic PDEs
the physical solution and the interaction of all scales can be described by certain
two- and $(n+1)$-scale limits \cite{Nguetseng:1989:Homog,GAHom2Scale1992,Allaire:1996:ReiteratedHomog}.
These limits take the form of solutions of \emph{high-dimensional, elliptic boundary value
problems}, which are independent of the scale parameters and posed on a Cartesian product
of the physical domain $D$ and of the $n$ ``unit-cells'' $Y_i$, $i=1,...,n$
that describe the structure of the fast scales of the multiscale solution.
As a result,
\emph{$(n+1)$-scale limits trade scale-resolving requirements
for high-dimensionality~\cite{S02_441}}.
Loosely speaking,
scale-resolution is traded for the ``curse of dimensionality'':
once efficient numerical approximations for such high-dimensional $(n+1)$-scale limiting problems
are available, \emph{robust, scale-independent discretizations of
multiscale models can be derived}.
This idea, put forward in \cite{S02_441}, has
been developed in the context of \emph{sparse tensor FEM} multiscale diffusion problems
in \cite{VHHCS2004} and,
subsequently, for elasticity and
 electromagnetics~\cite{VHH2008Monot,BXVHH2014,BXVHH2015b,BXVHH2015}.
In particular, \emph{algebraic convergence rates independent
of the scale parameter with weak or no dependence on the number of $n$ of fast variables}
were established.
The implementation of these sparse tensor FEM discretizations of the high-dimensional
limits requires, however,
\emph{explicit derivation of the PDEs which describe the $(n+1)$-scale limits}.
This may, in particular for nonlinear multiscale problems, not be feasible,
even though the existence of $(n+1)$-scale limits is mathematically assured.

\subsection{Contributions}
\label{sc:Contr}

We analyze the novel, tensor-structured
numerical approximation of the solution of a linear second-order elliptic
PDE whose diffusion tensor depends on $n+1$ separated scales,
i.e., in the classical setting of $(n+1)$-scale homogenization.
Specifically, following earlier work
\cite{Khoromskij:2015:Homog,Khoromskij:2017:Homog,VKIOMRCS2017}
we consider the \emph{quantized tensor-train finite-element method} (QTT-FEM),
combining adaptive low-rank tensor approximation with
\emph{quantization}~\cite{Oseledets:2010:QTT,Khoromskij:2011:QuanticsApprox}
to efficiently represent the multiscale structure of data.

In the present paper, we first prove that the QTT-FEM allows for
\emph{exponentially convergent} numerical approximations
to the scale-interaction
functions involved in the $(n+1)$-scale limits and, as a consequence,
to the homogenized solutions. Specifically, we construct ``by hand'' certain
approximations that, with respect to the discretization parameter,
are sufficiently accurate and have sufficiently low tensor ranks.

The idea of approximating the multiscale problem by reapproximating the homogenized problem
(derived by $(n+1)$-scale convergence~\cite{Allaire:1996:ReiteratedHomog,DCPDRZ2006}),
proposed for elliptic multiscale problems in \cite{S02_441},
was exploited in
the context of 
\emph{sparse grid} approximations~\cite{VHHCS2004,VHH2008Monot,HS11_2105}. 
However, our present perspective extends further, as the motivation for
considering approximations based on homogenization.
In practice, the QTT-FEM can completely bypass the homogenization procedure and
operate entirely on the physical domain, adaptively accessing the fine-scale information of the PDE.
Naturally, the numerical approximations found by this approach
are better adapted to the data and
are more efficient than any particular approximations we construct ``by hand''
through the re-approximation of the corresponding homogenized problem.
In Section~\ref{sc:NumReS}, we report numerical results obtained by such a
practical computational multiscale
QTT-FEM algorithm, built upon the TT~Toolbox~\cite{TT-Toolbox}.

\subsection{Structure of the present paper}
\label{sc:Struct}
In Section \ref{sec2}, we describe the $n$-scale homogenization problem,
and
present in particular the QTT discretization of this problem in the physical domain
in Section \ref{subsectensdiscr}.
The emphasis in Section \ref{sec2}
is to present the $n$-scale problem and its quantized, tensor-formatted
discretization entirely in the physical domain.
Section \ref{sec:UpSc} presents the asymptotic analysis of the
$n$-scale solution by the so-called unfolding method:
the asymptotic limit of the physical problem
is described by a high-dimensional one-scale problem.
To this end,
we recapitulate results from
\cite{Nguetseng:1989:Homog,GAHom2Scale1992,Allaire:1996:ReiteratedHomog}
on reiterated homogenization for linear, elliptic multiscale problems,
which are required in the ensuing numerical analysis of the QTT-FE
approach.

Section \ref{sec:Approx} will develop novel approximation rate results
for the solution of the $(n+1)$-scale limit which are, subsequently,
used to obtain \emph{quantized tensor-rank bounds} for the physical,
$(n+1)$-scale solution.

Section \ref{sc:NumReS} then will present numerical experiments
which model multiscale problems where the QTT-ranks of the numerical
solutions are explicitly estimated numerically.

Finally, Section \ref{sc:ConclRmk} and the Appendix
contain a discussion of the results and a few proofs postponed due to their technicality.

\section{Model elliptic multiscale problem}\label{sec2}

We consider a bounded ``physical'' domain $D \subset \R^d$
(with which, for notational convenience, we associate the macroscale $\varepsilon_0 = 1$)
and a moderate number $n\in\N$ of microscales
$\varepsilon_1,\ldots,\varepsilon_n$,
which we assume to be positive functions
of a scale parameter $\varepsilon$ such that
$\lim_{\varepsilon \to 0} \varepsilon_i = 0$
for all $i\in\Set{1,\ldots,n}$.
We additionally assume
\emph{asymptotic scale separation}:
\begin{equation}\label{eq:scale-sep}
	\lim_{\varepsilon \to 0} \varepsilon_{i+1}/\varepsilon_i=0
\end{equation}
for $i\in\Set{1,\ldots,n-1}$.

Further, we assume that there exist $n$ \emph{unit cells}
$Y_1,\ldots,Y_n$
such that
$D$ is partitioned into a union of translations of $\varepsilon_1 Y_1$
and each $Y_{i-1}$ with $i\in\Set{2,\ldots,n}$
is partitioned into a union of translations of $\varepsilon_i Y_i$.
Specifically, we deal with the case of
$Y_1,\ldots,Y_n = \IntOO{0}{1}^d$ in the present paper,
while more sophisticated constructions may be used to model,
e.g., perforated media.
For notational convenience, we set
$\bm{Y}_0 = \Set{0}$ and
$\bm{Y}_i = Y_1 \times\cdots\times Y_i$ for each $i\in\Set{1,\ldots,n}$.

To formulate a multiscale diffusion problem on $D$,
we consider a matrix function $A$
defined on $D \times \bm{Y}_n$,
which therefore depends on a macroscale (``slow'') variable
and on $n$ microscale (``fast'') variables.
We will consider multiscale diffusion coefficients $A^\varepsilon$ induced by
functions satisfying the following assumption.

\begin{assumption}\label{As:Coeff}
	$A \in L^{\infty}\Par{D  \QQ ; \,   C_{\#} \Par{\bm{Y}_n \QQ ; \, \R^{d \times d}_{\text{\rm sym}} \, }}$
	is essentially bounded and
	uniformly positive definite with constants $\Gamma$ and $\gamma$:
	$
		\gamma
		\leq
		\xi^\MT
		A(x,\bm{y}_n) \,
		\xi
		\leq \Gamma
	$
	for every unit vector
	$\xi\in\R^d$,
	a.e. $x\in D$ and all
	$\bm{y}_n\in \bm{Y}_n$.
\end{assumption}

Here and throughout, by $C_{\#}(\bm{Y}_n)$ we denote
the space of functions that are continuous
on $\overline{\bm{Y}_n}$
and $Y_i$-periodic with respect to the $i$th variable for
each $i\in\Set{1,\ldots,n}$.

For every $\varepsilon>0$, a function $A$ satisfying Assumption~\ref{As:Coeff} induces
a multiscale coefficient
$A^\varepsilon \in L^{\infty}\Par{D}$ as follows:
\begin{equation}\label{eq:Aeps}
A^\varepsilon\Par{x}
=
A \Par[2]{ x,\frac{x}{\varepsilon_1},\ldots,\frac{x}{\varepsilon_n} }
\quad\text{for all}\quad
x \in D
\, .
\end{equation}
With such a coefficient, we consider the following model variational problem on $V = H^1_0 \Par{D}$:

\begin{equation}\label{Eq:ProblemEps}
	\text{find $u^\varepsilon\in V$ such that}
	\quad
	\int_D \Par{\nabla v}^\MT A^\varepsilon \, \nabla u^\varepsilon
	=
	\int_D f v
	\quad\text{for all}\quad
	v\in V
	\, ,
\end{equation}
where $f\in L^2(D)$ is a forcing term.
Assumption~\ref{As:Coeff} and the Lax--Milgram theorem guarantee that this problem has a unique solution,
which satisfies the stability bound
\[
	\Seminorm{u^\varepsilon}_{H^1\Par{D}}
	\leq
	\gamma^{-1}
	\sup_{v \in V \SetMinus\Set{0}}
	\frac{\Abs{f(v)}}{\Seminorm{v}_{H^1\Par{D}}}
	\leq
	C \, \gamma^{-1}
	\sup_{v \in V \SetMinus\Set{0}}
	\frac{\Abs{f(v)}}{\Norm{v}_{L^2(D)}}
	=
	C \, \gamma^{-1} \Norm{f}_{L^2\Par{D}}
	\, ,
\]
where $C$ is the classical Poincar\'{e} constant for $D$.

Although the forcing term $f$ is assumed to be independent of the
scale parameter $\varepsilon$ for simplicity,
we hasten to add that all results
that follow admit a straightforward generalization
to the case when $f$ exhibits a microscale structure analogous
to the one expressed by~\eqref{eq:Aeps}.
\subsection{Low-rank tensor multilevel discretization}
\label{subsectensdiscr}
In this section,
we give an explicit construction of
the low-rank tensor multilevel discretization of the problem~\eqref{Eq:ProblemEps}
for the case when $D=Y_1=\cdots=Y_n=\IntOO{0}{1}^d$.
We start with defining the underlying
\emph{virtual grid} and the associated finite-element spaces.
\subsubsection{Virtual grids and low order finite-element spaces}
\label{sc:fesp-intro}
Let $L\in\Nz$ be arbitrary.
We define index sets $\mathcal{I}^L = \Set{1,\ldots,2^L}$,
and
$\mathcal{J}^L = \Set{1,\ldots,2^L -1 }$,
select the meshwidth at mesh level $L$ as $h_L = 2^{-L}$
and
consider a set of equispaced points on $\IntOO{0}{1}$:
\begin{equation}\label{Eq:Part1dNodes-intro}
	t^L_j = j h_L
	\quad\text{with}\quad j\in\Set{0}\cup\mathcal{I}^L
	\, .
\end{equation}
The corresponding continuous piecewise-linear functions
$\varphi^L_j$, $j\in\Set{0}\cup\mathcal{I}^L$ are given by the condition
$\varphi^L_j\Par{t^L_{j'}}=\delta_{j j'}$
for all $j,j'\in\Set{0}\cup\mathcal{I}^L$.
These functions form a basis in the finite-element space
$
\widetilde{U}^L = \Span \Set{\varphi^L_j \colon j\in\Set{0}\cup\mathcal{I}^L }
$,
whose subspace
$
	U^L
	=
	\Span \Set{\varphi^L_j \colon j\in\mathcal{J}^L }
$
allows to explicitly impose the boundary conditions of the
problem~\eqref{Eq:ProblemEps}.
Similarly,
the corresponding space of
piecewise-constant functions is
$
	\bar{U}^L
	=
	\Span
	\Set{\bar\varphi_i^L \colon i\in\mathcal{I}^L}
$
with
$\bar\varphi_i^L$ with $i\in\mathcal{I}^L$ given by
the condition $\bar\varphi_i^L\Restr{\IntOO{t^L_{i'-1}}{t^L_{i'}}}=\delta_{i i'}$
for all $i,i'\in\mathcal{I}^L$.

To obtain coefficients of finite-element approximations
with respect to these bases,
we will use the \emph{analysis operators}
$
	\varPhi^L
	\colon
	H^{1}\IntOO{0}{1}
	\rightarrow
	\C^{\mathcal{I}^L}
	\simeq
	\C^{2^L}
$
and
$
	\bar\varPhi^L
	\colon
	L^2\IntOO{0}{1}
	\rightarrow
	\C^{\mathcal{I}^L}
	\simeq
	\C^{2^L}
$ 
defined as follows:
for all
$v\in H^{1}\IntOO{0}{1}$, $w\in L^2\IntOO{0}{1}$
and $i\in\mathcal{I}^L$, we set
\begin{equation}\label{Eq:DefOpEF1d-intro}
	\Par{\varPhi^L v}_i
	=
	v \Par{t_i^L}
	\quad\text{and}\quad
	\Par{\bar\varPhi^L w}_i
	=
	2^L
	\int_{ t^L_{i-1} }^{t_i^L} w
	\, .
\end{equation}
Tensorizing the univariate basis functions defined above,
we obtain $d$-variate basis functions
that span the corresponding finite-element spaces:
\begin{equation}\label{eq:fe-spaces-intro}
	V^L
	=
	\bigotimes_{k=1}^{d}
	U^L
	\subset
	V
	\quad\text{and}\quad
	\bar{V}^L
	=
	\bigotimes_{k=1}^{d}
	\bar{U}^L
	\subset
	L^2(D)
	\, .
\end{equation}
Classical approximation bounds (see, e.g.,~\cite{Ciarlet:FEM}) give
\begin{equation}
	\label{eq:VL-approx-intro}
	\inf_{v^L \in V^L} \Norm{v - v^L}_{H^1\Par{D}} \leq C \, 2^{-tL} \, \Norm{v}_{H^{1+t}\Par{D}}
	\quad\text{for all}\quad
	v \in H^{1+t}\Par{D}
	\, ,
\end{equation}
where
$t>0$ is a fractional order of Sobolev smoothness
and $C>0$ is a coefficient that depends on $t$ but not on $L$.

Since the solution $u^\varepsilon$ of~\eqref{Eq:ProblemEps} 
may exhibit algebraic singularities at the boundary of $D$
due to a combination of the domain's geometry, boundary conditions and
diffusion coefficient,
$u^\varepsilon \in H^{1+t}\Par{D}$ may hold only for $t$ significantly less than one.
To efficiently approximate such solutions in low-rank form,
we will follow~\cite{Kazeev:PhD,KS:2017:QTTFE2d,MRS19_872}
in using the multilevel QTT format
for the low-rank separation of the indices associated with different levels and,
for example, not different physical variables.
This consists in applying the isomorphism
\begin{equation}\label{Eq:IndTransp}
	\bigotimes_{k=1}^d
	\bigotimes_{\ell=1}^L
	\C^{2}
	\simeq
	\bigotimes_{\ell=1}^L
	\CuBr[3]{
		\bigotimes_{k=1}^d
	\C^{2}
	}
\end{equation}
so that the $2^{dL}$ degrees of freedom in $V^L$ in \eqref{eq:fe-spaces-intro}
are represented by $d$-indices
corresponding to the $L$ levels of discretization,
each taking $2^d$ values that enumerate the elements of the corresponding factor
marked by curly brackets in~\eqref{Eq:IndTransp}.
To refer to this isomorphism explicitly, we define
$\Ten{\varPi}^L$ with $L\in\N$
as
the permutation matrix of order $2^{d L}$
satisfying
\begin{equation}\label{Eq:IndTranspMatrix}
	\Par{\Ten{\varPi}^L}_{
	\,
		i_{1,1}\,,\ldots,\,i_{d,1}
		, \ldots\ldots ,\,
		i_{1,L}\,,\ldots,\,i_{d,L}
	\;\;\;
		i_{1,1}\,,\ldots,\,i_{1,L}
		, \ldots\ldots ,\,
		i_{d,1}\,,\ldots,\,i_{d,L}
	}
	=
	1
\end{equation}
for all $i_{k\ell}\in\Set{1,2}$ with $k\in\Set{1,\ldots,d}$ and $\ell\in\Set{1,\ldots,L}$.

The elements of $V^L$ and $\bar{V}^L$
can be parametrized
by their coefficients extracted using
the analysis operators
\begin{equation}
	\label{Eq:DefAnOp-intro}
	\varPsi^L
	=
	\Ten{\varPi}^L
	\bigotimes_{k=1}^d
	\varPhi^L
	\colon
	\bigotimes_{k=1}^d
	H^1\IntOO{0}{1}
	\to
	\C^{2^{dL}}
	\quad\text{and}\quad
	\bar\varPsi^L
	=
	\Ten{\varPi}^L
	\bigotimes_{k=1}^d
	\bar\varPhi^L
	\colon
	L^2(D)
	\to
	\C^{2^{dL}}
	\, .
\end{equation}
Note that the restriction of
$\varPsi^L$
to $V^L$
is not surjective.
This lack of surjectivity stems from that we
choose to use nested finite-element spaces
$V^L$ with $L\in\N$ given by~\eqref{eq:fe-spaces-intro}
but represent every function from $V^L$ with $L\in\N$
by $2^{dL}$ values instead of $(2^L-1)^d$, the extra values,
in agreement with the boundary conditions of the problem~\eqref{Eq:ProblemEps}, being zero.

\subsubsection{Discrete multiscale problem and low-rank tensor parametrization}

For every $L\in\N$, we consider the following
discretization of the problem~\eqref{Eq:ProblemEps}:
\begin{equation}\label{Eq:ProblemEpsDiscr}
	\text{find $u^{\varepsilon,L}\in V^L$ such that}
	\quad
	\int_D \Par{\nabla v^L}^\MT A^\varepsilon \, \nabla u^{\varepsilon,L}
	=
	\int_D f v^L
	\quad\text{for all}\quad
	v^L\in V^L
	\, .
\end{equation}
As for the original problem,
Assumption~\ref{As:Coeff} and the Lax--Milgram theorem guarantee that the above discretization has a unique solution.
By the C\'{e}a's lemma, the discrete solution is quasi optimal:
$
	\Norm{u^{\varepsilon}-u^{\varepsilon,L}}_{H^1\Par{D}}
	\leq
	C \, \gamma^{-1} \Gamma \, 2^{-tL} \Norm{u^{\varepsilon}}_{H^{1+t}\Par{D}}
$,
where $C$ is the constant appearing in the approximation bound~\eqref{eq:VL-approx-intro}.

For a tensor $\Vec{u} \in \C^{n_1 \times \cdots \times n_L}$ with $L\in\N$ dimensions and
mode sizes $n_1,\ldots,n_L \in \N$, a representation
\begin{multline}\label{eq:TT-MPS}
	\Vec{u}_{i_1,\ldots,i_L}
	=
	\sum_{\alpha_1}^{r_1} \cdots \sum_{\alpha_{L-1}}^{r_{L-1}}
	U_1(1,i_1,\alpha_1) \,\cdot\, U_2(\alpha_1,i_2,\alpha_2)
	\\
	\cdots
	U_{L-1}(\alpha_{L-2},i_{L-1},\alpha_{L-1}) \,\cdot\, U_L(\alpha_{L-1},i_L,1)
\end{multline}
in terms of arrays
$U_\ell \in \C^{r_{\ell-1} \times n_\ell \times r_\ell}$ with $\ell\in\Set{1,\ldots,L}$,
where we use $r_0 = 1 = r_L$ for convenience, is referred to as
a \emph{tensor train} (TT) decomposition~\cite{Oseledets:2009:TT,Oseledets:2011:TT}
or, alternatively, as a \emph{matrix-product state} (MPS)
representation~\cite{White:1992:Lett69,Vidal:2003,Schollwoeck:2011:DMRG-MPS}.
The arrays $U_k$ with $\ell\in\Set{1,\ldots,L}$ are called
\emph{cores}, and the parameters $r_1,\ldots,r_{L-1}$, governing the number of
entries of the cores, are called \emph{ranks}. In the present paper, we use the
TT-MPS representation as a \emph{multilevel} tensor decomposition~\cite{Tyrtyshnikov:2003:VirtualLevels},
by which we mean that the indices of
a tensor represented as in~\eqref{eq:TT-MPS} represent the scales (not the physical dimensions) of the data.
In the context of the TT-MPS decomposition, this has been known in the literature as the
\emph{quantized tensor train decomposition}~\cite{Oseledets:2010:QTT,Khoromskij:2009:55,Khoromskij:2011:QuanticsApprox,MR3822367}.

\section{Reiterated homogenization and high-dimensional one-scale limit}
\label{sec:UpSc}

For analysis, instead of the original multiscale problem~\eqref{Eq:ProblemEps},
we consider a one-scale high-dimensional limit problem posed in~\eqref{limeqn} in this section.
The limit problem is obtained from the original multiscale problem~\eqref{Eq:ProblemEps}
by homogenization, analyzed for $n=1$, i.e., for a single microscale
in~\cite{BLP1978,Bakhvalov:Homog,Jikov:Homog,Murat:1997:H-convergence,Nguetseng:1989:Homog,Allaire:1992:TwoScaleConv},
and for $n>1$ fast scales by iteration in~\cite{Allaire:1996:ReiteratedHomog}.
For a general discussion, we refer to~\cite{Engquist:2008:AsymptoticNumerical}.
\subsection{One-scale high-dimensional limit problem}
\label{sec:1ScLimPrb}
To formulate reiterated homogenization, we consider
the following assumption, of which Assumption~\ref{As:Coeff}
is a particular case with $i=n$ and $A_n=A$.
\begin{assumption}[on a coefficient $A_i$ with $i\in\Set{0,\ldots,n}$
  microscales, with positive constants $\gamma$ and $\Gamma$]\label{As:Coeff-i}
	$A_i \in L^{\infty}\Par{D  \QQ ; \,   C_{\#} \Par{\bm{Y}_i \QQ ; \, \R^{d \times d}_{\text{\rm sym}} \, }}$
	is essentially bounded and
	uniformly positive definite with constants $\Gamma$ and $\gamma$:
	$
		\gamma
		\leq
		\xi^\MT
		\!
		A_i(x,\bm{y}_i) \,
		\xi
		\leq \Gamma
	$
	for every unit vector $\xi\in\R^d$, a.e. $x\in D$ and all $\bm{y}_i\in \bm{Y}_i$. 
\end{assumption}

For each step $i\in\Set{1,\ldots,n}$ of homogenization, we define
\begin{equation}\label{spacei}
	\begin{aligned}
		V_i
		&=
		L^2\Par[1]{
			D \times \bm{Y}_{i-1},
			H^1_{\#} \Par{Y_i}/_\R
		}
		\simeq
		L^2\Par{D}
		\otimes
		L^2\Par{Y_1} \otimes\cdots\otimes L^2\Par{Y_{i-1}}
		\otimes
		H^1_{\#} \Par{Y_i}/_\R
		\, ,
		\\
		W_i
		&=
		L^\infty\Par{D \times \bm{Y}_{i-1}, H^1_{\#}\Par{Y_i}/_\R }
		\, ,
	\end{aligned}
\end{equation}
and consider the Cartesian-product space
\begin{equation}\label{space}
	\bm{V}_i
	=
	V \times V_1 \times \cdots \times V_i
\end{equation}
endowed with the inner product
$\IProd{ \cdot }{ \cdot }_{ \bm{V}_i }$
given by
\begin{equation}\label{norm}
	\IProd{\boldsymbol{\psi}}{\boldsymbol{\phi}}_{ \bm{V}_i }
	=
	\sum_{\IndNorm{\alpha}=1}
	\IProd{ \partial^\alpha \psi_0 }{ \partial^\alpha \phi_0 }_{L^2(D)}
	+
	\sum_{j=1}^i
	\sum_{|\alpha_j|=1}
	\IProd{ \partial_j^{\alpha_j} \psi_j }{ \partial_j^{\alpha_j} \phi_j }_{ L^2\Par{ D\times \bm{Y}_{ j} } }
\end{equation}
for all
$
	\boldsymbol{\psi} = \Tuple{\psi_0,\psi_1,\ldots,\psi_i},
	\boldsymbol{\phi} = \Tuple{\phi_0,\phi_1,\ldots,\phi_i}
	\in
	\bm{V}_i
$.
We denote the norm
induced by $\IProd{\cdot}{\cdot}_{\bm{V}_i}$
with $\Norm{\cdot}_{ \bm{V}_i }$.
Here and throughout,
the symbol $\partial^\alpha$ with $\alpha\in\N_0^d$
denotes the differentiation of functions
with respect to the first $d$ scalar variables
indicated by the multi-index $\alpha$,
whereas
$\partial_i^{\alpha}$ with $i\in\Set{1,\ldots,n}$ and $\alpha\in\N_0^d$
denotes the differentiation of functions
with respect to the scalar variables $id+1,\ldots,\Par{i+1}d$
according to the multi-index $\alpha$.
Further, we define a bilinear form
$\mathsf{B}_i \!: \, \bm{V}_i \times \bm{V}_i \rightarrow \R$:
\begin{equation}\label{Eq:LimBLF}
	\mathsf{B}_i(\boldsymbol{\psi}, \boldsymbol{\phi})
	=
	\int_{D \times \bm{Y}_i}
	\;
	\Par[2]{
		\nabla \psi_0
		+
		\sum_{j=1}^i \nabla_{\! j} \psi_j
	}^\MT
	A_i
	\,
	\Par[2]{
		\nabla \phi_0
		+
		\sum_{j=1}^i \nabla_{\! j} \phi_j
	}
\end{equation}
for all
$
	\boldsymbol{\psi} = \Tuple{\psi_0,\psi_1,\ldots,\psi_i},
	\boldsymbol{\phi} = \Tuple{\phi_0,\phi_1,\ldots,\phi_i}
	\in
	\bm{V}_i
$,
where $A_i$ is a matrix function satisfying
Assumption~\ref{As:Coeff-i} with $i$ microscales and with positive constants
$\gamma$ and $\Gamma$.
Then the bilinear form $\mathsf{B}_i$ is continuous and coercive:
the inequalities
\begin{eqnarray}\label{BcoerBcont}
	\gamma
	\,
	\Norm{\boldsymbol{\phi}}_{ \bm{V}_i }^2
	\leq
	\mathsf{B}_i(\boldsymbol{\phi},\boldsymbol{\phi})
	\quad\text{and}\quad
	\mathsf{B}_i(\boldsymbol{\psi},\boldsymbol{\phi})
	\leq
	\Gamma
	\,
	\Norm{\boldsymbol{\psi}}_{ \bm{V}_i }
	\Norm{\boldsymbol{\phi}}_{ \bm{V}_i }
\end{eqnarray}
hold for all $\boldsymbol{\psi},\boldsymbol{\phi}\in\bm{V}_i$.
Then, since $f\in L^2\Par{D}$,
the problem of finding $\boldsymbol{u} \in \bm{V}_i$ such that
	\begin{equation}\label{limeqn}
			\mathsf{B}_i(\boldsymbol{u}, \boldsymbol{\phi})
			=
			\int_D f \phi_0
			\quad\text{for all}\quad
			\boldsymbol{\phi} = \Tuple{\phi_0,\phi_1,\ldots,\phi_i} \in \bm{V}_i
	\end{equation}
has a unique solution $\boldsymbol{u} = \Tuple{u_0,u_1,\ldots,u_i}$ (by the Lax--Milgram theorem).
For notational convenience, we introduce
\begin{equation}\label{eq:def-vi}
	v_i = \sum_{j=0}^i \nabla u_j
	\quad\text{with}\quad
	i=1,\ldots,n
	\, .
\end{equation}
We remark that the bilinear forms $\mathsf{B}_i$, $i=1,\ldots,n$,
in \eqref{limeqn} satisfy property~\eqref{BcoerBcont}
with constants uniform with respect to the scale parameter $\varepsilon$.

The problem~\eqref{limeqn} with $i=n$ microscales,
representing the result of $n$ iterations of homogenization
applied to the original multiscale problem~\eqref{Eq:ProblemEps},
approximates the multiscale problem in the following sense.
\begin{theorem}[Theorem~2.11 and equation~(2.9) in~\cite{Allaire:1996:ReiteratedHomog}]
\label{thm:Highdlimit}
	The solution $u^\varepsilon$ of the problem~\eqref{Eq:ProblemEps}
	converges weakly to $u_0$ in $H^1_0(D)$, and
	$\nabla u^\varepsilon$ $(n+1)$-scale converges to $v_n$.
\end{theorem}
Using the following result, the physical solution $u^\varepsilon$,
including the oscillations induced by the multiscale
structure of the diffusion coefficient~\eqref{eq:Aeps},
can be approximated in terms of the solution of the one-scale
high-dimensional limit problem.
\begin{theorem}[Theorem~2.14 in \cite{Allaire:1996:ReiteratedHomog}]\label{multicorrector}
Assume that the solution $(u,u_1,\ldots,u_n)$ of problem (\ref{limeqn})
is sufficiently smooth, say $u\in C^1(\overline{D})$ and
$u_i\in C^1(\overline{D},C^1_\#(\bm{Y}_i))$
for all $i\in\Set{1,\ldots,n}$.
Then, as $\varepsilon \rightarrow 0$,
\[
u^\varepsilon(x)
\rightarrow
u_0(x) + \sum_{i=1}^n\varepsilon_i u_i\biggl(x,{x\over\varepsilon_1},\ldots,{x\over\varepsilon_i}\biggr)
\quad\text{in}\quad
H^1(D)\;.
\]
\end{theorem}

\subsection{Convergence in physical variables for multiple scales.
Unfolding and averaging operators}
\label{sc:CnvPhysVar}
For problems with $n+1 > 2$ scales,
an error estimate in the form~\eqref{corrector} appears not to be available.
We still base the rank bounds for the QTT discretization on the
structure of the one-scale limiting problem. To this end,
generalizing~\cite[Definitions~2.1 and 2.16]{Cioranescu:2008:PeriodicUnfolding}
to the case of multiple microscales,
we introduce unfolding and averaging operators.
\begin{definition}\label{def:unfolding-folding}
	For all $i\in\Set{1,\ldots,n}$,
	the operators
	$
		\mathcal{T}_i^\varepsilon
		\colon
		L^2(D \times Y_{i+1} \times\cdots\times Y_n)
		\to
		L^2(D\times Y_i \times\cdots\times Y_n)
	$
	and
	$
		\mathcal{U}_i^\varepsilon
		\colon
		L^2(D\times Y_i \times\cdots\times Y_n)
		\to
		L^2(D \times Y_{i+1} \times\cdots\times Y_n)
	$
	of unfolding and averaging with respect to the $i$th microscale
	are defined by
	\[
		\Par[1]{\mathcal{T}_i^{\varepsilon} \QQ \phi}
		(x,y_i,y_{i+1},\ldots,y_n)
		=
		\phi\Par[2]{
			\varepsilon_i \SqBr[2]{\frac{x}{\varepsilon_i}}
			+
			\varepsilon_i \QQ y_i,
			y_{i+1},\ldots,y_i
		}
	\]
	for a.e. $(x,y_i,\ldots,y_n)\in D \times Y_i \times\cdots\times Y_n$
	and all $\phi\in L^2(D \times Y_{i+1} \times\cdots\times Y_n)$,
	where $\phi$ is extended by zero outside its domain,
	and
	\[
		\Par{\mathcal{U}^\varepsilon_i \QQ \varPhi} \Par{x,y_{i+1},\ldots,y_n}
		=
		\Abs{Y_i}^{-1}
		\int_{Y_i}
		\varPhi
		\Par[2]{
			\varepsilon_i \SqBr[2]{ \frac{x}{\varepsilon_i} }
			+
			\varepsilon_i \QQ z
			,
			\CuBr[2]{ \frac{x}{\varepsilon_i} },
			y_{i+1},\ldots,y_n
		}
		\d z
	\]
	for a.e. $(x,y_{i+1},\ldots,y_n)\in D \times Y_{i+1} \times\cdots\times Y_n$
	and all $\varPhi\in L^2(D\times \bm{Y}_n)$.

	For every $i\in\Set{1,\ldots,n}$, the
	$n$-microscale unfolding and averaging operators are defined as
	$
		\mathcal{T}^\varepsilon
		=
		\mathcal{T}_1^\varepsilon
		\circ \cdots \circ
		\mathcal{T}_n^\varepsilon
		\colon
		L^2(D)
		\to
		L^2(D\times \bm{Y}_n)
	$
	and
	$
		\mathcal{U}^\varepsilon
		=
		\mathcal{U}_n^\varepsilon
		\circ \cdots \circ
		\mathcal{U}_1^\varepsilon
		\colon
		L^2(D\times \bm{Y}_n)
		\to
		L^2(D)
	$.

\end{definition}
In the case of one microscale,
certain basic properties of the unfolding and averaging operators are analyzed
in~\cite{Cioranescu:2008:PeriodicUnfolding}.
In particular, by~\cite[Proposition~2.17]{Cioranescu:2008:PeriodicUnfolding},
the operator
	$
		\mathcal{U}_i^\varepsilon
		\colon
		L^2(D\times Y_i \times\cdots\times Y_n)
		\to
		L^2(D \times Y_{i+1} \times\cdots\times Y_n)
	$
is continuous and has norm $\Abs{Y_i}^{-1/2}$ for all $i\in\Set{1,\ldots,n}$.
This implies
\begin{equation}\label{eq:avg-op-err}
	\Norm{ \mathcal{U}^\varepsilon \Par{\varPhi -\widetilde{\varPhi} \QQ } }_{L^2(D)}
	\leq
	\Norm{ \varPhi -\widetilde{\varPhi} \QQ }_{L^2(D \times \bm{Y}_n)}
\end{equation}
for all $\varPhi,\widetilde{\varPhi}\in L^2(D \times \bm{Y}_n)$.

As in~\cite{Cioranescu:2008:PeriodicUnfolding}, one can show that
the solution $u^\varepsilon$ of the multiscale problem~\eqref{Eq:ProblemEps}
under the scale-separation condition~\eqref{eq:scale-sep} satisfies
\begin{equation}\label{eq:unfolded-grad-conv}
	\mathcal{T}^\varepsilon
	\nabla u^\varepsilon
	\to
	v_n
	\quad
	\text{strongly in }
	L^2(D\times \bm{Y}_{ n})
	\quad\text{as}\quad
	\varepsilon \to 0
	\, .
\end{equation}

Using the folding operator $\mathcal{U}^\varepsilon$, we can state
an analog of~\eqref{corrector} for several microscales, showing that
the scale-interaction functions $u_1,...,u_n$ in~\eqref{eq:unfolded-grad-conv} describe
to leading order
the
oscillations of the functions $u^\varepsilon$ with $\varepsilon > 0$
as they approach the weak limit
$u^0$.
\begin{lemma}\label{nscalecorrector}
	Under the scale-separation condition~\eqref{eq:scale-sep},
	for the multiscale problem~\eqref{Eq:ProblemEps}
	we have
	$
		\nabla u^\varepsilon
		-
		\mathcal{U}^\varepsilon
		v_n
		\to
		0
	$ 
	strongly in $L^2(D)$ as $\varepsilon \to 0$,
	the averaging operators being applied componentwise.
\end{lemma}

For a proof,
we refer to \cite[Theorem~6.1]{Cioranescu:2008:PeriodicUnfolding} for the case $n=1$ of a single microscale
and \cite[Remark~7.5]{Cioranescu:2008:PeriodicUnfolding} regarding the case of $n>1$ microscales.

\begin{remark}
When the unfolded solution $(u_0,u_1,\ldots,u_n)$
consists of infinitely differentiable functions of all variables,
this result can be inferred from the corrector result in Theorem~\ref{multicorrector}.
\end{remark}
\begin{theorem} \cite{VHHCS2004}
Assume $A\in C^{0,1}(D,C^{0,1}_\#(Y_1,\ldots,C^{0,1}_\#(Y_n)\ldots))$
so that in particular $A$ is Lipschitz with respect to each variable,
and is symmetric.
Then the homogenized coefficient $A_0$ is Lipschitz in $D$.

Assume moreover that the physical domain $D$
has a smooth boundary and that $f\in L^2(D)$.
Then the solution $(u_0,u_1,\ldots,u_n)$ of the limit problem~\eqref{limeqn}
satisfies $u_0\in H^2(D)$.
\end{theorem}
\subsection{Convergence in physical variables for two scales}
\label{sec:Cnv2scal}
We estimate the error between the solution $u^\varepsilon$ of the physical
problem~\eqref{Eq:ProblemEps}
in terms of the FE approximations of the limit problem (\ref{limeqn}).
We base this on an explicit error estimate between $u^\varepsilon$ and the correctors
for the two scale case ($n=1$).
\begin{proposition}
Assume that
$A\in C^\infty(\overline{D},C^\infty_\#(Y_1))^{d\times d}_{sym}$
and that the homogenized solution $u_0$ belongs to $H^2(D)$.
Then
\begin{equation}\label{corrector}
	\Norm[2]{ u^\varepsilon - \Par[2]{u_0(x) + \varepsilon u_1\Par[2]{x,\frac{x}{\varepsilon} } } }_{H^1(D)}
	\leq
	C \varepsilon^{\frac12}
	\, .
\end{equation}
The constant $C$ is independent of $\varepsilon$ but depends on $u_0$ and $u_1$.
\end{proposition}

\subsection{Recurrence for scale-interaction functions}

Let $i\in\Set{1,\ldots,n}$ and assume that $A_i$
is a matrix function satisfying Assumption~\ref{As:Coeff-i}
with $i$ microscales and positive constants $\gamma$ and $\Gamma$.
Then the limit problem~\eqref{limeqn}, posed on $D \times\bm{Y}_i$,
is well posed and has a unique solution.

Assume that $\xi\in\R^d$ is a unit vector. 
For a.e. $\Tuple{x,\bm{y}_{i-1}}\in D \times \bm{Y}_{ i-1}$,
define a bilinear form
$\mathsf{b}_i\Par{x,\bm{y}_{i-1},\cdot \,,\cdot\,} \!:  H^1_{\#}\Par{Y_i}/_\R \times H^1_{\#}\Par{Y_i}/_\R \rightarrow \R$
and a linear form
$\mathsf{f}_i\Par{x,\bm{y}_{i-1},\xi,\cdot \,} \!:  H^1_{\#}\Par{Y_i}/_\R \rightarrow \R$
as follows:
\begin{equation}\label{eq:bi-fi}
	\begin{aligned}
		\mathsf{b}_i \Par{x,\bm{y}_{i-1},\psi,\phi}
		&=
		\int_{Y_i}
		\;
		\Par{ \nabla \psi}^\MT
		A_i\Par{x,\bm{y}_{i-1}, \cdot \,}
		\,
		\nabla \phi
		\, ,
		\\
		\mathsf{f}_i \Par{x,\bm{y}_{i-1},\xi,\phi}
		&=
		-
		\int_{Y_i}
		\;
		\xi^\MT \!
		A_i\Par{x,\bm{y}_{i-1}, \cdot \,}
		\,
		\nabla \QQ \phi
	\end{aligned}
\end{equation}
for all $\psi,\phi\in H^1_{\#}\Par{Y_i}/_\R$.
Then the following holds
for a.e.
$\Tuple{x,\bm{y}_{i-1}}\in D \times \bm{Y}_{ i-1}$.

First, the assumption regarding $A_i$
results in
the continuity and ellipticity of
$\mathsf{b}_i \Par{x,\bm{y}_{i-1},\cdot \,,\cdot \,}$:
for all $\psi,\phi \in H^1_{\#}\Par{Y_i}/_\R$,
$
	\mathsf{b}_i \Par{x,\bm{y}_{i-1},\psi,\phi}
	\leq
	\Gamma
	\QQ
	\Seminorm{\psi}_{H^1\Par{Y_i} } \Seminorm{\phi}_{H^1\Par{Y_i} }
$
and
$
	\mathsf{b}_i \Par{x,\bm{y}_{i-1},\phi,\phi}
	\geq
	\gamma
	\QQ
	\Norm{\phi}_{H^1_{\#}\Par{Y_i}/_\R}^2
$.
Second, by the same argument, the linear form
$\mathsf{f}_i \Par{x,\bm{y}_{i-1},\xi,\cdot \,}$
is continuous:
\begin{equation}\nonumber
	\Abs[1]{ \mathsf{f}_i \Par{x,\bm{y}_{i-1},\xi,\phi} }
	\leq
	\Gamma
	\QQ
	\Norm{\phi}_{H^1_{\#}\Par{Y_i}/_\R}
	\quad\text{for all}\quad
	\phi \in H^1_{\#}\Par{Y_i}/_\R
	\, .
\end{equation}
By the Lax--Milgram theorem, the problem
of finding
$\mathsf{w}_\xi\Par{x,\bm{y}_{i-1}, \cdot \,} \in H^1_{\#}\Par{Y_i}/_\R$
such that
\begin{equation}\label{Eq:wi}
	\mathsf{b}_i
	\Par{x,\bm{y}_{i-1},
		\mathsf{w}_\xi \Par{x,\bm{y}_{i-1}, \cdot \,}
		,
		\phi
	}
	=
	\mathsf{f}_i
	\Par{x,\bm{y}_{i-1},\xi,\phi}
	\quad\text{for all}\quad
	\phi \in H^1_{\#}\Par{Y_i}/_\R
	\, .
\end{equation}
admits a unique solution, which satisfies
$
	\Norm{\mathsf{w}_\xi \Par{x,\bm{y}_{i-1}, \cdot \,}}_{ H^1_{\#}\Par{Y_i}/_\R }
	\leq
	\gamma^{-1} \QQ \Gamma
$.

Let
$\xi_1,\ldots,\xi_d$ be the columns of the identity matrix $I$ of order $d$.
Being valid for a.e. $\Tuple{x,\bm{y}_{i-1}}\in D \times \bm{Y}_{ i-1}$
and every unit vector $\xi\in\R^d$,
the above argument
defines
$w_i \in W_i^d$
whose components
$w_{i k} \in W_i$ with $k\in\Set{1,\ldots,d}$
are given by
$
	w_{i k} \Par{x,\bm{y}_{i-1},y_i}
	=
	\mathsf{w}_{\xi_k}\Par{x,\bm{y}_{i-1},y_i}
$
for a.e.
$\Tuple{x,\bm{y}_{i-1},y_i} \in D \times \bm{Y}_{ i-1} \times Y_i$
and for each $k\in\Set{1,\ldots,d}$.
Note that $w_i$ is also an element of $V_i^d$. Furthermore,
it is the only element of $V_i^d$ such that
%
\begin{equation}\label{Eq:wiw}
	\int_{D \times \bm{Y}_i}
	\Par[1]{I + \Jac_i w_i
	}
	\,
	A_i
	\,
	\nabla_{\! i} \QQ \phi
	=
	0
\end{equation}
%
for all $\phi \in V_i$.
Here, 
$\Jac_i$ denotes the differential operator returning the Jacobi matrix
with respect to the last variable (varying in $Y_i$),
as a function of all variables (taking values in $D \times \bm{Y}_i$).

Since
$A_i \in L^{\infty}\Par{D  \QQ ; \,   C_{\#} \Par{\bm{Y}_i \QQ ; \, \R^{d \times d}_{\text{\rm sym}} \, }}$,
one can define
$A_{i-1} \in L^{\infty}\Par{D  \QQ ; \,   C_{\#} \Par{\bm{Y}_{ i-1} \QQ ; \, \R^{d \times d}_{\text{\rm sym}} \, }}$
by setting
\begin{equation}\label{Eq:CoeffUpscaling}
	A_{i-1}(x,\bm{y}_{i-1})
	\begin{aligned}[t]
		&=
		\int_{Y_i}
		\Par[1]{I + \Jac_i \QQ  w_i\Par{x,\bm{y}_{i-1}, \cdot \,} }
		A_i(x,\bm{y}_{i-1},\cdot \,)
		\,
		\Par[1]{I + \Jac_i \QQ  w_i\Par{x,\bm{y}_{i-1}, \cdot \,} }^\MT
		\\
		&=
		\int_{Y_i}
		\Par[1]{I + \Jac_i \QQ  w_i\Par{x,\bm{y}_{i-1}, \cdot \,} }
		A_i(x,\bm{y}_{i-1},\cdot \,)
	\end{aligned}
\end{equation}
for a.e. $x\in D$
and
for all $\bm{y}_{i-1} \in \bm{Y}_{ i-1}$. 
By~\cite[Theorem~3.9]{BLP1978},
the matrix function $A_{i-1}$,
which is called an \emph{upscaled coefficient},
satisfies Assumption~\ref{As:Coeff-i}
with $i$ microscales and with the identical positive constants $\gamma$ and $\Gamma$.
The corresponding problem~\eqref{limeqn}, involving $i$ variables,
is therefore well posed and has a unique solution
$\Tuple{u_0,\ldots,u_{i-1}} \in \bm{V}_{i-1}$.

Since $u_{i-1} \in V_{i-1}$, we have
$\nabla_{\! i-1} \, u_{i-1} \in L^2 \Par{D \times \bm{Y}_{ i-1} }^d$.
On the other hand, we have noted that
$w_i \in W_i^d$,
so we can define $u_i \in V_i$ by setting
\begin{equation}\label{eq:uivec2}
	u_i(x,\bm{y}_{i-1},\cdot \,)
	=
	\Par[1]{ w_i \Par{x,\bm{y}_{i-1},\cdot \,} }^\MT
	\,
	\nabla_{\! i-1} \, u_{i-1} \Par{x,\bm{y}_{i-1}}
	\quad\text{in}\quad
	H^1_{\#}\Par{Y_i}/_\R
\end{equation}
for a.e. $x\in D$ and $\bm{y}_{i-1} \in \bm{Y}_{ i-1}$.
Indeed, this entails that $u_i(x,\bm{y}_{i-1},\cdot \,)$ has the
gradient
\begin{equation}\label{eq:uivec2grad}
	\nabla_{\! i} \QQ u_i(x,\bm{y}_{i-1},\cdot \,)
	=
	\Jac_i w_i \Par{x,\bm{y}_i}
	\,
	\nabla_{\! i-1} \, u_{i-1} \Par{x,\bm{y}_{i-1}}
	\quad\text{in}\quad
	L^2\Par{Y_i}
\end{equation}
for a.e. $x\in D$ and $\bm{y}_{i-1} \in \bm{Y}_{ i-1}$,
so that the bound
$\Norm{u_i}_{V_i}^2 \lesssim \Norm{w_i}_{W_i^d}^2 \, \Norm{ u_{i-1}}_{V_{i-1}}^2$
holds with an equivalence constant determined by the choice of a norm
for $W_i^d$. 
This implies that $\Tuple{u_0,\ldots,u_{i-1},u_i} \in \bm{V}_i$
and, as one verifies using~\eqref{Eq:wiw} and~\eqref{Eq:CoeffUpscaling},
also that this tuple solves the problem~\eqref{limeqn} with $i+1$ variables.

Applying the above argument iteratively,
we obtain the ``effective'' macroscopic diffusion coefficient
$A_0\in L^\infty\Par{D ; \, \R^{d \times d}_{\text{\rm sym}} }$:
\begin{equation}\label{eq:DefA0}
	A_0
	=
	\int_{Y_1}
	\cdots
	\int_{Y_n}
	\Par[1]{I + \Jac_{1} \QQ  w_{1} }
	\cdots
	\Par[1]{I + \Jac_{n} \QQ  w_{n} }
	\,
	A
	\, ,
\end{equation}
which satisfies Assumption~\ref{As:Coeff-i} with zero microscales
and with the identical constants $\gamma$ and $\Gamma$.
The ``effective'' problem for the homogenized limit $u_0$ reads:
find $u_0 \in V_0$
such that for every $\phi \in V_0$
\begin{equation}\label{eq:Phom}
	\int_D
	\,
	\Par{\nabla \phi}^\MT
	A_0
	\,
	\nabla u_0
	=
	\int_D f \phi
	\;.
\end{equation}
Then the solution $\Tuple{u_0,\ldots,u_n}\in\bm{V}_n$
of the limit problem~\eqref{limeqn}
with $n+1$ variables
can be solved using the recursion~\eqref{eq:uivec2},
so that
the scale-interaction functions $u_i$ and the sums of their gradients given by~\eqref{eq:def-vi}
satisfy
\begin{equation}\label{eq:uivec}
	u_i
	=
	w_i^\MT
	v_{i-1}
	\quad\text{and}\quad
	v_i
	=
	\Par[1]{I + \Jac_i w_i }^\MT
	v_{i-1}
	=
	\Par[1]{I + \Jac_i w_i }^\MT
	\cdots
	\Par[1]{I + \Jac_1 w_{1} }^\MT
	\nabla u_0
\end{equation}
in $V_i$ and $L^2(D \times \bm{Y}_i)^d$ respectively.
\subsection{Approximate recurrence for scale-interaction functions}
\label{sc:ApproxRec}
In order to obtain low-rank tensor-structured approximations of
$\Tuple{u_0,u_1,\ldots,u_n} \in \bm{V}_n$,
we use the following approximation scheme
with a discretization parameter $L\in\N$.
For every $i=1,\ldots,n$,
we approximate $w_i$ and $\Jac_i w_i $
by $w_i^L$ and $J_i^L$
in $W_i^d$ and $L^\infty(D \times \bm{Y}_{ i-1}, L^2(Y_i))^{d \times d}$
respectively.
Assuming that $u_0$ and $\nabla u_0$ are approximated by $u_0^L$ and $v^L_0$
in $V$ and $L^2 (D)^d$ respectively,
we follow~\eqref{eq:uivec} to define the corresponding approximations
$u_i^L$ and $v_i^L$ to $u_i$ and $v_i$ with $i\in\Set{1,\ldots,n}$:
in $V_i$ and $L^2(D \times \bm{Y}_i)^d$ respectively, we set
\begin{equation}\label{eq:uvi-approx}
	u_i^L
	=
	\Par[1]{ w_i^L }^\MT
	\!
	v_{i-1}^L
	\quad\text{and}\quad
	v_i^L
	=
	\Par[1]{I + J_i^L }^\MT
	v_{i-1}^L
	=
	\Par[1]{I + J_i^L }^\MT
	\cdots
	\Par[1]{I + J_1^L }^\MT
	\!
	v_0^L
	\, .
	%
\end{equation}
The associated errors can be represented by telescoping sums:
for example,
\begin{multline}\nonumber
	v_i - v_i^L
	=
		\Par[1]{I + \Jac_i w_i }^\MT
		\cdots
		\Par[1]{I + \Jac_1 w_{1} }^\MT
		\Par[1]{ v_0 - v_0^L}
		\\
		+
		\sum_{j=1}^i
		\CuBr[3]{
			\prod_{m=j+1}^i
			\Par[1]{I + \Jac_m w_m }^\MT
		}
		\,
		\Par[1]{\Jac_j w_j - J_j^L }^\MT
		\,
		\CuBr[3]{
			\prod_{m=1}^{j-1} \Par[1]{I + J_m^L}^\MT
		}
		\,
		v_0^L
\end{multline}
for every $i\in\Set{1,\ldots,n}$,
where sums and products over empty ranges are to be omitted.
Assuming that the errors
$w_i - w_i^L$, $\Jac_i w_i - J_i^L$ and
$v_0 - v_0^L$
are bounded, respectively, in
$W_i^d$, $L^\infty(D \times \bm{Y}_{ i-1}, L^2(Y_i))^{d \times d}$
and $L^2 (D)^d$
uniformly with respect to $L\in\N$ and $i\in\Set{1,\ldots,n}$,
we obtain, with a positive equivalence constant
independent of the discretization parameter $L\in\N$, the bounds
\begin{equation}\label{eq:rec-approx-error-grad}
	\Norm[1]{
		v_i - v_i^L
	}_{
		L^2(D \times \bm{Y}_i)^d
	}
	\lesssim
		\Norm{ v_0 - v_0^L }_{L^2 (D)^d}
		+
		\sum_{j=1}^i
		\Norm{\Jac_j w_j - J_j^L}_{
			L^\infty(D \times \bm{Y}_{ i-1}, L^2(Y_i))^{d \times d}
		}
\end{equation}
and
\begin{equation}\label{eq:rec-approx-error}
	\Norm{
		u_i - u_i^L
	}_{
		V_i
	}
	\lesssim
		\,
		\Norm{ v_0 - v_0^L }_{L^2 (D)^d}
		+
		\sum_{j=1}^{i-1}
		\Norm{\Jac_j w_j - J_j^L}_{
			L^\infty(D \times \bm{Y}_{ i-1}, L^2(Y_i))^{d \times d}
		}
		+
		\,
		\Norm{w_i - w_i^L}_{W_i}
	%
\end{equation}
for $i\in\Set{1,\ldots,n}$.
In Section~\ref{sec:high-dim-approx}, we construct
particular approximations
$w_i^L$, $J_i^L$, $u_i^L$ and $v_i^L$ with $i\in\Set{1,\ldots,n}$ and $L\in\N$
in the finite-element spaces specified in Section~\ref{sec:FEBasSpc}.

\section{Approximability under the assumption of analyticity}
\label{sec:Approx}
In the present section, we investigate
regularity and approximability of $u_0,u_1,\ldots,u_n$.
With the aim of establishing convergence rates and
(quantized) tensor rank bounds which are independent of the scales,
we impose additional assumptions on the data $D$, $A$ and $f$.
Specifically, we consider a \emph{tensor-product physical domain}
and \emph{analytic data}.

The first set of additional assumptions consists in the following.
\begin{assumption}\label{ass:Analytic1}
For every $\varepsilon$ and $i\in\Set{0,1,\ldots,n}$, we have
$\varepsilon_i=2^{-\lambda_i}$ with
$\lambda_i\in\N$ depending on $\varepsilon$
(we set $\lambda_0 \equiv 0$ for notational convenience).
For the physical domain and the unit cells, we have
$D = Y_1=\cdots=Y_n = \IntOO{0}{1}^d$.
The diffusion coefficient
$A$ is analytic and one-periodic with respect to
each of the last $nd$ scalar variables on $\overline{D \times \bm{Y}_{ n}}$.
The right-hand side $f$ is analytic on $\overline{D \times \bm{Y}_{ n}}$.

\end{assumption}

Assumption~\ref{ass:Analytic1} allows to prove that the
solution of the one-scale high-dimensional limiting problem
can be approximated by finite-element functions
of tensor ranks that are logarithmic in accuracy.
This implies that the solution of the
one-scale high-dimensional limiting problem
admits an infinite sequence of approximations that converge exponentially
with respect to the number of parameters used to represent them.

\subsection{Low-order finite-element approximation}
\label{subsec:FE}
%
In this section, we extend the construction of
finite-element spaces given in Section~\ref{sc:fesp-intro}
to address the boundary conditions of the
high-dimensional problem~\eqref{limeqn}
and establish main approximation results.
As stated in Assumption~\ref{ass:Analytic1},
We consider the case $D=Y_1=\cdots=Y_n=\IntOO{0}{1}^d$.
\subsubsection{Low-order approximation on an interval}
\label{sc:LowOrdApprox}
%
For $L\in\N$, in order to accommodate the periodic boundary conditions of the
high-dimensional problem~\eqref{limeqn},
we consider the following subspace of $\widetilde{U}^L$:
\begin{equation}\nonumber
	U_{\#}^L
	=
	\Span \Set[1]{\varphi^L_{\# \QQ j} \!:\; j\in\mathcal{I}^L}
	\, ,
\end{equation}
where
$\varphi^L_{\# \QQ j} = \varphi^L_j$
for every
$j\in\mathcal{J}^L$
and
$\varphi^L_{\# \QQ 2^L} = \varphi^L_0 + \varphi^L_{2^L}$.

We will use the \emph{analysis operators}
introduced in \eqref{Eq:DefOpEF1d-intro}
to extract the coefficients of finite-element approximations
in
$U^L$, $\bar{U}^L$ and $U_{\#}^L$.
To construct such approximations, we will use
the following projection operators, 
$
	\pi^L
	\colon
	H^1\IntOO{0}{1}
	\to
	\widetilde{U}^L
$
and
$
	\bar{\pi}^L
	\colon
	L^2\IntOO{0}{1}
	\to
	\bar{U}^L
$.
The first we define as the
operator of continuous, piecewise-linear Lagrange interpolation
at the nodes given in~\eqref{Eq:Part1dNodes-intro},
in the basis of
$\varphi^L_j$ with $j\in\Set{0}\cup\mathcal{I}^L$.
The second operator we define as
the operator of piecewise-constant $L^2$ approximation
associated with the basis functions
$\bar\varphi_i^L$ with $i \in \mathcal{I}^L$,
which are defined in Section~\ref{sc:fesp-intro}.
Note that
$
(\pi^L v)' = \bar{\pi}^L v'
$
for every
$v\in H^1\IntOO{0}{1}$.
Finally, both the projection operators
can be expressed in terms of the analysis operators
defined in~\eqref{Eq:DefOpEF1d-intro}:
for all $u\in H^1_0\IntOO{0}{1}$,
$v\in H^1_{\#}\IntOO{0}{1}$
and $w\in L^2\IntOO{0}{1}$,
we have
\begin{equation}\label{Eq:ReApprOpDef1d}
	\pi^L u
	=
	\sum_{j\in\mathcal{I}^L}
	\Par{\varPhi^L u}_j
	\,
	\varphi^L_j
	\, , \quad
	\pi^L v
	=
	\sum_{j\in\mathcal{I}^L}
	\Par{\varPhi^L v}_j
	\,
	\varphi^L_{\# j}
	\quad\text{and}\quad
	\bar{\pi}^L w
	=
	\sum_{i\in\mathcal{I}^L}
	\Par{\varPhi^L w}_i
	\,
	\bar\varphi_i^L
	\, .
\end{equation}

In the following proposition,
we summarize classical bounds for the projection operators
$\pi^L$ and $\bar{\pi}^L$ for $L\in\N$.

\begin{proposition}\label{Pr:Appr01}
	For all
	$v\in C\IntCC{0}{1} \cap C^2\IntOO{0}{1}$,
	$w\in C\IntCC{0}{1} \cap C^1\IntOO{0}{1}$
	and $L\in\N$,
	the projections $\pi^L v$ and $\bar{\pi}^L w$
	satisfy the error bounds
	\begin{equation}\nonumber
		\begin{gathered}
			\Norm{ v - \pi^L v }_{ L^\infty\IntOO{0}{1} }
			\leq
			2^{-2L-3}
			\,
			\Norm{ v'' }_{ L^\infty\IntOO{0}{1}  }
			\, ,\quad
			\Norm{ (v - \pi^L v)' }_{ L^\infty\IntOO{0}{1} }
			\leq
			2^{-L}
			\,
			\Norm{ v'' }_{ L^\infty\IntOO{0}{1} }
			\, ,
			\\
			\Norm{ w-\bar{\pi}^L w }_{ L^\infty\IntOO{0}{1} }
			\leq
			2^{-L}
			\,
			\Norm{ w' }_{ L^\infty\IntOO{0}{1} }
		\end{gathered}
	\end{equation}
	and the stability bounds
	\begin{equation}\nonumber
		\begin{gathered}
			\Norm{ \pi^L v }_{ L^\infty\IntOO{0}{1} }
			\leq
			\Norm{ v }_{ L^\infty\IntOO{0}{1}  }
			\, ,\quad
			\Norm{ (\pi^L v)' }_{ L^\infty\IntOO{0}{1} }
			\leq
			\Norm{ v' }_{ L^\infty\IntOO{0}{1} }
			\, ,
			\\
			\Norm{ \bar{\pi}^L w }_{ L^\infty\IntOO{0}{1} }
			\leq
			\Norm{ w }_{ L^\infty\IntOO{0}{1} }
			\, .
		\end{gathered}
	\end{equation}

\end{proposition}

\subsubsection{Low-order approximation on $D \times \bm{Y}_i$}
\label{sec:FEBasSpc}

From the univariate bases defined above,
we obtain by tensorization $d$-variate bases
which span the corresponding finite-element spaces:
\begin{equation}\label{eq:fe-spaces}
	\begin{gathered}
		\widetilde{V}^L
		=
		\bigotimes_{k=1}^{d}
		\widetilde{U}^L
		\subset
		H^{1}(D)
		\, ,\quad
		V^L
		=
		\bigotimes_{k=1}^{d}
		U^L
		=
		\widetilde{V}^L
		\cap
		H^1_0\Par{D}
		\, ,
		\\
		V^L_{\#}
		=
		\bigotimes_{k=1}^{d}
		U_{\#}^L
		=
		\widetilde{V}^L
		\cap
		H^{1}_{\#}(Y)
		\quad\text{and}\quad
		\bar{V}^L
		=
		\bigotimes_{k=1}^{d}
		\bar{U}^L
		\subset
		L^2(D)
		=
		L^2(Y)
	\end{gathered}
\end{equation}
with $L\in\N$.

Using the spaces of $d$-variate finite-element functions specified above, define
\begin{equation}\label{eq:def-i-sp}
	\bar{V}_i^L = \Par[1]{ \bar{V}^L }^{\otimes (i+1)}
	\, ,
	\quad
			\widetilde{V}_i^L
			=
			\bar{V}_{i-1}^L
			\otimes
			\widetilde{V}^L
	\quad\text{and}\quad
			V_{\! \# i}^L
			=
			\bar{V}_{i-1}^L
			\otimes
			V^L_{\#}
\end{equation}
for all $i\in\Set{0,\ldots,n}$ and $L\in\N$.

Further, for all $i\in\Set{0,\ldots,n}$ and $L\in\N$,
to construct approximations by finite-element functions from
$V_{\! \# i}^L$, we will use the operators
$
	\bar\varPi_i^L	\colon
	L^2(D \times \bm{Y}_i)
	\to
	\bar{V}_i^L
$
and
$
	\varPi_i^L
	\colon
	L^2(D \times \bm{Y}_{ i-1})
	\otimes
	\Par[1]{ H^1_{\#}\IntOO{0}{1} }^{\otimes d}
	\to
	V_{\# i}^L
$
given by
\begin{equation}\label{eq-def-proj-op}
	\bar\varPi_i^L
	=
	\bigotimes_{j=0}^i \bigotimes_{k=1}^d
	\bar{\pi}^L
	\quad\text{and}\quad
	\varPi_i^L
	=
	\bar\varPi_{i-1}^L
	\otimes \,
		\bigotimes_{k=1}^d
		\pi^L
	\, .
\end{equation}

The following accuracy bounds for
$\bar\varPi_i^L$ and $\varPi_i^L$
with $i\in\Set{1,\ldots,n}$ and $L\in\N$
can be derived from
Proposition~\ref{Pr:Appr01}.
\begin{lemma}\label{Lm:Approx-ndh}
	Let $i\in\Set{0,\ldots,n}$
	and $\Norm{\cdot}_{\infty}$ denote
	$\Norm{\cdot}_{L^\infty(D \times \bm{Y}_i)}$.
	Assume that
	$v\in C^1(\overline{D \times \bm{Y}_i})$
	and
	$w\in C^3(\overline{D \times \bm{Y}_i})$.
	Then
	the following error bounds hold
	for all $L\in\N$ and $k\in\Set{1,\ldots,d}$:
	\begin{equation}\nonumber
		\begin{aligned}
			\Norm{v - \bar\varPi_i^L v}_{
				\infty
			}
			&\leq
			2^{-L}
			\!
			\sum_{j'=0}^{i}
			\sum_{k'=1}^d
			\Norm{\partial_{j' k'} w}_{
				\infty
			}
			\, ,
			\\
			\Norm{w - \varPi_i^L w}_{
				\infty
			}
			&\leq
			2^{-L}
			\sum_{j'=0}^{i-1}
			\sum_{k'=1}^d
			\Norm{\partial_{j' k'} w}_{
				\infty
			}
			+
			2^{-2L-3}
			\sum_{k'=1}^d
			\Norm{\partial^2_{i k'} w}_{
				\infty
			}
			\, ,
			\\
			\Norm{\partial_{i k} (w - \varPi_i^L w)}_{
				\infty
			}
			&\leq
			2^{-L}
			\!
			\sum_{j'=0}^{i-1}
			\sum_{k'=1}^d
			\Norm{\partial_{j' k'}\partial_{i k} w}_{
				\infty
			}
			+
			2^{-L}
			\sum_{k'=1}^d
			\Norm{\partial^2_{i k'} \partial_{i k} w}_{
				\infty
			}
			\, .
		\end{aligned}
	\end{equation}
\end{lemma}
We give a proof of Lemma~\ref{Lm:Approx-ndh} in the Appendix.

For all $i\in\Set{0,\ldots,n}$ and $L\in\N$,
the projections produced by
the operators $\bar\varPi_i^L$ and $\varPi_i^L$, defined
by~\eqref{eq-def-proj-op},
can be parametrized
by the coefficients extracted using
the analysis operators
$
	\bar\varPsi_i^L
	\colon
	L^2(D \times \bm{Y}_i)
	\to
	\C^{2^{(i+1)dL}}
$
and
$
	\varPsi_i^L
	\colon
	L^2(D \times \bm{Y}_{ i-1})
	\otimes \,
	\Par[1]{ H^1\IntOO{0}{1} }^{\otimes d}
	\to
	\C^{2^{(i+1)dL}}
$
given by
\begin{equation}\label{Eq:DefAnOp}
	\bar\varPsi_i^L
	=
	\bigotimes_{j=0}^i
	\Par[1]{
		\Ten{\varPi}^L
		\bigotimes_{k=1}^d
		\bar\varPhi^L
	}
	\quad\text{and}\quad
	\varPsi_i^L
	=
	\bar{\varPhi}_{i-1}^L
	\otimes
	\Par[1]{
		\Ten{\varPi}^L
		\bigotimes_{k=1}^d
		\varPhi^L
	}
	\, .
\end{equation}
Note that while the restrictions of
$\varPsi_i^L$
and
$\bar\varPsi_i^L$
to $V^L_{\# i}$ and $\bar{V}^L_i$
are bijective, 
that of
$\varPsi_i^L$
to $V^L_i$
is only injective.
The lack of bijectivity stems from the fact that we
use nested finite-element spaces~\eqref{eq:fe-spaces}--\eqref{eq:def-i-sp}
and represent every function from $V_i^L$
by $2^{(i+1)dL}$
values associated with a uniform tensor-product grid.
We take into account this lack of bijectivity
in the design of our numerical method.

\subsection{High-order approximation}
\subsubsection{High-order approximation on an interval}
By $\widetilde{T}_\alpha$ with $\alpha\in\Nz$,
we denote the Chebyshev polynomials of the first kind orthogonal
on $\IntOO{0}{1}$:
\begin{equation}\label{Eq:DefCheb}
	\widetilde{T}_\alpha (x)
	=
	\cos \CuBr[1]{n \arccos \Par{2x-1}}
	\quad\text{for all}\quad
	x\in\IntOO{0}{1}
	\quad\text{and}\quad
	\alpha\in\Nz
	\, ,
\end{equation}
so that the orthogonality property
holds
with respect to the weight function
$\omega$ given by
\begin{equation}\label{Eq:ChebWeight}
	\omega(x)
	=
	1/\sqrt{x \, (1-x) \, }
	\quad\text{for all}\quad
	x\in\IntOO{0}{1}
	\, .
\end{equation}
Specifically, we have
\begin{equation}\label{Eq:OrthCheb}
	\IProd{\widetilde{T}_\alpha}{\widetilde{T}_{\alpha'}}_{L^2_\omega \IntOO{0}{1}}
	=
	\int_0^1 \!\! \omega \: \widetilde{T}_\alpha \, \widetilde{T}_{\alpha'}
	=
	\delta_{\alpha \alpha'}
	\,
	\Norm{\widetilde{T}_\alpha}_{L^2_\omega \IntOO{0}{1}}^2
	\quad\text{for all}\quad
	\alpha,\alpha'\in\Nz
	\, ,
\end{equation}
where
$\Norm{\widetilde{T}_0}_{L^2_\omega \IntOO{0}{1}}^2 = \pi$
and
$\Norm{\widetilde{T}_\alpha}_{L^2_\omega \IntOO{0}{1}}^2 = \displaystyle\frac{\pi}{2}$
for all $\alpha\in\N$.

Further, we consider the complex exponentials
$\widehat{T}_{\alpha}$
with $\alpha\in\Z$
defined as follows:
\begin{equation}\label{Eq:DefExp}
	\widehat{T}_{\alpha} \Par{x}
	=
	\exp\Par{ 2 \pi \IU \alpha x}
	\quad\text{for all}\quad
	x\in\IntOO{0}{1}
	\quad\text{and}\quad
	\alpha\in\Z
	\, .
\end{equation}
These are also orthogonal on $\IntOO{0}{1}$:
\begin{equation}\label{Eq:OrthExp}
	\IProd{\widehat{T}_{\alpha}}{\widehat{T}_{\alpha'}}_{L^2 \IntOO{0}{1}}
	=
	\int_0^1 \!\! \widehat{T}_{\alpha}^* \, \widehat{T}_{\alpha'}
	=
	\delta_{\alpha \alpha'}
	\quad\text{for all}\quad
	\alpha,\alpha'\in \Z
	\, .
\end{equation}

We will use the following notation for the spaces
of
univariate algebraic and trigonometric polynomials
of degree at most $p\in\Nz$:
\begin{equation}\label{Eq:DefPSp}
	\mathscr{P}_p
	=
	\Span \Set{\widetilde{T}_\alpha}_{\alpha=0}^p
	\quad\text{and}\quad
	\mathscr{P}_{\# p}
	=
	\Span \Set[1]{\widehat{T}_\alpha}_{\alpha=-p}^p
	\; ,
\end{equation}
where the span is meant with respect to the field $\C$.

We will use polynomial approximations obtained by the following orthogonal
projections onto
$\mathcal{P}_p$ and $\mathcal{P}_{\! \! \# p}$
with $p\in\Nz$:
	\[
		\begin{gathered}
			\pi_p
			=
			\frac{1}{\pi}
			\,
			\widetilde{T}_0
			\,
			\IProd{\widetilde{T}_0}{\cdot \,}_{L^2_\omega\IntOO{0}{1}}
			+
			\frac{2}{\pi}
			\,
			\sum_{ \alpha = 1}^{p}
			\widetilde{T}_{\alpha}
			\,
			\IProd{\widetilde{T}_{\alpha}}{\cdot \,}_{L^2_\omega\IntOO{0}{1}}
			\colon L^2_\omega\IntOO{0}{1} \to \mathscr{P}_p
			\, ,
			\\
			\pi_{\# p}
			=
			\widehat{T}_0
			\,
			\IProd{\widehat{T}_0}{\cdot \,}_{L^2\IntOO{0}{1}}
			+
			\sum_{ \pm\alpha = 1 }^p
			\widehat{T}_{\alpha}
			\,
			\IProd{\widehat{T}_\alpha}{\cdot \,}_{L^2\IntOO{0}{1}}
			\colon L^2\IntOO{0}{1} \to \mathscr{P}_{\# p}
			\, .
		\end{gathered}
	\]

\subsubsection{High-order approximation on $\overline{D \times \bm{Y}_i}$}
\label{sec:HOapprDY}

For every $i\in\Set{1,\ldots,n}$,
denoting by $\mathsf{id}$ the identity transformation of $\C^{id}$,
let us define the following tensor-product operators:
\[
	\varPi_{ i,p}
	=
	\Par[2]{
		\bigotimes_{k=1}^d
		\pi_p
	}
	\otimes
	\Par[2]{
		\bigotimes_{j=1}^i \bigotimes_{k=1}^d
		\pi_{\# p}
	}
	\colon
	L_{\omega^{\otimes d} \otimes \mathsf{id}}^2(D \times \bm{Y}_i)
	\to
	\Par[2]{
		\bigotimes_{k=1}^d \mathscr{P}_p
	}
	\otimes
	\Par[2]{
		\bigotimes_{j=1}^i \bigotimes_{k=1}^d \mathscr{P}_{\# p}
	}
\]
for all $p\in\Nz$.
Here, $\omega$ denotes the weight function in \eqref{Eq:ChebWeight}.
The following lemma verifies that,
when applied to analytic functions, these operators yield
approximations that converge exponentially with respect to $p\in\Nz$.

\begin{lemma}\label{Lm:Approx-ndp}
	Assume that $i\in\Set{1,\ldots,n}$ and $w\in V_i$ is analytic
	and one-periodic with respect to each of the last $id$ scalar variables
	on
	$\overline{D \times \bm{Y}_i}$.
	Let $\epsilon_0 > 0$.
	Then there exist positive constants $C$ and $c$
	such that, for any
	$\epsilon\in\IntOO{0}{\epsilon_0}$
	and for
	$p = \Ceil{ c \QQ \log \epsilon^{-1} }$,
	the following bounds hold
	for all $k\in\Set{1,\ldots,d}$ and $j\in\Set{1,\ldots,i}$:
	\begin{equation}\label{Eq:Approx-nd-bound}
		\begin{gathered}
			\Norm{
				w - \varPi_{ i,p} \, w
			}_{L^\infty\Par{D \times \bm{Y}_i}}
			\leq
			C \QQ \epsilon
			\, ,
			\qquad
			\Norm{
				\partial_{k}
				(w - \varPi_{ i,p} \, w)
			}_{L^\infty\Par{D \times \bm{Y}_i}}
			\leq
			C \QQ \epsilon \QQ p^2
			\, ,
			\\
			\Norm{
				\partial_{jk}
				(w - \varPi_{ i,p} \, w)
			}_{L^\infty\Par{D \times \bm{Y}_i}}
			\leq
			C \, \epsilon \QQ p
			\, .
		\end{gathered}
	\end{equation}
\end{lemma}
The result is classical; for completeness,
we provide a proof of Lemma~\ref{Lm:Approx-ndp} in the Appendix.

\begin{lemma}\label{Lm:Approx-ndhp}
	Let the assumptions of Lemma~\ref{Lm:Approx-ndp} hold
	and $\Norm{\cdot}_{\infty}$ denote
	$\Norm{\cdot}_{L^\infty(D \times \bm{Y}_i)}$.
	Then there exist positive constants $C$ and $c$
	such that, for any
	$L \in \N$
	and for
	$p = \Ceil{ c \QQ L }$,
	the following bounds hold
	for every $k\in\Set{1,\ldots,d}$:
	\begin{equation}\nonumber
		\Norm{
			\partial_{i k} \Par{
				w
				-
				\varPi_i^L \varPi_{ i,p}
				\,
				w
			}
		}_{\infty}
		\leq
		C \, p^2 \, 2^{-L}
		\quad\text{and}\quad
		\Norm{
			\partial_{i k} w
			-
			\bar\varPi_i^L
			\partial_{i k}
			\QQ
			\varPi_{ i,p}
			\,
			w
		}_{\infty}
		\leq
		C \, p^2 \, 2^{-L}
		\, .
	\end{equation}
\end{lemma}
We give a proof of Lemma~\ref{Lm:Approx-ndhp} in the Appendix.

\subsection{Low-rank tensor approximation}
\label{sc:LowRkTnsAppr}
In this section, for $i\in\Set{1,\ldots,n}$ and $L\in\N$, we consider
\begin{equation}\label{Eq:wi-pol-approx}
	w_i^L = \Par{ \varPi_i^L \varPi_{ i,p_L} w_{ik}}_{k=1}^d \in V_i^L
	\quad\text{and}\quad
	J_i^L  = \Par{ \bar\varPi_i^L \partial_{i k} \QQ \varPi_{ i,p_L} w_{i k'}}_{k',k=1}^d \in \Par[1]{ \QQ \bar{V}_i^L \QQ}^{d \times d}
\end{equation}
with a suitable $p_L \in \N$
as approximations to $w_i$ and $\Jac_i w_i$, where
$w_i$ is the solution of~\eqref{Eq:wiw}.
Then the approximation scheme~\eqref{eq:uvi-approx}
produces $u_i \in V_i^L$ and $v_i \in \Par[1]{ \QQ \bar{V}_i^L \QQ}^d$.

Section~\ref{sec:approx-acc} relates the error of
the approximation scheme~\eqref{eq:uvi-approx},
bounded by~\eqref{eq:rec-approx-error-grad}--\eqref{eq:rec-approx-error},
to the error of $w_i^L$ and $J_i^L$ as approximations to $w_i$ and $\Jac_i w_i$
for all $i\in\Set{1,\ldots,n}$ and $L\in\N$.

In Section~\ref{sec:approx-aux},
the error bounds proved in  Section~\ref{sec:approx-acc} are
followed by a quantized tensor-rank analysis,
which is based on auxiliary definitions and
rank bounds which are also provided in Section~\ref{sec:approx-aux}.

The analysis is based
on the following assumption regarding the approximation $v_0^L \in \bar{V}^L$
to $v_0$.
\begin{assumption}\label{ass:u0-low-rank}
	For all $L\in\N$,
	the subspace
	$\mathscr{S}^L = \Span\Set{\bar\varPsi_0^L v_{0,k}^L}_{k=1}^d$
	satisfies the following with some rank $r_L\in\N$:
	for every $\ell\in\Set{1,\ldots,L-1}$,
	there exist subspaces
	$\mathscr{L}^L_\ell\subset\R^{2^{d\ell}}$ and $\mathscr{M}^L_\ell\subset\R^{2^{d(L-\ell)}}$
	of dimensions at most $r_L$
	such that
	$\mathscr{S}^L \subset \mathscr{L}^L_\ell \otimes \mathscr{M}^L_\ell$.
\end{assumption}

The purpose of Sections~\ref{sec:high-dim-approx} and~\ref{sec:multiscale-approx}
is to bound the tensor ranks of
the coefficients of $u_1^L,\ldots,u_n^L$ and $\mathcal{U}^\varepsilon v_n^L$ 
in terms of both $L\in\N$ and $r_L\in\N$.
Finally, in Section~\ref{sec:RegAnSingFct}, we restrict the setting to the case of $d=2$
physical dimensions and invoke a result from~\cite{Kazeev:PhD,KS:2017:QTTFE2d}.
We remark that corresponding results
in space dimension $d=3$ are also available in \cite{MRS19_872}.
This gives simultaneous bounds on $r_L$ and on the errors
$\Norm{u_0 - u_0^L}_{L^2(D)}$
and
$\Norm{v_0 - v_0^L}_{L^2(D)^d}$
for every $L\in\N$,
which lead to an analogous result for 
$\bar{\varPsi}_{\! \star n}^L \, \mathcal{U}^\varepsilon v_n^L$.
\subsubsection{Accuracy of the approximation scheme}\label{sec:approx-acc}
Under Assumption~\ref{ass:Analytic1},
differentiating the equation expressing
the cell problem~\eqref{Eq:wiw} in the strong form
iteratively for $i=n,\ldots,1$,
one verifies that the solutions $w_i \in V_i^d$
with $i\in\Set{1,\ldots,n}$
satisfy the assumption of Lemma~\eqref{Lm:Approx-ndp}.
This gives that, 
with a positive constant $c$,
for any
$L \in \Nz$
and for
\begin{equation}\label{eq:wi-pol-degree}
	p_L = \Ceil{ c \QQ L }
	\, ,
\end{equation}
the approximations $w_i^L$ and $J_i^L$ defined by~\eqref{Eq:wi-pol-approx}
satisfy the error bounds
\begin{equation}\nonumber
		\Norm{ w_i - w_i^L }_{W_i}
	\lesssim
		L^2 \, 2^{-L}
		\quad\text{and}\quad
		\Norm{
			\Jac_i w_i - J_i^L
		}_{
			L^\infty(D \times \bm{Y}_{ i-1}, L^2(Y_i))^{d \times d}
		}
	\lesssim
		L^2 \, 2^{-L}
	%
\end{equation}
The equivalence holds with a positive constant
that is independent of $L\in\Nz$ and $i\in\Set{1,\ldots,n}$.
Then the bounds~\eqref{eq:rec-approx-error-grad}--\eqref{eq:rec-approx-error}
for the approximation scheme~\eqref{eq:uvi-approx} show that the resulting approximations
$u_i^L\in V_i$ and $v_i^L\in L^2(D \times \bm{Y}_i)^d$
satisfy the bounds
\begin{equation}\label{eq:rec-approx-error-exp}
	\Norm{
		u_i - u_i^L
	}_{
		V_i
	}
	\, , \,
	\Norm[1]{ v_i - v_i^L }_{ L^2(D \times \bm{Y}_i)^d }
	\lesssim
	\Norm{ v_0 - v_0^L }_{L^2 (D)^d} + L^2 \, 2^{-L}
\end{equation}
with a positive equivalence constant
independent of $L\in\N$ and $i\in\Set{1,\ldots,n}$.

\subsubsection{Auxiliary subspaces and results}\label{sec:approx-aux}

For all $L\in\Nz$ and $p\in\Nz$,
we will use the following notation for the sets of tensors obtained by evaluating
$d$-variate
algebraic and trigonometric polynomials
of maximum degree at most $p$
on a uniform tensor-product grid with $2^L$ nodes in each variable:
\begin{equation}\label{Eq:DefPolCoeffSp}
	\mathcal{P}_p^{L,d}
	=
	\Ten{\varPi}^L
	\bigotimes_{k=1}^d
	\varPsi^L
	\QQ
	\mathscr{P}_p
	\subset
	\C^{2^{dL}}
	\quad\text{and}\quad
	\mathcal{P}_{\! \! \# p}^{L,d}
	=
	\Ten{\varPi}^L
	\bigotimes_{k=1}^d
	\bar\varPsi^L
	\QQ
	\mathscr{P}_{\# p}
	\subset
	\C^{2^{dL}}
	\, .
\end{equation}

Let us extend~\eqref{Eq:DefPSp}
and~\eqref{Eq:DefPolCoeffSp}
by introducing,
for all $p\in\Nz$, $L\in\N$ and $\lambda\in\Z$,
\begin{equation}\label{Eq:DefPSp-Coeff-Trig-ell}
	\mathscr{P}_{\# p,\lambda}
	=
	\Span \Set[1]{\widehat{T}_\alpha (2^{\lambda} \,\cdot\, ) }_{\alpha=-p}^p
	\quad\text{and}\quad
	\mathcal{P}_{\! \! \# p,\lambda}^{L,d}
	=
	\Ten{\varPi}^L
	\bigotimes_{k=1}^d
	\bar\varPsi^L
	\,
	\mathscr{P}_{\# p,\lambda}
	\subset
	\C^{2^{dL}}
	\, .
\end{equation}

We will use several results, stated below,
to analyze the low-rank structure of the approximations
$u_i^L\in V_i$ and $v_{i}^L\in L^2(D \times \bm{Y}_i)^d$
with
$i\in\Set{1,\ldots,n}$ and $L\in\N$,
defined by
\eqref{eq:uvi-approx} and~\eqref{Eq:wi-pol-approx},
as elements of the respective spaces
$\mathscr{Q}_i^L$ with $i\in\Set{1,\ldots,n}$ and $L\in\N$,
given by~\eqref{eq:ui-approx-Space}.
\begin{proposition}\label{Pr:AlgPolTensStruct}
	For all $p\in\Nz$, $L\in\N$ and $\ell\in\Set{1,\ldots,L-1}$,
	we have
	$
		\mathcal{P}_p^{L,d}
		\subset
		\mathcal{P}_p^{\ell,d}
		\otimes
		\mathcal{P}_p^{L-\ell,d}
	$.
\end{proposition}

The embedding stated in Proposition~\ref{Pr:AlgPolTensStruct} means the following:
for every $u\in \mathcal{P}_p^{L,d} \subset \C^{2^{dL}}$, there exist
$u'\in \mathcal{P}_p^{\ell,d} \subset \C^{2^{d\ell}}$
and
$u''\in \mathcal{P}_p^{L-\ell,d} \subset \C^{2^{d(L-\ell)}}$
such that $u = u' \otimes u''$ in the sense of the Kronecker product of vectors (tensors).
The proof follows trivially from the binomial formula applied to the standard
basis of monomials.

An immediate consequence of Proposition~\ref{Pr:AlgPolTensStruct}
is that
the tensor of the values of any $d$-variate polynomial
of maximum degree at most $p\in\N$
at any tensor-product uniform grid with $2^L$ entries in each dimension
can be represented in the multilevel TT-MPS format
with the transposition~\eqref{Eq:IndTranspMatrix}
with ranks not exceeding $(p+1)^d$.
This was originally shown,
in the case of $d=1$, in~\cite[Corollary~13]{Grasedyck:2010:HTens}.
The language of space factorization,
which we adopt for Assumption~\ref{ass:u0-low-rank}, Proposition~\ref{Pr:AlgPolTensStruct}
and for the whole section,
is different from that of~\cite{Oseledets:2009:QTT:Dokl,Oseledets:2010:QTT,Khoromskij:2011:QuanticsApprox,Grasedyck:2010:HTens};
we use it here to mostly avoid lengthy expressions with numerous indices
associated with nodes of tensor-product grids.

The additional notation~\eqref{Eq:DefPSp-Coeff-Trig-ell}
allows to state the following analog
of Proposition~\ref{Pr:AlgPolTensStruct} for trigonometric polynomials,
which is an immediate consequence of the
separability of the exponential function.
\begin{proposition}\label{Pr:TrigPolTensStruct}
	For all $p\in\Nz$, $L\in\N$
	and $\ell\in\Set{1,\ldots,L-1}$,
	we have
	$
		\mathcal{P}_{\! \! \# p
		}^{L,d}
		\subset
		\mathcal{P}_{\! \! \# p
		}^{\ell,d}
		\otimes
		\mathcal{P}_{\! \! \# p,
		-\ell}^{L-\ell,d}
	$.
\end{proposition}

\subsubsection{Approximation of the one-scale high-dimensional limit problem}\label{sec:high-dim-approx}


Apart from realizing arbitrary accuracy, the approximations
given by~\eqref{eq:uvi-approx} and~\eqref{Eq:wi-pol-approx}
are also structured in the sense that
\begin{equation}\label{Eq:wi-pol-approx-struct}
	\varPsi_i^L \QQ w_i^L
	\in
	\mathcal{P}_{p_L}^{L,d}
	\otimes
	\bigotimes_{j=1}^i \mathcal{P}_{\! \! \# p_L}^{L,d}
\end{equation}
for all $L\in\Nz$ and $i\in\Set{1,\ldots,n}$.
Then, for all $L\in\Nz$ and $i\in\Set{1,\ldots,n}$,
relations~\eqref{Eq:wi-pol-approx-struct}
and \eqref{eq:uvi-approx} with~\eqref{Eq:wi-pol-approx}
result in
$
	\varPsi_i^L \QQ u_i^L
	\in
	\mathscr{Q}_i^L
$
and
$
	\bar\varPsi_i^L \QQ v_i^L
	\in
	(\mathscr{Q}_i^L)^d
$, 
where
\begin{equation}\label{eq:ui-approx-Space}
	\mathscr{Q}_i^L
	=
	\CuBr[2]{ \mathscr{S}^L  \odot \, \mathcal{P}_{(i+1) \QQ p_L}^{L,d} }
	\otimes \,
	\bigotimes_{j=1}^i \mathcal{P}_{\! \! \# (i+1-j) \QQ p_L}^{L,d}
	\subset
	\C^{2^{(i+1)dL}}
	\, .
\end{equation}
Here, the operation ``$\odot$'' between two spaces
denotes taking the span of the set of pointwise products
of all pairs of elements from the respective spaces.

\begin{theorem}\label{thm:limit-problem}
	Let Assumptions~\ref{As:Coeff}, \ref{ass:Analytic1} and~\ref{ass:tech} hold
	and $\Tuple{u_0,u_1,\ldots,u_n} \in \bm{V}_n$ be the solution of~\eqref{limeqn}.
	Consider
	$v^L_0 \in \bar{V}^L$ with $L\in\N$ satisfying
	Assumption~\ref{ass:u0-low-rank} and
	such that
	$
		\Norm{v_0 - v_0^L}_{L^2 (D)^d}
		\leq
		C_0
		\,
		L^{\gamma_0}
		\,
		2^{-\alpha L}
	$
	for all $L\in\N$ with $\alpha\in\IntOC{0}{1}$
	and with positive constants $C_0$ and $\gamma_0$.
	Then the approximations
	$u_i^L \in V_i^L$ with $L\in\N$ and $i\in\Set{1,\ldots,n}$
	given by~\eqref{eq:uvi-approx} and~\eqref{Eq:wi-pol-approx}
	satisfy the following
	with positive constants $\tilde{C}$ and $\tilde{c}$ and
	$\tilde{\gamma} = \max\Set{2,\gamma_0}$.

	For all $L\in \N$ and $i\in\Set{1,\ldots,n}$,
	the bound
	$
		\Norm{u_i - u_i^L}_{V}
		\leq
		\tilde{C}
		\,
		L^{\tilde{\gamma}}
		\,
		2^{-\alpha L}
	$
	holds
	and
	the coefficient tensor
	$\varPsi_i^L u_i^L$
	admits a decomposition of the form~\eqref{eq:TT-MPS}
	with $(n+1)L$ levels
	and ranks bounded from above by
	$R_L = \tilde{c} \QQ L^{(n+1)d} \QQ r_L$
\end{theorem}
\begin{proof}
	Consider $i\in\Set{1,\ldots,n}$ and $L\in \N$.
	The claimed accuracy bound follows from~\eqref{eq:rec-approx-error-exp}.

	To bound the first $L$ ranks,
	we consider $\ell\in\Set{1,\ldots,L}$
	and factorize $\mathscr{Q}_i^L$ so that the first factor is a subspace
	of $\C^{2^{d\ell}}$. Then this subspace corresponds
	to the $\ell$ coarsest levels of the macroscale, and its dimension
	majorates the corresponding rank of $\varPsi_i^L \QQ u_i^L \in \mathscr{Q}_i^L$.
	First, we obtain from Proposition~\ref{Pr:AlgPolTensStruct} that
	$
		\mathcal{P}_{(i+1) \QQ p_L}^{L,d}
		\subset
		\mathcal{P}_{(i+1) \QQ p_L}^{\ell,d}
		\otimes
		\C^{2^{d(L-\ell)}}
	$,
	where the dimension of $\mathcal{P}_{(i+1) \QQ p_L}^{\ell,d}\subset\C^{2^{d\ell}}$
	is $((i+1)p_L + 1)^d$.
	On the other hand, by Assumption~\ref{ass:u0-low-rank},
	we have
	$\mathscr{S}^L \subset \mathscr{L}^L_\ell \otimes \C^{2^{d(L-\ell)}}$,
	where the dimension of $\mathscr{L}^L_\ell \subset\C^{2^{d\ell}}$
	does not exceed $r_L$.
	This results in
	$
		\mathscr{Q}_i^L
		\subset
		\widetilde{\mathscr{L}}^L_\ell
		\otimes
		\C^{2^{d(L-\ell)}}
		\otimes
		\,
		\bigotimes_{j=1}^i \mathcal{P}_{\! \! \# (i+1-j) \QQ p_L}^{L,d}
		%
	$,
	where
	$
		\widetilde{\mathscr{L}}^L_\ell
		=
		\mathscr{L}^L_\ell \odot \mathcal{P}_{(i+1) \QQ p_L}^{L,d}
	$.
	We note then that
	$
		\Dim \widetilde{\mathscr{L}}^L_\ell
		\leq
		\Dim
		\mathscr{L}^L_\ell
		\cdot
		\Dim
		\mathcal{P}_{(i+1) \QQ p_L}^{L,d}
		\leq
		((i+1)p_L+1)^d \, r_L
		\leq
		\tilde{c} \QQ L^d \QQ r_L = R_L
	$
	for a suitable positive constant $\tilde{c}$
	independent of $L$,
	due to the linear dependence~\eqref{eq:wi-pol-degree}
	of $p_L$ on $L$.

	To bound the other ranks,
	we now consider $j\in\Set{1,\ldots,i}$ and $\ell\in\Set{0,\ldots,L}$
	and factorize $\mathscr{Q}_i^L$ so that the last factor is a subspace
	of $\C^{2^{d(L-\ell) + (i-j)dL}}$. This subspace corresponds
	to
	the $L-\ell$ finest levels of the $j$th microscale and
	all levels of all finer microscales.
	The dimension of this
	subspace
	majorates the corresponding rank of $\varPsi_i^L \QQ u_i^L \in \mathscr{Q}_i^L$.

	Applying Proposition~\ref{Pr:TrigPolTensStruct}, we obtain
	$
		\mathcal{P}_{\! \! \# (i+1-j) \QQ p_L}^{L,d}
		\subset
		\mathcal{P}_{\! \! \# (i+1-j) \QQ p_L}^{\ell,d}
		\otimes
		\mathcal{P}_{\! \! \# (i+1-j) \QQ p_L,-\ell}^{L-\ell,d}
	$,
	where the dimension of both the factors is $(2(i+1-j)p_L + 1)^d$.
	Then we have
	\begin{equation}\nonumber
		\mathscr{Q}_i^L
		\subset
		\CuBr[2]{ \mathscr{S}^L  \odot \, \mathcal{P}_{(i+1) \QQ p_L}^{L,d} }
		\, \otimes \,
		\bigotimes_{m=1}^{j-1} \mathcal{P}_{\! \! \# (i+1-m) \QQ p_L}^{L,d}
		\, \otimes \,
		\mathcal{P}_{\! \! \# (i+1-j) \QQ p_L}^{\ell,d}
		\, \otimes \,
		\tilde{\mathscr{M}}^L_{j,\ell}
	\end{equation}
	with
	$
		\tilde{\mathscr{M}}^L_{j,\ell}
		=
		\mathcal{P}_{\! \! \# (i+1-j) \QQ p_L,-\ell}^{L-\ell,d}
		\, \otimes \,
		\bigotimes_{m=j+1}^i \mathcal{P}_{\! \! \# (i+1-m) \QQ p_L}^{L,d}
	$.
	For the last factor, we find that
	\begin{multline}\nonumber
		\Dim \tilde{\mathscr{M}}^L_{j,\ell}
		\leq
		\Dim \mathcal{P}_{\! \! \# (i+1-j) \QQ p_L,-\ell}^{L-\ell,d}
		\prod_{m=j+1}^i
		\Dim \mathcal{P}_{\! \! \# (i+1-m) \QQ p_L}^{L,d}
		=
		\prod_{m=j}^i
		(2 (i+1-m) \QQ p_L + 1)^d
		\\
		\leq
		\tilde{c} \QQ L^{(i+1)d}
		=
		R'_L
		\, .
	\end{multline}
	As above, the latter inequality holds with a suitable positive constant $\tilde{c}$
	independent of $L$
	due to the linear dependence~\eqref{eq:wi-pol-degree}
	of $p_L$ on $L$.
\end{proof}

\subsubsection{Approximation of the multiscale problem}\label{sec:multiscale-approx}


In this section, we analyze the low-rank structure and accuracy of
$\mathcal{U}^\varepsilon \QQ v_n^L$ as
an approximation to $\mathcal{U}^\varepsilon \QQ v_n$.
As in Section~\ref{sec:high-dim-approx},
we develop our analysis here under Assumptions~\ref{As:Coeff},
\ref{ass:Analytic1} and~\ref{ass:u0-low-rank}.
Additionally, we make the following technical assumption, which
simplifies the analysis of tensor structure in the present section.
\begin{assumption}\label{ass:tech}
	For every $i\in\Set{1,\ldots,n}$, we have $L \geq \lambda_i - \lambda_{i-1}$.
\end{assumption}

We start with defining finite-element subspaces in which we will consider
averaged approximations.
To this end, we set
	\begin{equation}\nonumber
		I
		=
		\begin{pmatrix}
			1 & 0 \\ 0 & 1
		\end{pmatrix}
		\quad\text{and}\quad
		\xi
		=
		\frac12
		\begin{pmatrix}
			1 \\ 1
		\end{pmatrix}
	\end{equation}
	and, for every $L\in\N$ and $i\in\Set{1,\ldots,n}$,
	consider the matrices
	\begin{equation}\label{eq:def-M}
		\tilde{\bm{M}}_{i-1}^L
		=
		I^{\otimes d(\lambda_i-\lambda_{i-1})}
		\otimes
		\Par[2]{\QQ\frac12 \QQ \xi^\MT}^{\otimes d(L - \lambda_i + \lambda_{i-1})}
		\quad\text{and}\quad
		\bm{M}_{i-1}^L
		=
		I^{\otimes d\lambda_{i-1}}
		\otimes
		\tilde{\bm{M}}_i^L
		\, .
	\end{equation}
	The action of the matrix $\bm{M}_{i-1}^L$ on the coefficient tensor of
	a piecewise-constant function subordinate to a uniform partition with
	$2^{d(L+\lambda_{i-1})}$ elements of linear size $2^{-(L+\lambda_{i-1})}$
	consists in averaging the function
	over the $2^{d \lambda_i}$ cells of scale $\varepsilon_i = 2^{-\lambda_i}$.
	The matrix $\tilde{\bm{M}}_{i-1}^L$, on the other hand, represents the same
	averaging operation on subtensors corresponding to single cells
	of scale $\varepsilon_{i-1} = 2^{-\lambda_{i-1}}$.
	The order of the factors in~\eqref{eq:def-M}
	reflects the use of transposition~\eqref{Eq:IndTranspMatrix}
	for the coefficient tensors, see~\eqref{Eq:DefAnOp}.

	Using the matrices introduced in~\eqref{eq:def-M}, we
	iteratively define the following spaces:
	$
		\mathscr{S}_0^L
		=
		\mathscr{S}^L
		\odot
		\,
		\mathcal{P}_{(n+1) \QQ p_L}^{L,d}
		\subset
		\C^{2^{dL}}
	$
	and
	\begin{equation}\label{eq:ui-approx-SpaceAvgX}
		%
		\mathscr{S}_i^L
		=
		\Par[1]{
			\bm{M}_{i-1}^L
			\,
			\mathscr{S}_{i-1}^L
		}
		\otimes
		\mathcal{P}_{\# (n+1-i) \QQ p_L}^{L,d}
		\subset
		\C^{2^{d(\lambda_i+L)}}
	\end{equation}
	for all $i\in\Set{1,\ldots,n}$ and $L\in\N$.
Eventually, we are interested the subspaces with index $i=n$,
which are relevant due to the following result. It is a corollary
of auxiliary technical Lemmas~\ref{lem:WL-avg} and~\ref{lem:Qstar-avg},
which we present in the Appendix.
\begin{lemma}\label{lem:Savg}
	For all $L\in\N$ 	and $v \in \bar{V}_n^L$ such that
	$\bar{\varPsi}_n^L \, v \in \mathscr{Q}_n^L$,
	we have
	$\mathcal{U}^\varepsilon \QQ v \in \bar{V}_{\! 0}^{\lambda_n+L}$ and
	$\bar{\varPsi}_0^{\lambda_n+L} \, \mathcal{U}^\varepsilon \QQ v \in \mathscr{S}_n^L$.
\end{lemma}

	Iterating~\eqref{eq:ui-approx-SpaceAvgX} under Assumption~\ref{ass:tech},
	we arrive at
	\begin{multline}\label{eq:ui-approx-SpaceAvgXX}
		\mathscr{S}_i^L
		=
		\Par{ \bm{M}_0^L \mathscr{S}_0^L }
		\otimes
		\Par[2]{
			\QQ
			\bigotimes_{j=1}^{i-1}
			\tilde{\bm{M}}_j^L \, \mathcal{P}_{\# (n+1-j) \QQ p_L}^{L,d}
		}
		\otimes
		\mathcal{P}_{\# p_L}^{L,d}
		\\
		=
		\Par{ \bm{M}_0^L \mathscr{S}_0^L }
		\otimes
		\Par[2]{
			\QQ
			\bigotimes_{j=1}^{i-1}
			\mathcal{P}_{\# (n+1-j) \QQ p_L}^{\lambda_{j+1}-\lambda_j,d}
		}
		\otimes
		\mathcal{P}_{\# p_L}^{L,d}
		\, ,
	\end{multline}
	where the second equality holds due to
	that
	$
		\tilde{\bm{M}}_j^L \, \mathcal{P}_{\# (n+1-j) \QQ p_L}^{L,d}
		=
		\mathcal{P}_{\# (n+1-j) \QQ p_L}^{\lambda_j-\lambda_{j-1},d}
	$
	for any $j\in\Set{1,\ldots,n}$
	by Proposition~\ref{Pr:TrigPolTensStruct}.
\begin{theorem}\label{thm:limit-problem-avg}
	Let Assumptions~\ref{As:Coeff}, \ref{ass:Analytic1} and~\ref{ass:tech} hold
	and $\Tuple{u_0,u_1,\ldots,u_n} \in \bm{V}_n$ be the solution of~\eqref{limeqn}.
	Consider $v_0 = \nabla u_0$,
	$v_n = \nabla u_0 + \nabla_1 u_1 + \cdots + \nabla_n u_n$
	and
	$v^L_0 \in \bar{V}^L$ with $L\in\N$ satisfying
	Assumption~\ref{ass:u0-low-rank} and
	such that
	$
		\Norm{v_0 - v_0^L}_{L^2 (D)^d}
		\leq
		C_0
		\,
		L^{\gamma_0}
		\,
		2^{-\alpha L}
	$
	for all $L\in\N$ with $\alpha\in\IntOC{0}{1}$
	and with positive constants $C_0$ and $\gamma_0$.
	Then the approximations
	$\mathcal{U}^\varepsilon \QQ  v_n^L \in (\bar{V}^{\lambda_n+L})^d$
	with $L\in\N$
	defined by~\eqref{eq:uvi-approx} and~\eqref{Eq:wi-pol-approx}
	satisfy the following
	with positive constants $\tilde{C}$ and $\tilde{c}$ and
	$\tilde{\gamma} = \max\Set{2,\gamma_0}$.

	For all $L\in\N$ sufficiently large (to satisfy Assumption~\ref{ass:tech}),
	the error bound
	$
		\Norm{
			\mathcal{U}^\varepsilon \QQ v_n - \mathcal{U}^\varepsilon \QQ  v_n^L
		}_{L^2(D)^d}
		\leq
		\tilde{C}
		\,
		L^{\tilde{\gamma}}
		\,
		2^{-\alpha L}
	$
	holds
	and
	the coefficient tensor
	$\bar{\varPsi}_0^{\lambda_n+L} \, \mathcal{U}^\varepsilon \QQ v_n^L$
	admits a decomposition of the form~\eqref{eq:TT-MPS}
	with $\lambda_n + L$ levels
	and with ranks bounded from above by
	$R_L = \tilde{c} \QQ L^{(n+1)d} \QQ r_L$.
\end{theorem}
\begin{proof}
	The claimed accuracy bound follows from~\eqref{eq:rec-approx-error-exp}
	combined with~\eqref{eq:avg-op-err}.

	The claimed rank bounds follows from the fact that
	$\bar{\varPsi}_0^{\lambda_n+L} \, \mathcal{U}^\varepsilon \QQ v_n^L \in \mathscr{S}_n^L$
	by Lemma~\ref{lem:Savg} combined with~\eqref{eq:ui-approx-SpaceAvgXX},
	Assumption~\ref{ass:u0-low-rank} and
	Propositions~\eqref{Pr:AlgPolTensStruct} and~\eqref{Pr:TrigPolTensStruct}.

	First, let us consider $\ell\in\Set{1,\ldots,\lambda_1 - 1}$.
	From Assumption~\ref{ass:u0-low-rank} and
	Proposition~\ref{Pr:AlgPolTensStruct}, we obtain
	$
		\mathscr{S}^L
		\subset
		\mathscr{L}^L_{\ell}
		\otimes
		\C^{2^{d(L-\ell)}}
	$
	and
	$
		\mathcal{P}_{(i+1) \QQ p_L}^{L,d}
		\subset
		\mathcal{P}_{(i+1) \QQ p_L}^{\ell,d}
		\otimes
		\C^{2^{d(L-\ell)}}
	$,
	where the dimensions of the first factors are bounded form above by
	$r_L$ and $((i+1)p_L + 1)^d$ respectively.
	This implies
	the inclusion
	$
		\bm{M}_0^L \mathscr{S}_0^L
		\subset
		\widetilde{\mathscr{L}}^L_\ell
		\otimes
		\C^{2^{d(\lambda_1-\ell)}}
	$
	with a subspace
	$\widetilde{\mathscr{L}}^L_\ell \subset \C^{2^{d\ell}}$
	of dimension at most $r_L \QQ ((i+1)p_L + 1)^d$
	and hence, by~\eqref{eq:ui-approx-SpaceAvgXX}, also
	$
		\mathscr{S}_n^L
		\subset
		\widetilde{\mathscr{L}}^L_\ell
		\otimes
		\C^{2^{d(\lambda_n+L-\ell)}}
	$.
	The first $\lambda_1-1$ ranks of
	$\bar{\varPsi}_0^{\lambda_n+L} \, \mathcal{U}^\varepsilon \QQ v_n^L$
	therefore do not exceed
	$r_L \QQ ((i+1)p_L + 1)^d$.

	To obtain a bound for all the remaining ranks at once, let us set
	$\lambda_{n+1} = \lambda_n + L$ for notational convenience and
	consider $\ell\in\Set{\lambda_k,\ldots,\lambda_{k+1} - 1}$
	with $k\in\Set{1,\ldots,n}$.
	Inasmuch as the corresponding factors indexed by $j\in\Set{k,\ldots,n}$
	in~\eqref{eq:ui-approx-Space} and~\eqref{eq:ui-approx-SpaceAvgX}
	with $i=n$ are completely analogous,
	the second part of the proof of Theorem~\ref{thm:limit-problem}
	applies herein upon replacing $L$ with $\lambda_{j+1}-\lambda_j$
	in superscript in the $j$th term for every $j\in\Set{k,\ldots,n}$.
\end{proof}

\subsubsection{The case of a separable scalar diffusion coefficient}\label{sec:sep-coeff}

Improved bounds can be obtained under additional
scale-separability assumptions
on the diffusion coefficient $A^\varepsilon$.
For example, let us consider the case when
the function $A$ is of separable form
\begin{equation}\nonumber
	A = \Par{a_0 \otimes a_1 \otimes \cdots \otimes a_n} \, K \, ,
\end{equation}
where
$K_i$ is a symmetric positive-definite matrix with spectrum in $\IntCC{\gamma}{\Gamma}$
for some positive constants $\gamma$ and $\Gamma$,
$a_0 \in L^{\infty} (D)$, $a_i \in C_{\#}(Y_i)$ for every
$i\in\Set{1,\ldots,n}$ and $I$ is the identity matrix of order $d$.
Let us also assume that
$a_0(x) > 1$ for a.e. $x\in D $ and $a_i(y)>1$ for a.e.
$y \in Y_i$ and every $i\in\Set{1,\dots,n}$.

Let us consider the following modification of Assumption~\ref{As:Coeff-i}.
\begin{assumption}[on a coefficient $A_i$ with $i\in\Set{0,\ldots,n}$
  microscales, with positive constants $\gamma$ and $\Gamma$]\label{As:Coeff-Sep}
	The coefficient $A_i$ is of the form
	$
		A_i
		=
		\,
		\Par{a_0 \otimes a_1 \otimes \cdots \otimes a_{i-1} \otimes a_i } \, K_i
	$,
	where $K_i$ is a symmetric positive-definite matrix with spectrum in $\IntCC{\gamma}{\Gamma}$.
\end{assumption}

Note that under the conditions imposed on $A$ in this section,
Assumption~\ref{As:Coeff-Sep} holds for $i=n$ with $K_n=K$.

For every $i\in\Set{1,\ldots,n}$, under Assumption~\ref{As:Coeff-Sep},
the problem~\eqref{eq:bi-fi} becomes
\begin{equation}\label{eq:bi-fi-sep}
	\begin{aligned}
		\mathsf{b}_i \Par{x,y_1,\ldots,y_{i-1},\psi,\phi}
		&=
		a_0(x) \QQ a_1(y_1)\cdots a_{i-1}(y_{i-1})
		\int_{Y_i}
		a_i
		\,
		\Par{ \nabla \psi}^\MT
		K_i
		\,
		\nabla \phi
		\, ,
		\\
		\mathsf{f}_i \Par{x,y_1,\ldots,y_{i-1},\xi,\phi}
		&=
		-
		a_0(x) \QQ a_1(y_1)\cdots a_{i-1}(y_{i-1})
		\int_{Y_i}
		\;
		a_i
		\,
		\xi^\MT
		K_i
		\,
		\nabla \QQ \phi
	\end{aligned}
\end{equation}
for all $\psi,\phi\in H^1_{\#}\Par{Y_i}/_\R$
and for a.e.
$\Tuple{x,y_1,\ldots,y_{i-1}}\in D \times \bm{Y}_{ i-1}$.
As a result, the solution of~\eqref{Eq:wi} is independent of $x\in D$
and $\bm{y}_{i-1} \in \bm{Y}_{i-1}$,
so that $w_i \in \Par{H^1_{\#}\Par{Y_i}/_\R}^d$.
Then the upscaled coefficient $A_{i-1}$, given by~\eqref{Eq:CoeffUpscaling},
satisfies Assumption~\ref{As:Coeff-Sep} with $i-1$ microscales and
\begin{equation}\nonumber
	K_{i-1}
	=
		\int_{Y_i}
		a_i
		\,
		\Par[1]{I + \Jac \QQ  w_i }
		K_i
		\Par[1]{I + \Jac \QQ  w_i }^\MT
	=
		\int_{Y_i}
		a_i
		\,
		\Par[1]{I + \Jac \QQ  w_i }
		K_i
	>
	0
	\, ,
\end{equation}
where the integrand is a function of a single microscale variable
taking values in $Y_i$.

Iterating this argument,
one finds that
Assumption~\ref{As:Coeff-Sep} holds for every $i\in\Set{0,\ldots,n-1}$
with the same constants as for $i=n$
and that
the scale-interaction functions and their gradients,
given by~\eqref{eq:uivec},
are separable.
Consequently, the factors in the right-hand sides
of equalities~\eqref{eq:uivec},
depending on variables corresponding to only a single scale each,
can be approximated independently.
This allows to consider, instead of the spaces $\mathscr{Q}_i^L$
and $\mathscr{S}_{i}^L$ with $i\in\Set{1,\ldots,n}$,
specified in~\eqref{eq:ui-approx-Space} and~\eqref{eq:ui-approx-SpaceAvgX},
spaces of separable tensors,
and to thereby avoid the dependence on $n$ of the exponent in the rank bounds for
$\varPsi_n^L u_n^L,\ldots,\varPsi_n^L u_n^L$ and
$\bar{\varPsi}_0^{\lambda_n+L} \, \mathcal{U}^\varepsilon \QQ v_n^L$
in Theorems~\ref{thm:limit-problem} and~\ref{thm:limit-problem-avg}.

\subsubsection{The case of two dimensions: approximation of functions with corner singularities}\label{sec:RegAnSingFct}
In the remainder of this section, we consider the case of $d=2$.
We will use spaces of functions
defined on a polygonal domain that are analytic on the closure of
the domain except
a number of points
where algebraic singularities of certain order may occur.

With any set
$\varTheta$
of a finite number of distinct points in $\R^2$,
we associate the weight function
$\chi_{\varTheta}$
given by
\begin{equation}\label{Eq:wAlg}
\chi_{\varTheta}(x)
=
\prod_{\theta\in\varTheta}
\|x - \theta \|_{2}
\quad\text{for all}\quad
x\in \R^2
\, ,
\end{equation}
where $\| \cdot \|_{2}$ denotes the Euclidean norm on $\R^2$.

To quantify the analytic regularity of solutions to the high-dimensional one-scale
problem,
we use weighted Sobolev spaces and associated countably normed classes
as introduced in~\cite{Kondratjev:1967:BVP,Kondratjev:1963:BVPconical,Babuska:1988:RegularityI,Babuska:1988:hpCurvedBoundary,Babuska:1989:RegularityII}
and denoted here by
$H^{m,\ell}_{\varTheta, \beta}(\varOmega)$
and
$ C^\ell_{\varTheta, \beta}(\varOmega)$
with $\ell\in\Set{1,2}$, $m\in\Set{0,1,\ldots,\ell}$ and $\beta\in\IntCO{0}{1}$,
where $\varOmega\subset\R^2$ is a polygonal domain and
$\varTheta$
is a set
of $S\in\N$ distinct points in $\overline{\varOmega}$.

Specifically, we will use the following weighted Sobolev spaces:
\begin{equation}\nonumber
	H^{m,0}_{\varTheta, \beta}\Par{\varOmega}
	=
	\Set[1]{
	u\!: \varOmega\rightarrow \R: \; \chi_{\varTheta}^{\beta+\IndNorm{\alpha}} \, \partial^{\alpha} u \in L^2\Par{\varOmega}
	\quad\text{if}\quad
	0 \leq \IndNorm{\alpha} \leq m
	}
\end{equation}
for all $\ell \geq 0$
and
\begin{equation}\nonumber
	H^{m,\ell}_{\varTheta, \beta} \Par{\varOmega}
	=
	\Set[1]{
	u \in H^{\ell-1}\Par{\varOmega}: \; \chi_{\varTheta}^{\beta+\IndNorm{\alpha}-\ell} \, \partial^{\alpha} u \in L^2\Par{\varOmega}
	\quad\text{if}\quad
	0 \leq \IndNorm{\alpha} \leq m
	}
\end{equation}
for all $m \geq \ell \geq 1$, where the differentiation is understood in the weak sense.
By setting
\begin{equation}\label{Eq:HklbetaSeminorm}
\Seminorm{u}_{H^{m,\ell}_{\varTheta, \beta}\Par{\varOmega}}^{2}
=
\sum_{\IndNorm{\alpha}=m}
\Norm{ \chi_{\varTheta}^{\beta+m-\ell} \, \partial^\alpha u }_{L^2\Par{\varOmega}}^{2}
\quad\text{for all}\quad
u\in H^{m,\ell}_{\varTheta, \beta}\Par{\varOmega}
,
\end{equation}
we introduce $\Seminorm{\cdot}_{H^{m,\ell}_{\varTheta, \beta}\Par{\varOmega}}$,
a seminorm on $H^{m,\ell}_{\varTheta, \beta}\Par{\varOmega}$.
Also, by setting
\begin{equation}\nonumber
\begin{split}
\Norm{ u }^2_{H^{m,0}_{\varTheta, \beta}\Par{\varOmega}}
&=
\phantom{\Norm{ u }^2_{H^{\ell-1}\Par{\varOmega}}+}
\sum_{k=0}^m \Seminorm{ u }^2_{H^{m,0}_{\varTheta, \beta}\Par{\varOmega}}
\quad\text{for}\quad
\quad
u\in H^{m,0}_{\varTheta, \beta}\Par{\varOmega}
,
\quad
m \geq 0
,
\\
\Norm{ u }^2_{H^{m,\ell}_{\varTheta, \beta}\Par{\varOmega}}
&=
\Norm{ u }^2_{H^{\ell-1}\Par{\varOmega}}
+
\sum_{k=\ell}^m \Seminorm{ u }^2_{H^{k,\ell}_{\varTheta, \beta}\Par{\varOmega}}
\quad\text{for}\quad
\quad
u\in H^{m,\ell}_{\varTheta, \beta}\Par{\varOmega}
,
\quad
m \geq \ell \geq 1
,
\end{split}
\end{equation}
we define $\Norm{ \cdot }^2_{H^{m,\ell}_{\varTheta, \beta}\Par{\varOmega}}$,
a norm on $H^{m,\ell}_{\varTheta, \beta}\Par{\varOmega}$ for any $\ell,m\in\Nz$ such that $m \geq \ell$.

\begin{definition}[%
	analyticity of a function with point algebraic singularities,
	with positive constants $M$ and $\rho$%
  ]\label{Df:Clbeta}
Let $\beta\in\IntCO{0}{1}$, $\varOmega\subset\R^2$ be a polygonal domain,
$\varTheta$
be a finite set of distinct points in $\overline{\varOmega}$
and
$\ell\in\Set{1,2}$.
Then
$u\in  C^\ell_{\varTheta, \beta}(\varOmega)$
if $u\in H^{\ell,\ell}_{\varTheta, \beta}(\varOmega)$
and
there exist positive constants $M$ and $\rho$ such that,
for all $\alpha\in\N_0^2$ with $|\alpha| \geq \ell-1$,
\begin{equation}\nonumber
	\sup_{x \in \varOmega}
	\;
	\chi_{\varTheta}^{\beta + |\alpha| - \ell +1} (x) \,
	\big|\partial^{\alpha} u (x) \big|
	\leq
	M \rho^{|\alpha|} \QQ
	|\alpha|!
	\; .
\end{equation}
\end{definition}

The following result is a consequence
of~\cite[Theorems~3.4--3.5]{Babuska:1988:hpCurvedBoundary}
for the iterated-homogenization scheme of~\ref{Eq:wiw}--\ref{eq:uivec2grad}.
\begin{proposition}\label{Pr:u0-C2beta-anlyticity}
	Assume that $\varTheta$ is the set of vertices of the unit square $D=(0,1)^2$.
	Let Assumptions~\ref{As:Coeff} and~\ref{ass:Analytic1} hold.
	Then the solution $u_0$ of the homogenized problem~\eqref{eq:Phom}
	satisfies $u_0 \in C^2_{\varTheta, \beta}(D)$
	with some $\beta \in [0,1)$.
\end{proposition}

Note that this statement remains valid for an arbitrary curvilinear polygon
with an analytic boundary~\cite{Babuska:1988:hpCurvedBoundary}.
However, even in the setting of Assumption~\ref{ass:Analytic1},
the exponent $\beta$ depends on the transformation diagonalizing the diffusion
coefficient at the vertices of $D$ and can be estimated in terms of
the spectral bounds $\gamma$ and $\Gamma$.

We will now combine
the weighted-analyticity statement of Proposition~\ref{Pr:u0-C2beta-anlyticity}
with rank bounds for the QTT-FE approximation
of functions from
$C^2_{\varTheta, \beta}(D)$ in~\cite{Kazeev:PhD,KS:2017:QTTFE2d}.
\begin{theorem}\label{Th:QTT-FE-exp-C2beta}
Assume that $\beta \in [0,1)$ and
$\varTheta$
is the set of the vertices of $D$.
Let $u_0 \in  V \cap C^2_{\varTheta, \beta}(D)$.
Then the following holds with positive constants $C$ and $c$.

For every $L\in\N$,
there exist $u_0^L \in V^L$ and $v_0^L \in (\bar{V}^L)^d$
satisfying Assumption~\ref{ass:u0-low-rank} with $r_L = \Ceil{cL^2}$
and such that
$
	\Norm {u_0 - u_0^L }_{H^1(D)}
	\, , \,
	\Norm{ \nabla u_0 - v_0^L}_{L^2(D)^d}
	\leq
	C L^3 \,
	2^{-(1-\beta) L}
$.

\end{theorem}
\begin{proof}
	The statement regarding $u_0^L$ with $L\in\N$ follows immediately
	from either~\cite[Theorem~5.3.7]{Kazeev:PhD}
	or~\cite[Theorem~5.16]{KS:2017:QTTFE2d}.
	In particular,
	for all $L\in\N$ and $\ell\in\Set{0,1,\ldots,L}$,
	there exist subspaces
	$\hat{\mathscr{L}}^L_\ell\subset\R^{2^{d\ell}}$ and
	$\hat{\mathscr{M}}^L_\ell\subset\R^{2^{d(L-\ell)}}$, both
	of dimension at most $r_L = \Ceil{cL^2}$,
	where $c$ is a positive constant independent of $L$, such that
	$\varPsi_0^L u_0^L \in \hat{\mathscr{L}}^L_\ell \otimes \hat{\mathscr{M}}^L_\ell$.

	To obtain the statement regarding $v_0^L$ with $L\in\N$,
	we consider
	$
		v_{0,k}^L
		=
		\CuBr[1]{
			\Par{ \bar{\pi}^L }^{\otimes (k-1)}
			\otimes
			\mathsf{id}
			\otimes
			\Par{ \bar{\pi}^L }^{\otimes (d-k)}
		}
		\,
		\partial_k
		u_0
		\in \bar{V}^L
	$
	for all $k\in\Set{1,\ldots,d}$ and $L\in\N$.
	Bounds analogous to those of Proposition~\ref{Pr:Appr01}
	yield
	$
		\Norm{ \nabla u_0^L - v_0^L}_{L^2(D)^d}
		\lesssim
		2^{-L}
		\Seminorm{u_0^L}_{H^1(D)}
	$
	with an equivalence constant independent of $L\in\N$.
	Then the triangle inequality gives the error bound claimed for
	$v_0^L$.
	Further, the action of the operators $\bar{\pi}^L$ and $\partial_k$,
	up to scaling,
	consists in adding to and subtracting from the coefficient tensor
	its single-position
	shift along the respective dimension,
	which preserves the piecewise-polynomial structure used to establish
	rank bounds in~\cite[Lemma~4.6.1 and Corollary~4.6.2]{Kazeev:PhD}
	and in~\cite[Lemma~5.13 and Corollary~5.14]{KS:2017:QTTFE2d}.
	Inspecting those proofs, one concludes that
	the rank analysis given there applies verbatim to
	$v_{0,k}^L$ with $k\in\Set{1,\ldots,d}$ and $L\in\N$:
	for every $\ell\in\Set{0,1,\ldots,L}$, we have
	$\bar\varPsi_0^L v_{0,k}^L \in \hat{\mathscr{L}}^L_\ell \otimes \hat{\mathscr{M}}^L_\ell$,
	where the subspaces $\hat{\mathscr{L}}^L_\ell$ and $\hat{\mathscr{M}}^L_\ell$
	are identical to those constructed in the same proofs for $\varPsi_0^L u_0^L$.
	This shows that $u_0^L$ and $v_0^L$
	satisfy Assumption~\ref{ass:u0-low-rank} with $r_L = \Ceil{cL^2}$.
\end{proof}

The following is a corollary of
Theorems~\ref{thm:limit-problem}, \ref{thm:limit-problem-avg}
and~\ref{Th:QTT-FE-exp-C2beta}
and Proposition~\ref{Pr:u0-C2beta-anlyticity}.
\begin{corollary}\label{cor:2d}
	Assume that $D=(0,1)^2$.
	Let Assumptions~\ref{As:Coeff}, \ref{ass:Analytic1} and~\ref{ass:tech} hold
	and $\Tuple{u_0,u_1,\ldots,u_n} \in \bm{V}_n$ be the solution of~\eqref{limeqn}
	and $v_n = \nabla u_0 + \nabla_1 u_1 + \cdots + \nabla_n u_n$.
	Then the approximations
	$u_i^L \in V_i^L$ with $L\in\N$ and $i\in\Set{1,\ldots,n}$
	and
	$\mathcal{U}^\varepsilon \QQ  v_n^L \in (\bar{V}^{\lambda_n+L})^d$
	with $L\in\N$,
	defined by~\eqref{eq:uvi-approx} and~\eqref{Eq:wi-pol-approx},
	satisfy the following
	with $\beta\in\IntCO{0}{1}$ and
	with positive constants $\tilde{C}$ and $\tilde{c}$.

	For all $L\in\N$ sufficiently large (to satisfy Assumption~\ref{ass:tech}),
	the error bound
	$
		\sum_{i=0}^n \Norm{u_i - u_i^L}_{V}
		+
		\Norm{
			\mathcal{U}^\varepsilon \QQ v_n - \mathcal{U}^\varepsilon \QQ  v_n^L
		}_{L^2(D)^d}
		\leq
		\tilde{C}
		\,
		L^3
		\,
		2^{-(1-\beta) L}
	$
	holds and each of the coefficient tensors
	$\varPsi_i^L u_i^L$ with $i\in\Set{0,\ldots,n}$
	and
	$\bar{\varPsi}_0^{\lambda_n+L} \, \mathcal{U}^\varepsilon \QQ v_n^L$
	admit decompositions of the form~\eqref{eq:TT-MPS}
	with ranks bounded from above by
	$R_L = \tilde{c} \QQ L^{2(n+2)}$.
\end{corollary}

\section{Numerical results}
\label{sc:NumReS}

We implement two approaches for the approximate numerical solution of
the multiscale problem~\eqref{Eq:ProblemEps}.

The first approach consists in immediately solving
the discretization~\eqref{Eq:ProblemEpsDiscr}
of the multiscale problem~\eqref{Eq:ProblemEps},
seeking the solution in the form of
the multilevel TT-MPS decomposition~\eqref{eq:TT-MPS}.
The implementation is based on the recent result~\cite{BK:2020:StabPrec}
on the preconditioning of elliptic second-order operators,
which allows to avoid the ill-conditioning and numerical instability
associated with the use of fine discretizations (large $L$)
and with using the multilinear decomposition~\eqref{eq:TT-MPS}
instead of storing all the entries of the coefficient tensor independently.
We modified the BPX preconditioner developed in~\cite{BK:2020:StabPrec},
following the original derivations¸
so as to accommodate the Dirichlet boundary conditions of~\eqref{Eq:ProblemEps},
imposed on the whole of the boundary.

The second approach consists in solving
the high-dimensional one-scale limit problem~\eqref{limeqn}
in a form analogous to~\eqref{eq:TT-MPS}
and computing $u_0^L$ and $\mathcal{U}^\varepsilon \QQ  v_n^L$
that approximate
$u^\varepsilon$ and $\nabla u^\varepsilon$
in the respective $L^2$ norms.

We emphasize that the first approach bypasses the limit problem~\eqref{limeqn}
and aims at solving directly discretizations of the
multiscale problem~\eqref{Eq:ProblemEps}.
The second approach, on the contrary, explicitly involves the limit problem
as an auxiliary computational problem.
Neither approach requires the computation of effective (or ``homogenized'',
``upscaled'') coefficients.
We cover the second approach only in some of the experiments,
for reference and comparison.
We did not incorporate the BPX preconditioner
developed in~\cite{BK:2020:StabPrec} in the second approach,
so it can be used only for relatively coarse virtual grids (up to $L=15$ when $d=1$).
The source code
developed for our numerical experiments
is publicly available\footnote{\url{https://bitbucket.org/rakhuba/msqtt2d_numexp}}.

\subsection{Two scale problem, $n=1$, $d=1$}
\label{sc:2Scale}
We start with an instance of the problem~\eqref{eq:Aeps}--\eqref{Eq:ProblemEps}
with two scales, $D = Y = (0,1)$, $d=1$
and
 \begin{equation}\label{eq:1d_knownsol}
 	{d\over dx}
 	\left(A^{\varepsilon}(x){d u^\varepsilon\over dx}\right)=1
   \quad \mbox{in}\quad D, \qquad u^\varepsilon(0)=u^\varepsilon(1)=0\;,
 \end{equation}
 where 
\begin{equation}\label{eq:coef_1d}
	A\Par{x, y} = {2\over 3}(1+x)\Par[1]{1+\cos^2\Par{2\pi y}}
	\quad\text{for all}\quad
	x \in D
	\quad\text{and}\quad
	y\in Y
	\, ,
\end{equation}
see \cite[Section 6.1]{VHHCS2004}, where this problem was solved with a sparse-grid FEM approach.
The two-scale limiting equation has the exact (homogenized) solution
$u_0$ given by
\begin{equation}\label{eq:num_u0}
u_0(x)={3\over 2\sqrt{2}}\left(x-{\log(1+x)\over\log 2}\right)
\end{equation}
for all $x\in D$ and $y\in Y$
and the scale interaction term $u_1$ is given by
\begin{equation}\label{eq:num_u1}
	u_1(x,y)
	=
	{3\over 2\sqrt{2}}
	\left(1-{1\over(1+x)\log 2}\right)\left({1\over 2\pi}\tan^{-1}\left({\tan2\pi y\over\sqrt{2}}\right)-y+\phi(y) + C\right)
\end{equation}
for all $x\in D$ and $y\in Y$,
where $\phi$ is chosen to the ensure continuity of $u_1$:
\begin{equation}\label{eq:phifun}
	\phi(y) =
	\begin{cases}
		0, & y\in [0, 1/4] \\
		\frac 12, & y\in (1/4, 3/4] \\
		1, & y\in (3/4, 1].
	\end{cases}
\end{equation}

We consider two approaches to approximate solution of the problem~\eqref{eq:Aeps}--\eqref{Eq:ProblemEps}
with \eqref{eq:coef_1d}: QTT-FEM discretization of
the multiscale problem~\eqref{eq:Aeps}--\eqref{Eq:ProblemEps} with \eqref{eq:coef_1d}
and the QTT-FEM discretization of the corresponding high-dimensional limit problem~\eqref{limeqn}.
For the first one we introduce nested grids with $2^{\ell} - 1$ interior points
and the corresponding FE discretization using piecewise-linear hat functions.
For every $\ell$ the Galerkin solution is parametrized by
a $2^\ell$-component vector $\mathbf{u}^{\varepsilon,\ell}$,
including zero coefficient of the basis function at corresponding to node $1$.

The multidimensional limiting one-scale problem is discretized using tensor product basis functions with $2^\ell$ basis functions both for the physical variable $x$ and for fast variables $y_i$. 
This discretization produces
coefficient tensors
$\mathbf{u}^{\ell}_i \in \R^{2^{(i+1)\ell}}$ with $i\in\Set{0,1,\dots,n}$.

The goal is to find QTT approximations $\mathbf{u}^{\varepsilon,\ell}_\textrm{qtt}$ and $\mathbf{u}^{\ell}_{i,\textrm{qtt}}$
with $i=0,\ldots,n$
of
$\mathbf{u}^{\varepsilon,\ell}$ of the
multiscale problem
and $\mathbf{u}^{\ell}_i$
with $i=0,\ldots,n$
of the one-scale limit problem
respectively.
We denote the $\Seminorm{\cdot}_{H^1(D)}$ error corresponding to ${u}^{\varepsilon,\ell}_\textrm{qtt}$
as follows:
$
	\delta_\ell^\textrm{exact} = \Seminorm[1]{ {u}^{\varepsilon,\ell}_\textrm{qtt} - u^\varepsilon}_{H^1(D)}
$.
Since the exact solution $u^\varepsilon$ is not
available, we use instead the extrapolated solution
\begin{equation}\label{eq:extrap}
	u_\textrm{ext}^\varepsilon = 2 {u}^{\varepsilon,L}_\textrm{qtt} - {u}^{\varepsilon,L-1}_\textrm{qtt}
\end{equation}
with $L=50$.
In numerical experiments we therefore measure the following error:
$
	\Seminorm[1]{ {u}^{\varepsilon,\ell}_\textrm{qtt} - u^\varepsilon_\textrm{ext} }_{H^1(D)}
	\approx
	\Seminorm[1]{ {u}^{\varepsilon,\ell}_\textrm{qtt} - u^\varepsilon }_{H^1(D)}
$.

As for the one-scale limit problem corresponding to
the problem~\eqref{eq:Aeps}--\eqref{Eq:ProblemEps} with \eqref{eq:coef_1d},
we have its exact solution $(u_0,u_1)$ available through~\eqref{eq:num_u0}
and~\eqref{eq:num_u1}. So errors can be exactly computed as
$
	\tilde\delta_\ell = |||u_0 - u_0^\ell, \{u_i - u_i^\ell\}|||
$,
where
$
	|||u_0, \{u_i\}|||
	=
	\Norm[1]{\nabla u_0}_{L^{2}(D)}
	+
	\sum_{i=1}^n \Norm{\nabla_i u_i}_{L^{2}(D \times \bm{Y}_{1} \times \dots \times \bm{Y}_{i})}
$.

To find QTT approximations $\mathbf{u}^{\varepsilon,\ell}_\textrm{qtt}$
and $\mathbf{u}^{\ell}_{i,\textrm{qtt}}$ with $i\in\Set{0,1,\ldots,n}$,
we take the two approaches described in the beginning of Section~\ref{sc:NumReS}.
Figures~\ref{fig:conv_multscale} and~\ref{fig:conv_onescale} illustrate
convergence with respect to the virtual grid level $l$ for each of them.
In the both cases as anticipated we observe first order convergence.

\begin{figure}
\begin{subfigure}{.5\textwidth}
  \centering
  \includegraphics[width=\linewidth]{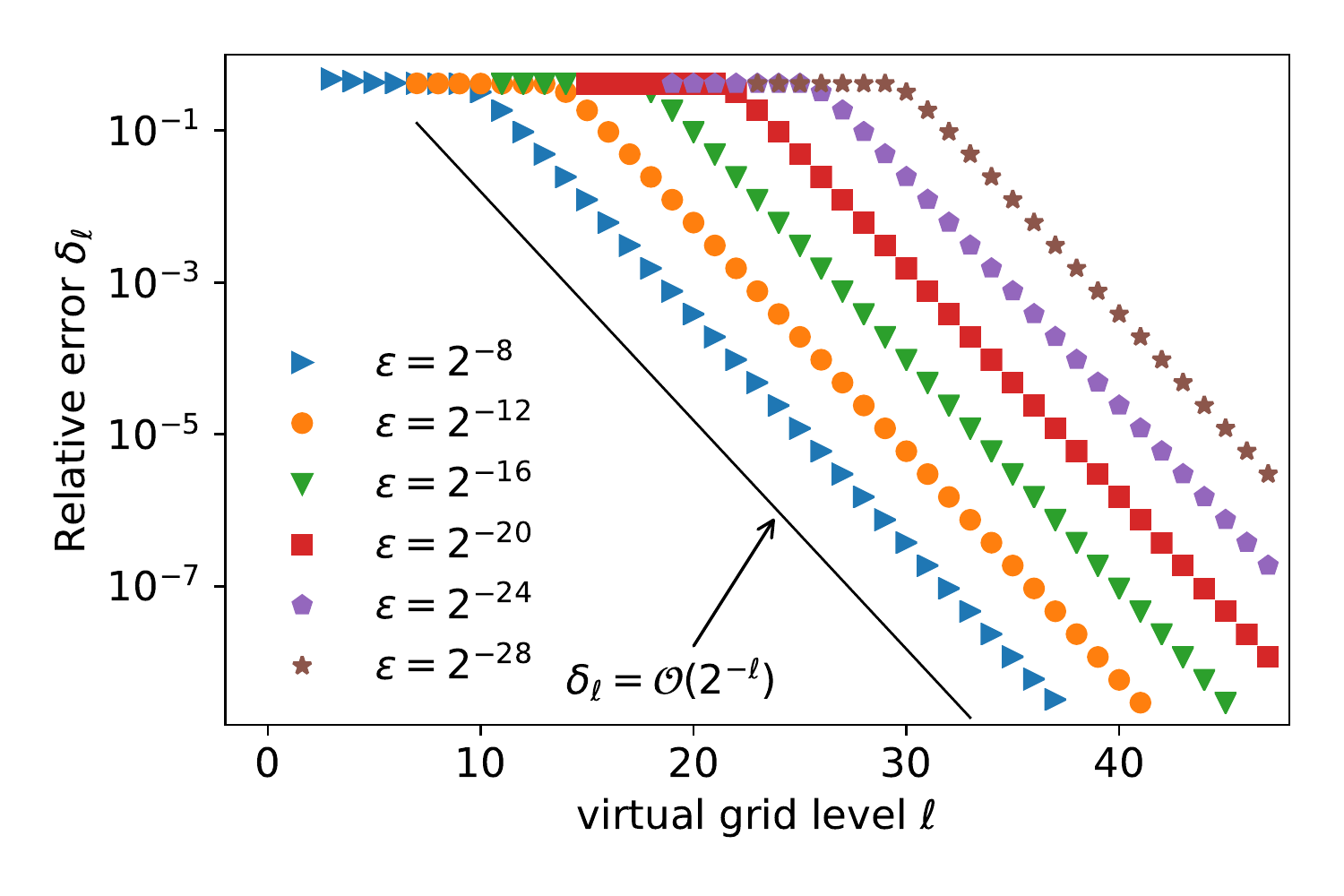}
  \caption{}
  \label{fig:conv_multscale}
\end{subfigure}%
\begin{subfigure}{.5\textwidth}
  \centering
  \includegraphics[width=\linewidth]{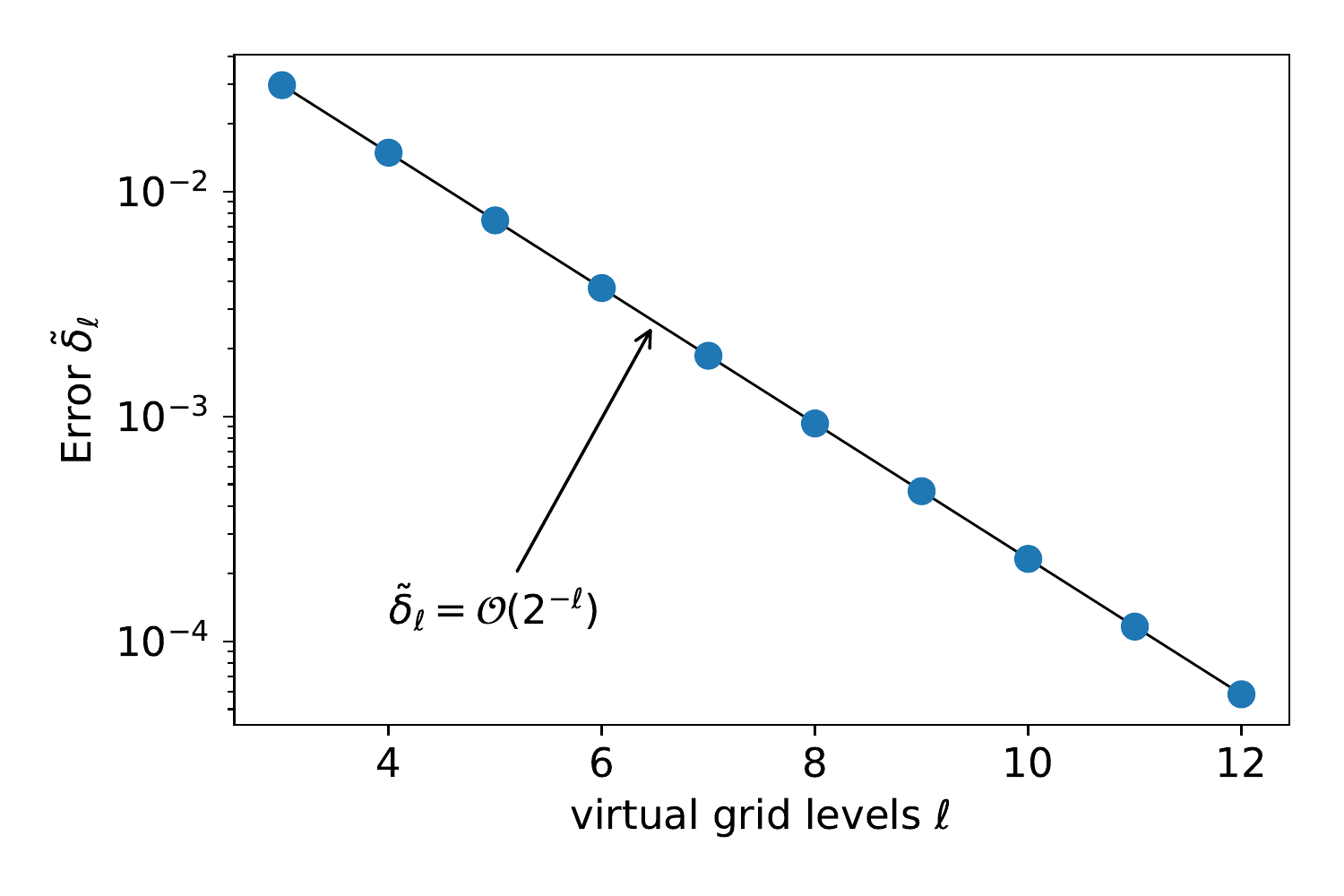}
  \caption{}
  \label{fig:conv_onescale}
\end{subfigure}
\caption{
  Error $\delta_\ell$, defined in~\eqref{eq:ErrDelell},
  w.r.t. the number of virtual grid levels for QTT-FEM for
  (a) the instance~\eqref{eq:1d_knownsol}
  of the multiscale problem~\eqref{Eq:ProblemEps}
  (with different values of the scale parameter $\varepsilon$)
  and
  (b) the one-scale limit problem~\eqref{limeqn}
  corresponding to the problem~\eqref{eq:Aeps}--\eqref{Eq:ProblemEps} with \eqref{eq:coef_1d}.
  Reference lines represent first-order convergence
  w.r.t. the meshwidth $h_\ell = 2^{-\ell}$.
}
\label{fig:QTTvsQT3}
\end{figure}

Next we investigate the QTT rank dependence of
$\mathbf{u}^{\varepsilon,\ell}_\textrm{qtt}$ of the QTT-FE solution of
the multiscale problem~\eqref{eq:Aeps}--\eqref{Eq:ProblemEps} with \eqref{eq:coef_1d}.
To this end, we first approximate $\mathbf{u}^{\varepsilon,\ell}$
by calculating $\mathbf{u}^{\varepsilon,\ell}_\textrm{qtt}$
with $10^{-12}$ tolerance of QTT arithmetic and \texttt{amen\_solve},
which is utilized to solve arising linear systems.
Then we calculate the error $\delta_\ell$
\begin{equation}\label{eq:ErrDelell}
	\delta_\ell = \left|{u}^{\varepsilon,\ell} - u^\varepsilon_\textrm{ext}\right|_{H^1(D)}.
\end{equation}
Finally, we calculate a sequence of truncated representations
$\texttt{round}(\mathrm{u}^{\varepsilon,\ell}, \texttt{tol})$ for different tolerance values $\texttt{tol}$.
We introduce notation
$\mathbf{u}^{\varepsilon,\ell}_\textrm{qtt}[\tau_\ell]
= \texttt{round}(\mathrm{u}^{\varepsilon,\ell}, \tau_\ell)$.
The goal is to find the largest $\tau_\ell$ so that the following inequality holds:
\begin{equation}\label{eq:tau}
	\left|{u}^{\varepsilon,\ell} - u^\varepsilon_\textrm{ext}\right|_{H^1(D)}
 \leq 2 \left|{u}^{\varepsilon,\ell}_\textrm{qtt}[\tau_\ell] - u^\varepsilon_\textrm{ext}\right|_{H^1(D)},
\end{equation}
where ${u}^{\varepsilon,\ell}_\textrm{qtt}[\tau_\ell]$ is the FE interpolant:
\[
{u}^{\varepsilon,\ell}_\textrm{qtt}[\tau_\ell]
= \sum_{j \in \mathcal{I}^\ell} \mathbf{u}^{\varepsilon,\ell}_\textrm{qtt}[\tau_\ell]_j\, \varphi_j^\ell
\]
Figure~\ref{fig:conv_multscale} presents the dependence of the rank
of $\texttt{round}(\mathrm{u}^{\varepsilon,\ell}, \tau_\ell)$
against the $H^1$ error~$\delta_\ell$:
\begin{equation}\label{eq:delta_l}
\delta_\ell = \left|{u}^{\varepsilon,\ell}_\textrm{qtt}[\tau_\ell] - u^\varepsilon_\textrm{ext}\right|_{H^1(D)}
\end{equation}

Next we investigate the QTT rank dependence of the
QTT-FEM solution $\mathbf{u}^{\ell}_\textrm{lim-qtt}$ of the
high-dimensional, one-scale limiting problem:
$
	u^{\ell}_{\textrm{lim}} = u_0^\ell + \mathcal{U}^\varepsilon u_1^\ell
$.
We set $\varepsilon = 2^{-\ell_\varepsilon}$ thus obtaining solution given by
coefficient tensor $\bold{u}^{\ell}_{\textrm{lim}}$ of length  $2^{\ell + \ell_\varepsilon}$,
which is approximated in QTT format by $\bold{u}^{\ell}_{\textrm{lim-qtt}}$.

For the both cases we observe polylogarithmic scaling of the effective QTT-rank
of both $\bold{u}^\ell_\textrm{qtt}$ and $\bold{u}^{\ell}_\textrm{lim-qtt}$
with respect to the error in $|\cdot|_{H^1(D)}$
or with respect to the truncation parameter $\tilde\delta$: 
\begin{equation} \label{eq:fitting}
	r = \mathcal{O}(\log^\kappa \delta^{-1}).
\end{equation}
In Figures \ref{fig:rank_mult_n1d1} and \ref{fig:rank_nodefect_n1d1} we fit the parameter $\kappa$.
Figure~\ref{fig:rank_mult_n1d1} illustrates that $\kappa$
barely depends on
$\varepsilon$.

\begin{figure}
{
	\centering
	\includegraphics[width=0.7\linewidth]{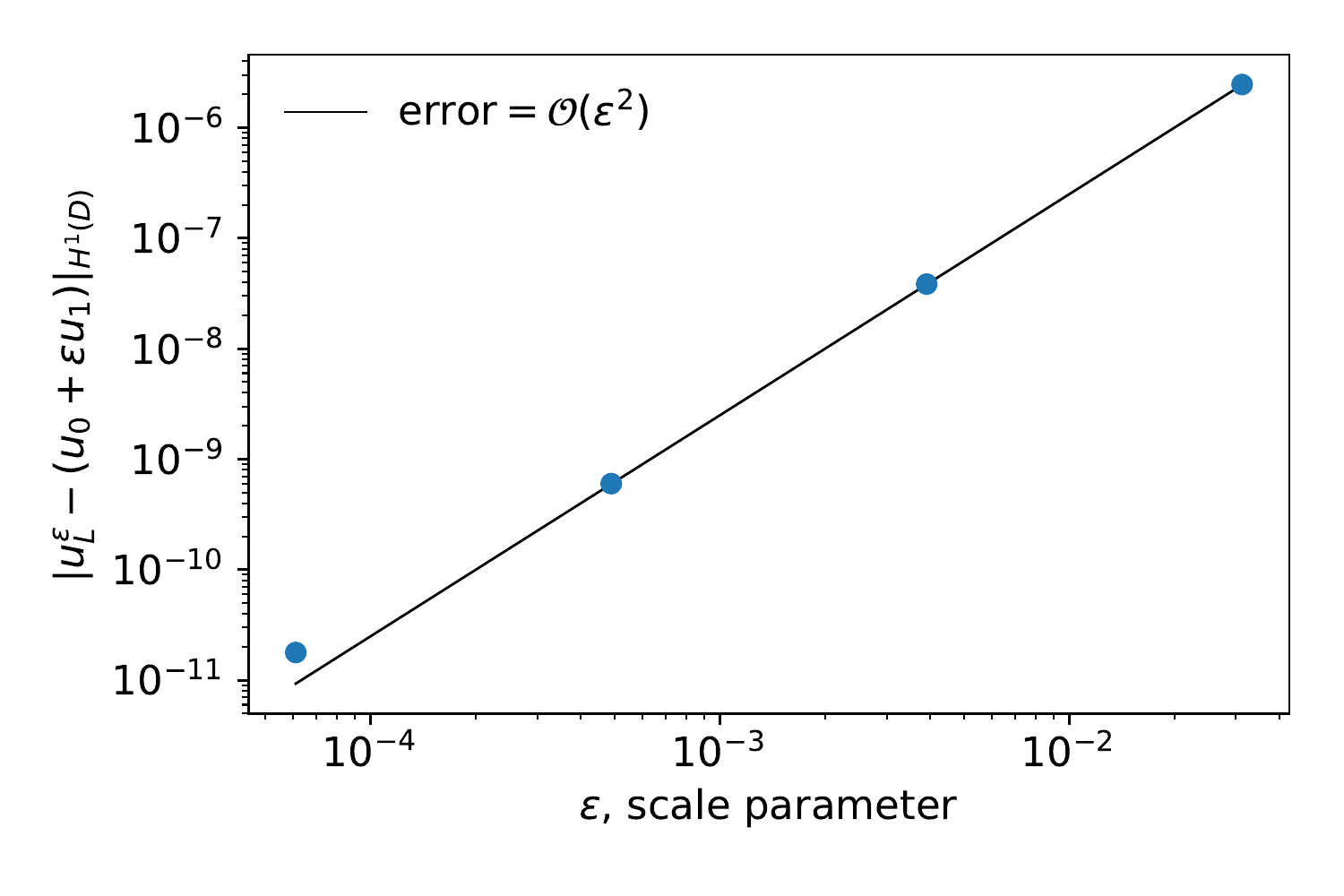}
	\caption{
		$H^1(D)$ error between the FE solution of two-scale limiting problem and FE solution
		$u^\varepsilon_L$, $L=50$ of the physical problem against scale parameter $\varepsilon$.
	}
	\label{fig:1d_homerr}
}
\end{figure}

\begin{figure}
\begin{subfigure}{.5\textwidth}
  \centering
  \includegraphics[width=\linewidth]{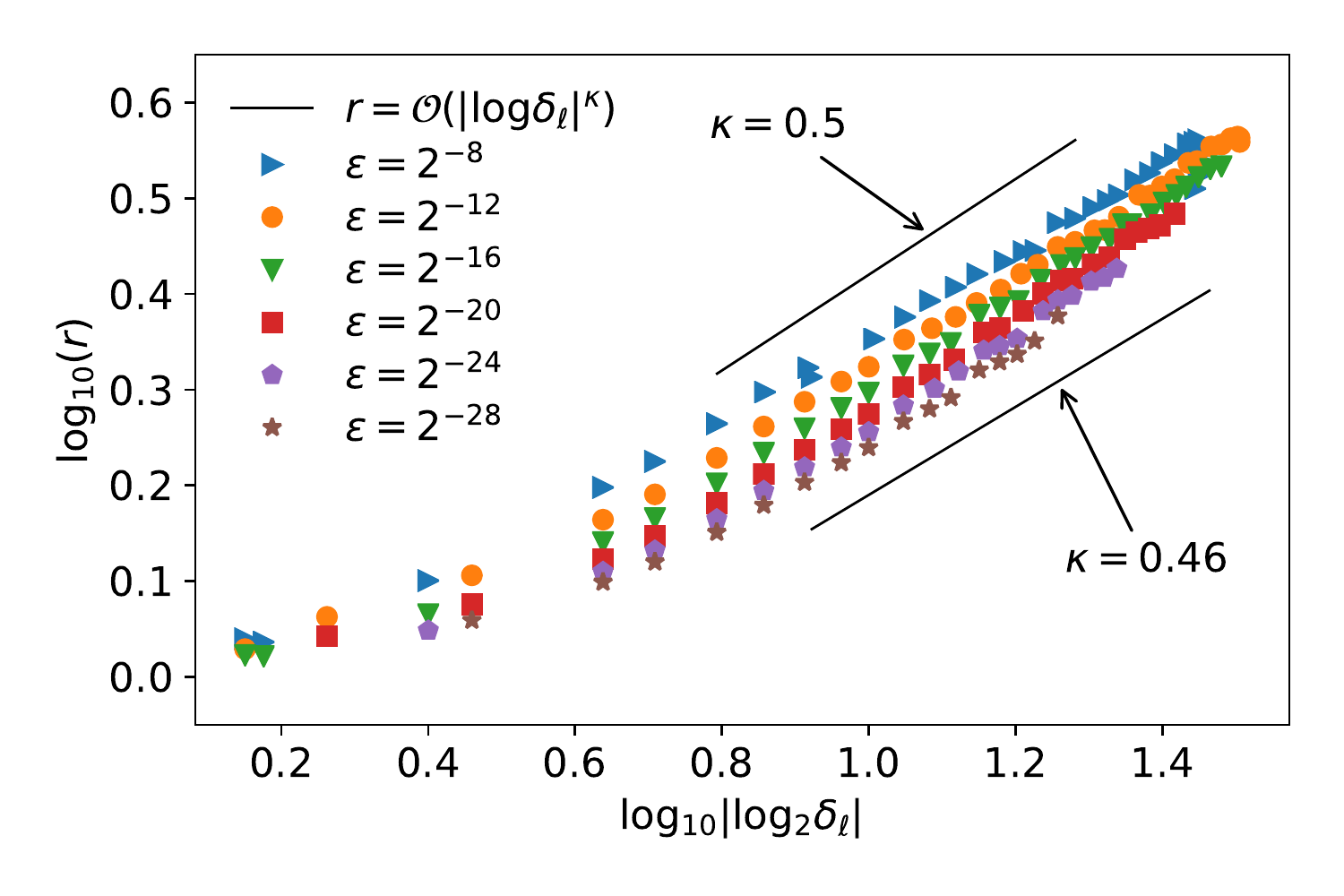}
  \caption{}
  \label{fig:rank_mult_n1d1}
\end{subfigure}%
\begin{subfigure}{.5\textwidth}
  \centering
  \includegraphics[width=\linewidth]{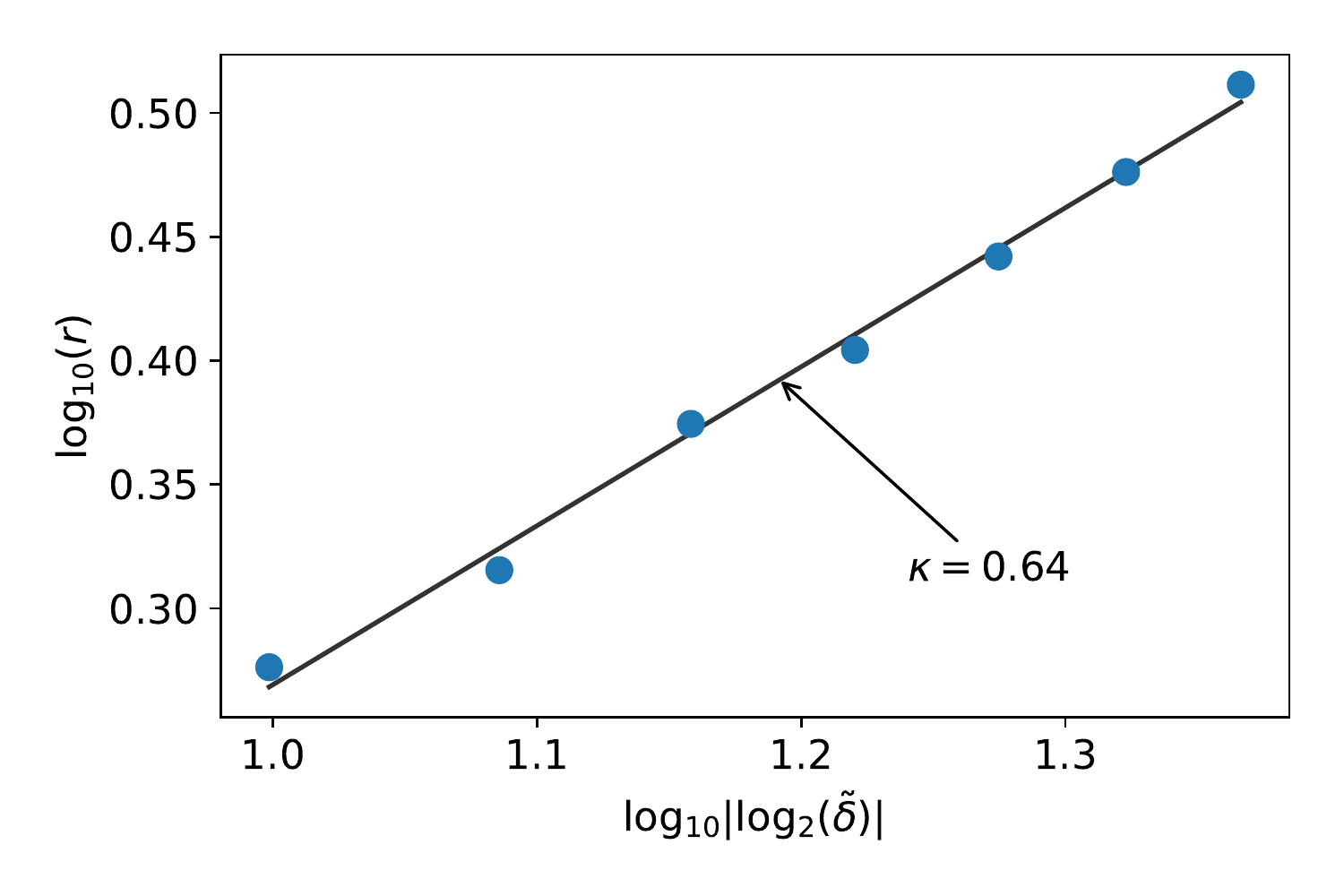}
  \caption{}
  \label{fig:rank_nodefect_n1d1}
\end{subfigure}
\caption{
	Multiscale multiscale problem~\eqref{eq:Aeps}--\eqref{Eq:ProblemEps} with
	the coefficient given by~\eqref{eq:coef_1d}.
	(a) QTT-FEM for the multiscale problem:
	effective rank $r$ vs. $|\cdot|_{H^1(D)}$-error for different $\varepsilon$.
	(b) QTT-FEM for the corresponding one-scale limit problem:
	effective rank $r$ vs. rounding parameter $\tilde\delta$ for $\ell = 10$ and $\ell_\varepsilon=17$.
}
\label{fig:nodefect_n1d1}

\end{figure}

\subsection{$(n+1)$-scale problem}

In this section, we consider the problem~\eqref{eq:Aeps}--\eqref{Eq:ProblemEps}
with $n+1$ scales, $D = Y = (0,1)$ and
\begin{equation}\label{eq:1d_nscale}
	A\Par{x,y_1,\ldots,y_n}
	=
	\Par[3]{ \frac{2}{3}}^n (1+x)
	\prod_{i=1}^n
	\Par[1]{ 1+\cos^2 \Par{2\pi y_i} }
\end{equation}
for all $x\in D$ and $y_1,\ldots,y_n \in Y$.
We discretize the problem using QTT-FEM with number of virtual grid levels $L=50$.
We fix the finest scale parameter to be $\varepsilon_n = 2^{-20}\approx 10^{-6}$
and then select the remaining scale parameters as follows
\[
	\varepsilon_k = 2^{2(n-k)} \varepsilon_n, \quad k=1,\dots, n-1
	\, .
\]
In Figure~\ref{fig:multiscale_rank} the effective rank values (obtained for the fixed truncation threshold $10^{-8}$)
against the number of scales are presented.
In this plot, we observe superlinear growth of the effective rank in the given range of the number of scales.
In absolute values, the effective rank increased approximately from $2.2$ for $n=1$ to $3.8$ for $n=9$, which only leads to a moderate increase of the total amount of work to solve the problem.

\begin{figure}
{
	\centering
	\includegraphics[width=0.7\linewidth]{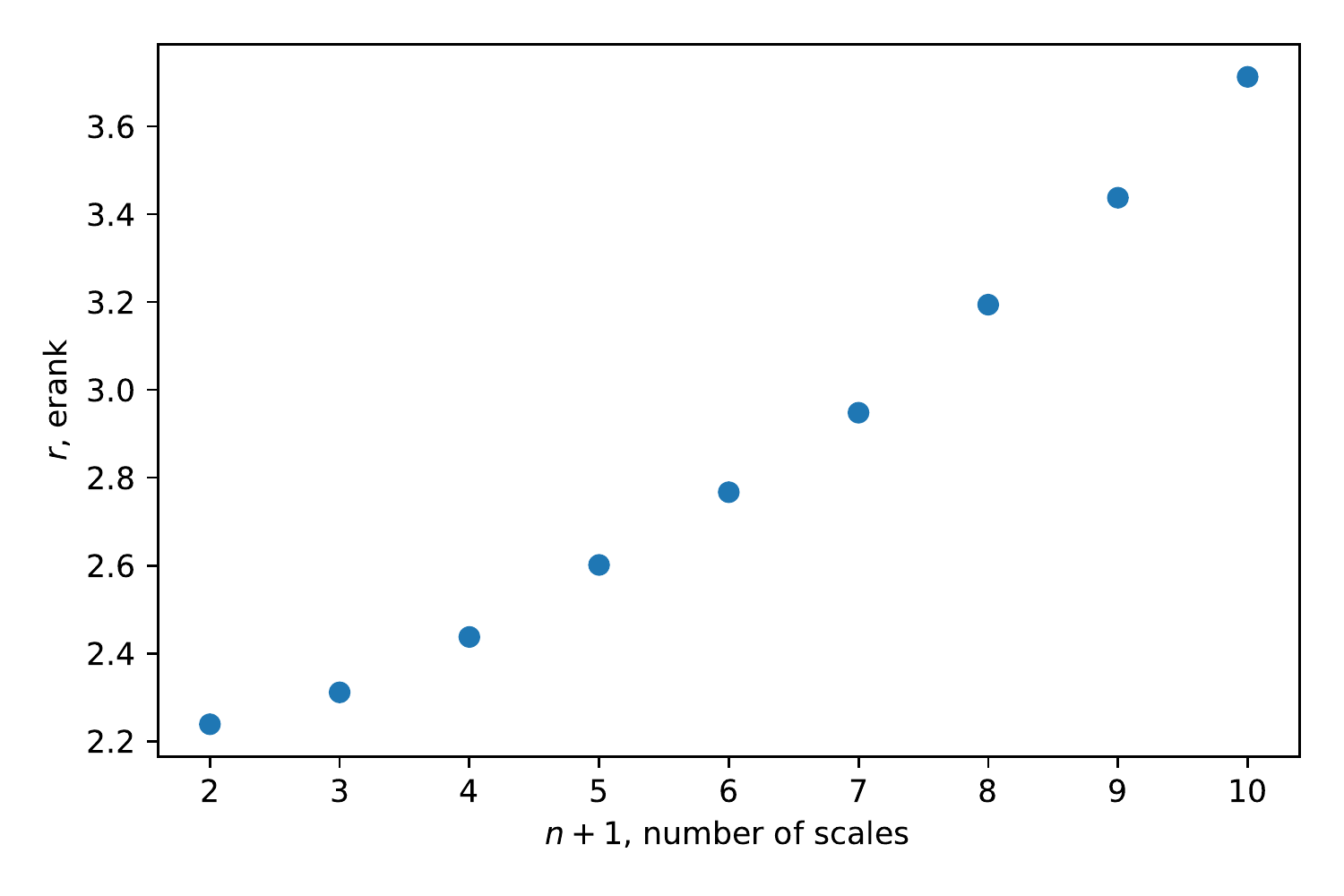}
	\caption{
	QTT-FEM for the multiscale problem~\eqref{eq:Aeps}--\eqref{Eq:ProblemEps} with
	the coefficient given by~\eqref{eq:1d_nscale}:
	effective rank $r$ vs. number $n+1$ of scales.
}
\label{fig:multiscale_rank}
}
\end{figure}

\subsection{Two scale problem in two physical dimensions}
\label{sc:2Sc2d}
In this section, we consider
the problem~\eqref{eq:Aeps}--\eqref{Eq:ProblemEps}
with two scales, $D = Y = (0,1)^2$ and $A = a I$, where $I$
is the identity matrix of order two and
\begin{equation}\label{eq:coef_2d}
	a\Par{x,y}
	=
	\Par[1]{ 1+\cos^2 \Par{2\pi y_1} }
        \Par[1]{ 1+\cos^2 \Par{2\pi y_2} }
\end{equation}
for all $x\in D$ and $y=(y_1,y_2)\in Y$.

Similarly to the one-dimensional case,
we introduce nested tensor-product grids with $(2^{\ell} - 1)^2$
interior points (see Sections~\ref{sc:fesp-intro} and~\ref{subsec:FE}).
On this grid we introduce FE basis functions that are tensor product of one-dimensional piecewise-linear hat functions.
Then for every $\ell$ the Galerkin solution is parametrized by the $2^{2\ell}$-component vector $\mathbf{u}^{\varepsilon,\ell}$.
The error and ranks are measured as described in Section~\ref{sc:2Scale}.
In Figure~\ref{fig:err_mult_2d} we plot the error w.r.t. the extrapolated solution~\eqref{eq:extrap} against virtual grid level $\ell$.
As anticipated we observe
first-order convergence with respect to the meshwidth $h_\ell = 2^{-\ell}$.

Figure~\ref{fig:rank_mult_2d} presents effective numerical rank of
$\texttt{round}(\mathrm{u}^{\varepsilon,\ell}, \tau_\ell)$ with $\tau_\ell$
being the smallest positive value satisfying~\eqref{eq:tau}.
We fit the effective numerical rank versus $\delta_\ell$
defined in~\eqref{eq:delta_l} using $r = \mathcal{O}(|\log \delta|^\kappa)$.
As for the case with one physical dimension,
the fitted values of $\kappa$ hardly depend
on the scale parameter $\varepsilon$.
\begin{figure}
\begin{subfigure}{.5\textwidth}
  \centering
  \includegraphics[width=\linewidth]{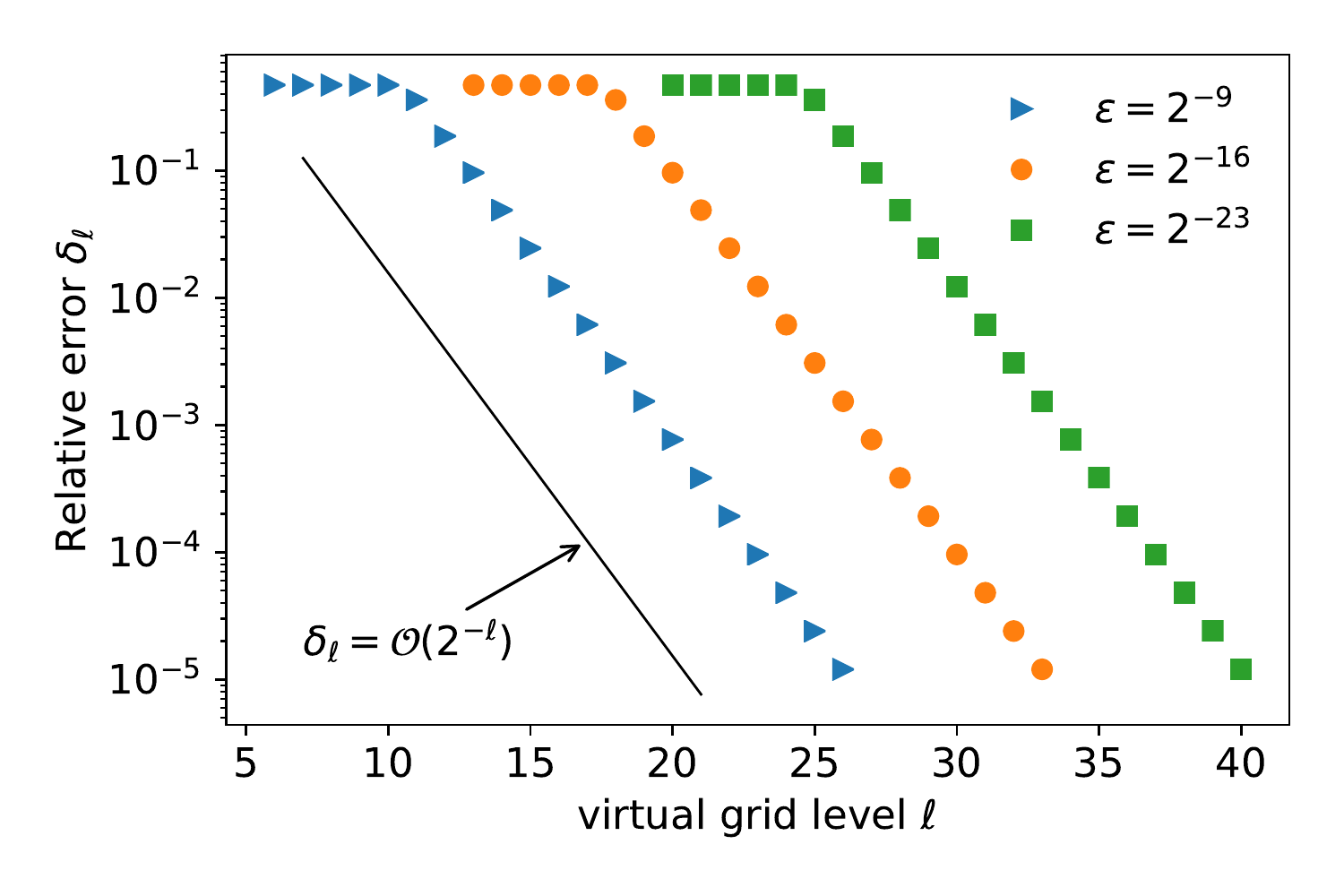}
  \caption{}
  \label{fig:err_mult_2d}
\end{subfigure}%
\begin{subfigure}{.5\textwidth}
  \centering
  \includegraphics[width=\linewidth]{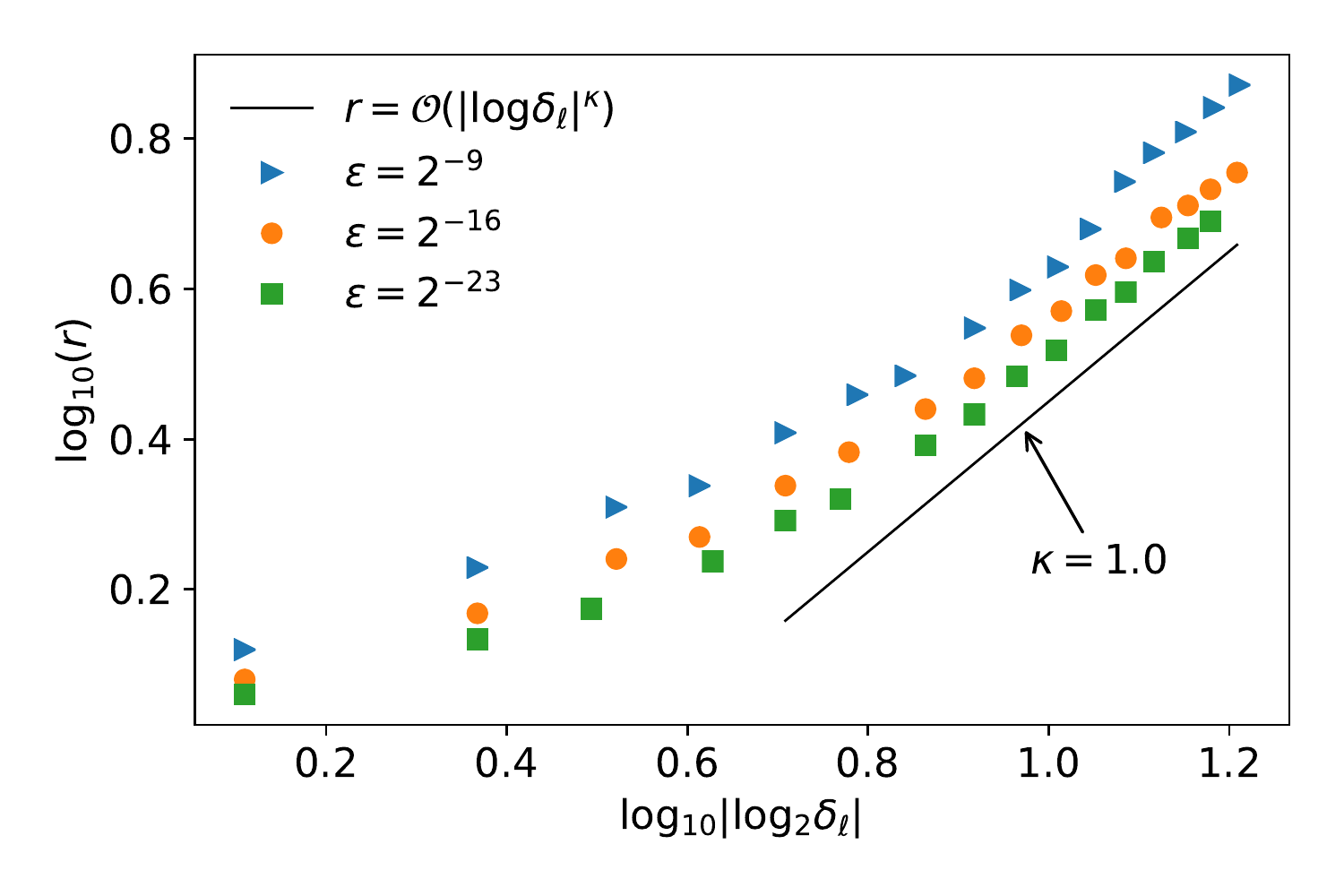}
  \caption{}
  \label{fig:rank_mult_2d}
\end{subfigure}
\caption{(a) Error $\delta_\ell$, defined in~\eqref{eq:ErrDelell}, w.r.t. the number of virtual
  grid levels for QTT-FEM for different $\varepsilon$.
  (b) Dependence of effective rank w.r.t. the
      seminorm $|\cdot|_{H^1(D)}$ of the error for different $\varepsilon$.}
\end{figure}
\section{Conclusions and Generalizations}
\label{sc:ConclRmk}
The present analysis and numerical experiments is focused on the
model linear elliptic multiscale problem~\eqref{Eq:ProblemEps}.
Here, the physical length scales are assumed to be \emph{asymptotically separated},
and the dependence of the diffusion coefficient $A^\varepsilon$ on the
fast variables $y_1,...,y_n$ is assumed to be periodic.
Similar structure and results hold for other types of PDEs
(e.g. \cite{BXVHH2014,BXVHH2015} and the references there).
The corresponding development of QTT-FE approaches for these
problem classes is a natural extension of the present analysis.

The assumptions allow to consider, instead of the original
$d$-dimensional multiscale problem,
a one-scale limit problem which is high-dimensional.
Analogous high-dimensional one-scale limit problems are obtained
for perforated materials, and for so-called reticulated structures,
as well as so-called lattice-materials; we refer to the survey
\cite{Cioranescu:2008:PeriodicUnfolding} and to the references there.
Additionally, we point out that
high-dimensional one-scale limit problems with the same,
tensorized structure as those considered here
arise also for certain \emph{non-periodic multiscale problems},
which fall into the class of the so-called \emph{homogenization structures},
as proposed by Nguetseng in \cite{NGuetsHomStrI}.
We also emphasize that analogous results are available for
\emph{nonlinear problems with multiple scales}; we refer to
\cite{VHH2008Monot} and the references there for further details.
The results of the present paper indicate that the resulting
(nonlinear) one-scale high-dimensional limit problems
can also be solved efficiently by QTT-FE discretization,
combined with a nonlinear solver.

We obtained the QTT rank bounds of the solution of the high-dimensional,
one-scale limit problem under strong (analyticity) assumptions on the
data which implied, as we showed,
the corresponding analyticity of the
solutions $u_i(x,y_1,...,y_i)$; this, in turn, allowed us to prove
bounds on the TT-rank of the solution that are logarithmic in accuracy.
This naturally leads to the question whether
analogous results can be expected in the case that we do not have analyticity.
Consider, for example,
the case where the unit cells $Y_i = (0,1)^d$ have `holes', i.e.
$Y_i = (0,1)^d \backslash O_i$, where $O_i \subset\subset Y_i$
is polyhedral, e.g. a cube centered at the point $(1/2,...,1/2)$
with edge length $1/2$.
The corresponding generalization of
unfolding homogenization is given in \cite{DCPDRZ2006},
In this case, the gradient $v_n$, given by~\eqref{eq:def-vi},
exhibits singularities on $\partial O_i$ with respect to the $i$th microscopic variable, for each $i=1,\dots,n$,
so that analyticity of $v_i$ with respect to
$y_i \in \overline{ Y_i \backslash O_i}$ can not be expected anymore.
Regularity results for the parametric unit-cell problems
    in countably normed spaces are available (for $n=1$ microscale
    and $d=2$ space dimensions) in \cite{MatMel03}.
    When combined with the QTT-FE approximations
    from \cite{KS:2017:QTTFE2d} (in space dimension $d=2$),
    also in this case, QTT-FE approximation rate and rank bounds
    completely analogous to the results in the present note can be obtained.

\clearpage

\section*{Appendix}

\subsection*{Proof of Lemma~\ref{Lm:Approx-ndh}}
\hfill
\begin{proof}
	Let $\mathsf{id}$ denote the identity transformation
	with respect to a scalar variable ranging in $\IntOO{0}{1}$.
	For all $L\in\N$ and $k\in\Set{1,\ldots,d}$,
	the errors bounded by the claim
	can be represented by telescoping sums as follows:
	\begin{multline}\nonumber
		v - \bar\varPi_i^L v
		=
		\sum_{k'=1}^d
		\bar\varPi_{i-1}^L
		\otimes
		\CuBr[3]{
			\Par[3]{
				\,
				\bigotimes_{k=1}^{k'-1}
				\pi^L
			}
			\otimes
			\Par[1]{ \mathsf{id} - \pi^L }
			\otimes
			\mathsf{id}^{\otimes (d-k')}
		}
		\,
		v
		\\
		+
		\sum_{j'=1}^i
		\sum_{k'=1}^d
		\bar\varPi_{j'-1}^L
		\otimes
		\CuBr[3]{
			\Par[3]{
				\,
				\bigotimes_{k=1}^{k'-1}
				\bar{\pi}^L
			}
			\otimes
			\Par[1]{ \mathsf{id} - \bar{\pi}^L }
			\otimes
			\mathsf{id}^{\otimes (d-k')}
		}
		\otimes
		\mathsf{id}^{\otimes (i-j')d}
		\otimes
		\mathsf{id}^{\otimes d}
		\,
		v
		\, ,
	\end{multline}
	\begin{multline}\nonumber
		w
		-
		\varPi_i^L w
		=
		\sum_{k'=1}^d
		\bar\varPi_{i-1}^L
		\otimes
		\CuBr[3]{
			\Par[3]{
				\,
				\bigotimes_{k=1}^{k'-1}
				\pi^L
			}
			\otimes
			\Par[1]{ \mathsf{id} - \pi^L }
			\otimes
			\mathsf{id}^{\otimes (d-k')}
		}
		\,
		w
		\\
		+
		\sum_{j'=1}^i
		\sum_{k'=1}^d
		\bar\varPi_{j'-1}^L
		\otimes
		\CuBr[3]{
			\Par[3]{
				\,
				\bigotimes_{k=1}^{k'-1}
				\bar{\pi}^L
			}
			\otimes
			\Par[1]{ \mathsf{id} - \bar{\pi}^L }
			\otimes
			\mathsf{id}^{\otimes (d-k')}
		}
		\otimes
		\mathsf{id}^{\otimes (i-j')d}
		\otimes
		\mathsf{id}^{\otimes d}
		\,
		w
		\, ,
	\end{multline}
	\begin{multline}\nonumber
		\partial_{i k} (w - \varPi_i^L w)
		=
		\sum_{k'=1}^d
		\bar\varPi_{i-1}^L
		\otimes
		\CuBr[3]{
			\Par[3]{
				\,
				\bigotimes_{k=1}^{k'-1}
				\pi^L
			}
			\otimes
			\partial_{ik}
			\Par[1]{ \mathsf{id} - \pi^L }
			\otimes
			\mathsf{id}^{\otimes (d-k')}
		}
		\,
		w
		\\
		+
		\sum_{j'=1}^i
		\sum_{k'=1}^d
		\bar\varPi_{j'-1}^L
		\otimes
		\CuBr[3]{
			\Par[3]{
				\,
				\bigotimes_{k=1}^{k'-1}
				\bar{\pi}^L
			}
			\otimes
			\Par[1]{ \mathsf{id} - \bar{\pi}^L }
			\otimes
			\mathsf{id}^{\otimes (d-k')}
		}
		\otimes
		\mathsf{id}^{\otimes (i-j')d}
		\otimes
		\mathsf{id}^{\otimes d}
		\partial_{ik} w
		\, .
	\end{multline}
	Applying Proposition~\ref{Pr:Appr01}
	to these representations, we obtain the claimed bounds.
\end{proof}

\subsection*{Proof of Lemma~\ref{Lm:Approx-ndp}}
\hfill
\begin{proof}
	The exponentials and shifted Chebyshev polynomials
	defined by~\eqref{Eq:DefExp} and~\eqref{Eq:DefCheb}
	form orthogonal bases in the spaces
	$L^2\IntOO{0}{1}$ and $L^2_{\omega}\IntOO{0}{1}$ respectively, where
	$\omega$ is the Chebyshev weight function given by~\eqref{Eq:ChebWeight}.
	It follows from the assumption that
	$w\in L^2_{\omega^{\otimes d} \otimes \mathsf{id}}(D \times \bm{Y}_i)$,
	so that
	$w$ can be represented by the following absolutely convergent series:
	\begin{equation}\label{Eq:L2expansion}
		w
		=
		\sum_{\alpha\in\Nz^d}
		\sum_{\beta_1\in\Z^d}
		\cdots
		\sum_{\beta_i\in\Z^d}
		c_{\alpha,\beta_1,\ldots,\beta_i}
		\,
			\Par[2]{
				\bigotimes_{k=1}^d
				\widetilde{T}_{\alpha_k}
			}
			\otimes
			\,
			\Par[2]{
				\bigotimes_{j=1}^i
				\bigotimes_{k=1}^d
				\widehat{T}_{\beta_{jk}}
			}
		\quad\text{in}\quad
		L^2_{\omega^{\otimes d} \otimes \mathsf{id}}(D \times \bm{Y}_i)
		\, ,
	\end{equation}
	where,
	due to~\eqref{Eq:OrthCheb} and~\eqref{Eq:OrthExp},
	the coefficients satisfy
	\begin{equation}\label{Eq:L2coefficient}
		c_{\alpha,\beta_1,\ldots,\beta_i}
		=
		\frac{\Abs{\kappa_\alpha}}{\pi^d}
		\,
		\IProd[2]{
			\Par[2]{
				\bigotimes_{k=1}^d
				\widetilde{T}_{\alpha_k}
			}
			\otimes
			\,
			\Par[2]{
				\bigotimes_{j=1}^i
				\bigotimes_{k=1}^d
				\widehat{T}_{\beta_{jk}}
			}
		}{\, w}_{
			L^2_{\omega^{\otimes d} \otimes \mathsf{id} }(D \times \bm{Y}_i)
		}
	\end{equation}
	for all $\alpha\in\Nz^d$ and $\beta_1,\ldots,\beta_i\in\Z^d$
	with
	$\kappa_0=1$, $\kappa_{\pm\alpha} = 2 \, (-1)^\alpha$ for each $\alpha \in\N$
	and $\kappa_\alpha = \kappa_{\alpha_1} \cdots \kappa_{\alpha_d}$
	for every $\alpha\in\N^d$.

	The entire function $\mathscr{z} \colon \C \to \C$
	given by $\mathscr{z}(\zeta)= (1-\cos 2\pi\zeta)/2$ for all $\zeta \in\C$
	bijectively maps
	each of the intervals
	$\IntOO{0}{1/2}$ and $\IntOO{1/2}{1}$
	onto $\IntOO{0}{1}$.
	Then, introducing
	$
		\mathscr{Z}
		=
		\mathscr{z}^{\otimes d} \otimes \mathsf{id}
		\colon
		\C^{(i+1)d}
		\to
		\C^{(i+1)d}
	$,
	we can substitute $\mathscr{Z}$ in~\eqref{Eq:L2coefficient} to
	express the coefficients of $w$ as follows:
	\begin{equation}\label{Eq:L2coefficient-sub}
		c_{\alpha,\beta_1,\ldots,\beta_i}
		=
		2^{-d}
		\!\!\!\!
		\sum_{\sigma\in \Set{\pm 1}^d}
		\!\!
		\widehat{c}_{\sigma \odot \alpha,\beta_1,\ldots,\beta_i}
	\end{equation}
	for all $\alpha\in\Nz^d$ and $\beta_1,\ldots,\beta_i\in\Z^d$,
	where $\sigma \odot \alpha = (\sigma_1 \alpha_1,\ldots,\sigma_d \alpha_d)$
	for any $\sigma\in \Set{\pm 1}^d$ and $\alpha\in\Nz^d$
	and
	\begin{equation}\label{Eq:L2coefficientZ}
		\widehat{c}_{\beta_0,\beta_1,\ldots,\beta_i}
		=
		\kappa_\alpha
		\,
		\IProd[2]{
			\Par[2]{
				\bigotimes_{k=1}^d
				\widehat{T}_{\beta_k}
			}
			\otimes
			\,
			\Par[2]{
				\bigotimes_{j=1}^i
				\bigotimes_{k=1}^d
				\widehat{T}_{\beta_{jk}}
			}
		}{\, w \circ \mathscr{Z}}_{
			L^2(D \times \bm{Y}_i)
		}
	\end{equation}
	for all $\beta_0,\beta_1,\ldots,\beta_i\in\Z^d$.

	For every
	$\delta>0$,
	the function $\mathscr{z}$ bijectively maps
	$
		\mathscr{S}_\delta
		=
		\Set[1]{
			\xi - \IU \eta
			\colon
			\xi\in\IntOO{0}{1}
			\, ,
			\eta\in\IntOO{0}{\delta}
		}
		\subset \C
	$
	onto
	$
		\mathcal{E}_\delta
		=
		\Set[1]{
			(1-a_\eta \cos 2\pi\xi)/2 - \IU \, (b_\eta \sin 2\pi\xi)/2
			\colon
			\xi\in\IntOO{0}{1}
			\, ,
			\eta\in\IntOO{0}{\delta}
		}
	$,
	where
	$a_\eta = \cosh 2\pi\eta$ 
	and
	$b_\eta = \sinh 2\pi\eta$ for every $\eta>0$. 
	Note that $\mathcal{E}_\delta \cup \IntOC{(1-a_\delta)/2}{1}$
	is the image of the standard
	open Bernstein ellipse with parameter $\rho = e^{2\pi\delta}$
	(with foci $\pm 1$ and semi-axes $a_\delta$ and $b_\delta$)
	under the affine mapping $\C \ni z \mapsto (1-z)/2 \in\C$.
	Since the function
	$w$ is analytic on $\overline{D \times \bm{Y}_i}$
	by assumption,
	it admits extension by analytic continuation to an open neighborhood of
	$\overline{D \times \bm{Y}_i}$.
	Specifically, for some $\delta_{i0},\delta_{i1},\ldots,\delta_{ii}>0$,
	it has an extension to $\overline{\mathcal{G}_i}$,
	where
	\[
		\mathcal{G}_i
		=
		\CuBr[4]{
			\bigtimes_{k=1}^d \mathcal{E}_{\delta_{i0}}
		}
		\times
		\CuBr[4]{
			\bigtimes_{j=1}^i
			\bigtimes_{k=1}^d
			\mathscr{S}_{\delta_{ij}}
		}
		\, ,
	\]
	that is holomorphic on $\mathcal{G}_i$
	and continuous on $\overline{\mathcal{G}_i}$.
	We identify the original function $w$ with this (unique) extension
	and set $M_i = \sup_{z \in \mathcal{G}_i} \Abs{w(z)}$.
	For the domain
	\[
		\mathcal{D}_i
		=
		\bigtimes_{j=0}^i
		\bigtimes_{k=1}^d
		\mathscr{S}_{\delta_{ij}}
		\, ,
	\]
	we have
	$\mathcal{G}_i = \mathscr{Z}(\mathcal{D}_i)$
	and $\sup_{\zeta \in \mathcal{D}_i} \Abs{(w \circ \mathscr{Z})(\zeta)} = M_i$.
	Furthermore,
	$w \circ \mathscr{Z}$ is holomorphic on $\mathcal{D}_i$,
	continuous on $\overline{\mathcal{D}_i}=\mathscr{Z}(\overline{\mathcal{D}_i})$
	and one-periodic with respect to
	each of its $(i+1)d$ variables.
	Using these properties and applying the Cauchy--Goursat theorem
	for the domain $\mathcal{D}_i$, we obtain
	\begin{equation}\nonumber
		\widehat{c}_{\beta_0,\beta_1,\ldots,\beta_i}
		=
		\kappa_\alpha
		\idotsint\displaylimits_{
			\bigtimes_{j=0}^i
			\IntCC{-\IU \delta_{ij}}{1-\IU \delta_{ij}}^d
		}
			\Par[2]{
				\bigotimes_{j=0}^i
				\bigotimes_{k=1}^d
				\widehat{T}_{\beta_{jk}}^*
			}
			\,
			\Par{ w \circ \mathscr{Z} }
	\end{equation}
	and hence
	\begin{equation}\label{Eq:L2coefficientBound}
		\Abs{\widehat{c}_{\beta_0,\beta_1,\ldots,\beta_i}}
		\leq
		M_i \QQ \kappa_\alpha
		\QQ
		\exp
		\Par[2]{
			- \sum_{j=0}^i 2 \pi \delta_{ij} \Abs{\beta_j}
			\,
		}
	\end{equation}
	for all $\beta_0,\beta_1,\ldots,\beta_i\in\Nz^d$.
	Then~\eqref{Eq:L2coefficient-sub} gives
	\begin{equation}\label{Eq:L2coefficientBound2}
		\Abs{c_{\alpha,\beta_1,\ldots,\beta_i}}
		\leq
		M_i \QQ \kappa_\alpha
		\QQ
		\exp
		\Par[2]{
			- \sum_{j=0}^i 2 \pi \delta_{ij} \Abs{\beta_j}
			\,
		}
	\end{equation}
	for all $\alpha\in\Nz^d$ and $\beta_1,\ldots,\beta_i\in\Z^d$.

	Now we set
	$\delta_* = \min\Set{\delta_{i0},\delta_{i1},\ldots,\delta_{ii}}$
	and verify the claimed bounds for
	$c = (2\pi\delta_*)^{-1}$, $p=\Ceil{c \QQ \log \epsilon^{-1}}$
	and a suitable positive constant $C$.
	Let
	$\mathcal{I}_{0} = \Set{0,1,\ldots,p - 1}$,
	$\mathcal{J}_{0} = \Set{0,\pm 1,\ldots,\pm(p - 1)}$
	and
	$\mathcal{I}_{1} = \Nz \SetMinus \mathcal{I}_{0}$,
	$\mathcal{J}_{1} = \Z \SetMinus \mathcal{J}_{0}$.
	Using the product index sets
	$\bm{\mathcal{I}}_{\mu} = \mathcal{I}_{\mu_1} \times \cdots \times \mathcal{I}_{\mu_d}$
	and
	$\boldsymbol{\mathcal{J}}_{\! \mu} = \mathcal{J}_{\mu_1} \times \cdots \times \mathcal{J}_{\mu_d}$
	with $\mu\in \Set{0,1}^d$,
	we can recast the expansion~\eqref{Eq:L2expansion}
	in $L^2_{\omega^{\otimes d} \otimes \mathsf{id}}(D \times \bm{Y}_i)$
	as follows:
	\begin{equation}\label{Eq:L2expansion2}
		w
		=
		\sum_{m=0}^{(i+1)d}
		\sum_{\substack{
			\mu,\nu_1,\ldots,\nu_i\in\Set{0,1}^d
			\colon
			\\
			\Abs{\mu}+\sum_{j=1}^i \Abs{\nu_j} = m
		}}
		\;
		\sum_{\substack{
			\alpha\in\boldsymbol{\mathcal{I}}_{\mu}
			\\
			\beta_1\in \boldsymbol{\mathcal{J}}_{\! \nu_1}
			\\
			\cdots
			\\
			\beta_i\in \boldsymbol{\mathcal{J}}_{\! \nu_i}
		}}
		\!\!\!\!
		c_{\alpha,\beta_1,\ldots,\beta_i}
		\,
			\Par[2]{
				\bigotimes_{k=1}^d
				\widetilde{T}_{\alpha_k}
			}
			\otimes
			\,
			\Par[2]{
				\bigotimes_{j=1}^i
				\bigotimes_{k=1}^d
				\widehat{T}_{\beta_{jk}}
			}
			\, .
	\end{equation}
	In the right-hand side of~\eqref{Eq:L2expansion2}, the term of the outer sum
	corresponding to $m=0$ is $\varPi_{ i,p} \, w$,
	and the remainder can be bounded using~\eqref{Eq:L2coefficientBound2}:
	\begin{multline}\label{Eq:L2expansionErrorNorm}
			\Norm{
				w - \varPi_{ i,p} \, w
			}_{L^\infty\Par{D \times \bm{Y}_i}}
			\leq
			\sum_{m=1}^{(i+1)d}
			\sum_{\substack{
				\mu,\nu_1,\ldots,\nu_i\in\Set{0,1}^d
				\colon
				\\
				\Abs{\mu}+\sum_{j=1}^i \Abs{\nu_j} = m
			}}
			\;
			\sum_{\substack{
				\alpha\in\boldsymbol{\mathcal{I}}_{\mu}
				\\
				\beta_1\in \boldsymbol{\mathcal{J}}_{\! \nu_1}
				\\
				\cdots
				\\
				\beta_i\in \boldsymbol{\mathcal{J}}_{\! \nu_i}
			}}
			\Abs{c_{\alpha,\beta_1,\ldots,\beta_i}}
			\\
			\leq
			\frac{ M_i \, 2^{d-1} \, 2^{id} }{ (1-\lambda)^{(i+1)d} }
			\!\!
			\sum_{m=1}^{(i+1)d}
			\!\!
			\epsilon^m
			\binom{(i+1)d \,}{m}
			\leq
			C_0
			\QQ
			\epsilon
			\, ,
	\end{multline}
	where $\lambda=e^{-2\pi\delta_*}\in\IntOO{0}{1}$ and
	$
		C_0
		=
		M_i \QQ (i+1) \QQ d
			\,
			2^{(i+1)d-1}
			\Par{ 1+\epsilon_0 }^{(i+1)d-1}
			/
		(1-\lambda)^{(i+1)d}
		> 0
	$.
	This gives the first of the bounds~\eqref{Eq:Approx-nd-bound}
	with any constant $C \geq C_0$, selected independently of $\epsilon$.

	For derivatives of the shifted Chebyshev polynomials and exponentials,
	we have $\Norm{\widetilde{T}_\alpha '}_{L^\infty\IntOO{0}{1}} = 2\alpha^2$
	for all $\alpha\in\Nz$
	and
	$\Norm{\widehat{T}_\beta '}_{L^\infty\IntOO{0}{1}} = 2\pi\Abs{\beta}$
	for all $\beta\in\Z$.
	Note that there exist positive constants $\gamma_1$ and $\gamma_2$
	such that
	$\sum_{\beta=r}^{\infty} \beta \lambda^\beta \leq \gamma_1 \QQ (1-\lambda)^{-1} r \sum_{\beta=r}^{\infty} \lambda^\beta$
	and
	$\sum_{\beta=r}^{\infty} \beta^2 \lambda^s \leq \gamma_2 \QQ (1-\lambda)^{-2} r^2 \sum_{\beta=r}^{\infty} \lambda^\beta$
	for any $r\in\Nz$.
	Using this, we obtain, as in~\eqref{Eq:L2expansionErrorNorm},
	the following inequalities:
	\begin{equation}\nonumber
			\Norm{
				\partial_{k}
				(w - \varPi_{ i,p} \, w)
			}_{L^\infty\Par{D \times \bm{Y}_i}}
			\leq
			\frac{2\gamma_2 \QQ C_0 \QQ \epsilon \QQ p^2 }{(1-\lambda)^2}
			\, ,
			\;
			\Norm{
				\partial_{jk}
				(w - \varPi_{ i,p} \, w)
			}_{L^\infty\Par{D \times \bm{Y}_i}}
			\leq
			\frac{2\pi \gamma_1 \QQ C_0 \QQ \epsilon \QQ p}{1-\lambda}
			\, .
	\end{equation}
	for all $k\in\Set{1,\ldots,d}$ and $j\in\Set{1,\ldots,i}$.
	This proves the last two of the bounds~\eqref{Eq:Approx-nd-bound}
	with a suitable positive constant $C$,
	which can be chosen independently of $\epsilon\in\IntOO{0}{\epsilon_0}$.
\end{proof}

\subsection{Proof of Lemma~\ref{Lm:Approx-ndhp}}
\hfill
\begin{proof}
	Let $L\in\N$.
	Using the triangle inequality,
	we bound the errors as follows:
	\begin{equation}\label{Lm:Approx-ndhp::Eq:triangle}
		\begin{aligned}
			\Norm{
				\partial_{i k}
				\QQ
				\Par{
					w
					-
					\varPi_i^L \varPi_{ i,p}
					\,
					w
				}
			}_\infty
			&\leq
			\Norm{
				\partial_{i k}
				\QQ
				\Par{
					\mathsf{id} - \varPi_{ i,p}
				}
				\,
				w
			}_\infty
			+
			\Norm{
				\partial_{i k}
				\QQ
				\Par{
					\mathsf{id} - \varPi_i^L
				}
				\QQ
				\varPi_{ i,p}
				\,
				w
			}_\infty
			\, ,
			\\
			\Norm{
				\partial_{i k} w
				-
				\bar\varPi_i^L
				\partial_{i k}
				\QQ
				\varPi_{ i,p}
				\QQ
				w
			}_{\infty}
			&\leq
			\Norm{
				\partial_{i k}
				( \mathsf{id} - \varPi_{ i,p} )
				\,
				w
			}_{\infty}
			+
			\Norm{
				( \mathsf{id} - \bar\varPi_i^L )
				\QQ
				\partial_{i k}
				\QQ
				\varPi_{ i,p}
				\,
				w
			}_{\infty}
		\end{aligned}
	\end{equation}
	for every $k\in\Set{1,\ldots,d}$.
	By Lemma~\ref{Lm:Approx-ndp},
	there exist positive constants
	$C_0$ and $c$ such that,
	for $p=\Ceil{c L}$,
	we have
	\begin{equation}\label{Lm:Approx-ndhp::Eq:ndpErr}
			\Norm{
				\partial_{k}
				\QQ
				( \mathsf{id} - \varPi_{ i,p} )
				\,
				w
			}_{\infty}
			\leq
			C_0 \, p^2 \QQ 2^{-L}
			\, ,
			\quad
			\Norm{
				\partial_{jk}
				\QQ
				( \mathsf{id} - \varPi_{ i,p} )
				\,
				w
			}_{\infty}
			\leq
			C_0 \, p \, 2^{-L}
			\, .
	\end{equation}

	Certain derivatives of $\varPi_{ i,p} \, w$ can be bounded in terms of
	first-order derivatives of $\varPi_{ i,p} \, w$ using
	the Bernstein's inequality for trigonometric polynomials.
	Applying it together with the bounds~\eqref{Lm:Approx-ndhp::Eq:ndpErr}
	and Lemma~\ref{Lm:Approx-ndh},
	we obtain
	\begin{multline}\label{Lm:Approx-ndhp::Eq:Err1}
			\Norm{
				\partial_{ik}
				\QQ
				( \mathsf{id} - \varPi_i^L )
				\QQ
				\varPi_{ i,p}
				\,
				w
			}_{
				\infty
			}
			\leq
			2^{-L}
			\sum_{k'=1}^d
			\Norm{\partial_{k'}\partial_{i k} \varPi_{ i,p} \, w}_{
				\infty
			}
			\\
			+
			2^{-L}
			\sum_{j'=1}^{i-1}
			\sum_{k'=1}^d
			\Norm{\partial_{j' k'}\partial_{i k} \varPi_{ i,p} \, w}_{
				\infty
			}
			+
			2^{-L}
			\sum_{k'=1}^d
			\Norm{\partial^2_{i k'} \partial_{i k} \varPi_{ i,p} \, w}_{
				\infty
			}
			\\
			\leq
			2^{-L}
			\sum_{k'=1}^d
			2 \pi p \, \Norm{\partial_{k'} \varPi_{ i,p} \, w}_\infty
			+
			2^{-L}
			\sum_{j'=1}^{i-1}
			\sum_{k'=1}^d
			2 \pi p \, \Norm{\partial_{j' k'} \varPi_{ i,p} \, w}_\infty
			\\
			+
			2^{-L}
			\sum_{k'=1}^d
			(2 \pi p)^2 \, \Norm{\partial_{i k} \varPi_{ i,p} \, w}_{
				\infty
			}
			\leq
			2^{-L}
			\sum_{k'=1}^d
			2 \pi p \, \CuBr[1]{
				\Norm{\partial_{k'} \QQ w}_\infty
				+
				C_0 \, p^2 \QQ 2^{-L}
			}
			\\
			+
			2^{-L}
			\sum_{j'=1}^i
			(2 \pi p)^2 \, \CuBr[1]{
				\Norm{\partial_{j k} \QQ w}_\infty
				+
				C_0 \, p^2 \QQ 2^{-L}
			}
			\leq
			C_1 \, p^2 \, 2^{-L}
	\end{multline}
	for every $k\in\Set{1,\ldots,d}$
	with a positive constant $C_1$ independent of $L$.
	The same approach leads to the bound
	\begin{multline}\label{Lm:Approx-ndhp::Eq:Err0}
			\Norm{
				( \mathsf{id} - \bar\varPi_i^L )
				\QQ
				\partial_{i k}
				\QQ
				\varPi_{ i,p}
				\,
				w
			}_{\infty}
			\leq
			2^{-L}
			\sum_{k'=1}^d
			\Norm{
					\partial_{k'}
					\partial_{i k}
				\QQ
				\varPi_{ i,p}
				\QQ
				w
			}_\infty
			+
			2^{-L}
			\sum_{j'=1}^{i}
			\sum_{k'=1}^d
			\Norm{
					\partial_{j' k'}
					\partial_{i k}
				\QQ
				\varPi_{ i,p}
				\QQ
				w
			}_\infty
			\\
			\leq
			2^{-L}
			\sum_{k'=1}^d
			2\pi p \,
			\Norm{
					\partial_{k'}
				\QQ
				\varPi_{ i,p}
				\QQ
				w
			}_\infty
			+
			2^{-L}
			\sum_{j'=1}^{i}
			\sum_{k'=1}^d
			2\pi p \,
			\Norm{
				\partial_{j' k'}
				\QQ
				\varPi_{ i,p}
				\QQ
				w
			}_\infty
			\\
			\leq
			2^{-L}
			\sum_{k'=1}^d
			2\pi p \,
			\CuBr[1]{
				\Norm{\partial_{k'} \QQ w}_\infty
				+
				C_0 \, p^2 \, 2^{-L}
			}
			+
			2^{-L}
			\sum_{j'=1}^{i}
			\sum_{k'=1}^d
			2\pi p \,
			\CuBr[1]{
				\Norm{\partial_{j' k'} \QQ w}_\infty
				+
				C_0 \, p \, 2^{-L}
			}
			\\
			\leq
			C_2 \, p^2 \, 2^{-L}
	\end{multline}
	for every $k\in\Set{1,\ldots,d}$
	with a positive constant $C_2$ independent of $L$.

	Combining inequalities~\eqref{Lm:Approx-ndhp::Eq:Err1}
	and~\eqref{Lm:Approx-ndhp::Eq:Err0}
	with~\eqref{Lm:Approx-ndhp::Eq:ndpErr}
	and~\eqref{Lm:Approx-ndhp::Eq:triangle},
	we obtain the claimed error bounds
	with $C = C_0 + \max\Set{C_1,C_2}$.
\end{proof}

\subsection{Auxiliary results for Lemma~\ref{lem:Savg}}
\hfill

Lemma~\ref{lem:Savg} is based on the following
auxiliary Lemmas~\ref{lem:WL-avg} and~\ref{lem:Qstar-avg}.
We formulate these lemmas in terms of intermediate, starred
finite element spaces with corresponding analysis operators and low-rank subspaces,
which reflect the iterative averaging of all the $n$ microscales,
as defined in Definition~\ref{def:unfolding-folding}.

First, similarly to as in~\eqref{eq:def-i-sp},
using the functions $\lambda_0,\ldots,\lambda_n$ of $\varepsilon$
from Assumption~\ref{ass:Analytic1},
we define the space
\begin{equation}\label{eq:def-i-sp-avg}
	\bar{V}_{\! \star \QQ i}^L
	=
	\bar{V}^{\lambda_i+L}
	\otimes
	\bigotimes_{j=i+1}^n
	\bar{V}^L
\end{equation}
for every $i\in\Set{0,1,\ldots,n}$,
so that
$\bar{V}_{\! \star \QQ 0}^L = \bar{V}_{\! n}^L$
and
$\bar{V}_{\! \star \QQ n}^L = \bar{V}_{\! 0}^{\lambda_n+L}$ 
Further, as in~\eqref{Eq:DefAnOp},
we define
an analysis operator
$
	\bar{\varPsi}_{\! \star \QQ i}^L
	\colon
	L^2(D \times \bm{Y}_i)
	\to
	\C^{2^{d(\lambda_i+L)+(n-i)dL}}
$
by setting
\begin{equation}\label{Eq:DefAnOpAvg}
	\bar{\varPsi}_{\! \star \QQ i}^L
	=
	\bar\varPsi^{\lambda_i+L}
	\otimes
	\bigotimes_{j=i+1}^n
	\bar\varPsi^L
\end{equation}
for every $i\in\Set{0,1,\ldots,n}$.
Then $\bar{\varPsi}_{\! \star \QQ 0}^L$ and $\bar{\varPsi}_{\! \star \QQ n}^L$
are identical to $\bar{\varPsi}_n^L$ and $\bar{\varPsi}_n^{\lambda_n+L}$ respectively.

Note that the starred finite-element spaces are
introduced in~\eqref{eq:def-i-sp-avg} so as to ensure
that
averaging an element of each of these spaces (except the last)
produces an element from the next space.
Indeed, the following embedding property follows from Definition~\ref{def:unfolding-folding}
and equality~\eqref{eq:def-i-sp-avg}.
\begin{lemma}\label{lem:WL-avg}
	For all $L\in\N$ and $i\in\Set{1,\ldots,n}$,
	we have
	$\mathcal{U}^\varepsilon_i \, \bar{V}_{\! \star \QQ i-1}^L \subset \bar{V}_{\! \star \QQ i}^L$.
\end{lemma}

In order to analyze how the structure of functions from $\bar{V}_n^L$
with coefficients from $\mathscr{Q}_n^L$ is transformed under averaging,
we define, for every $i\in\Set{0,1,\ldots,n}$,
	\begin{equation}\label{eq:ui-approx-SpaceAvg}
		\mathscr{Q}_{\star i}^L
		=
		\mathscr{S}_i^L
		\,
		\otimes
		\bigotimes_{j=i+1}^n \mathcal{P}_{\! \! \# (n+1-j) \QQ p_L}^{L,d}
		\subset
		\C^{2^{d(\lambda_i +L)+ (n-i)dL}}
		\: .
	\end{equation}
In particular, the so defined subspaces
$\mathscr{Q}_{\star 0}^L$ and $\mathscr{Q}_{\star n}^L$ coincide with
$\mathscr{Q}_n^L$ and $\mathscr{S}_n^L$,
given by~\eqref{eq:ui-approx-Space} and~\eqref{eq:ui-approx-SpaceAvg}
respectively. These intermediate subspaces satisfy the following relation.
%
\begin{lemma}\label{lem:Qstar-avg}
	For all $L\in\N$, $i\in\Set{1,\ldots,n}$
	and $v \in \bar{V}_{\! \star \QQ i-1}^L$ such that
	$\bar{\varPsi}_{\! \star \QQ i-1}^L \, v \in \mathscr{Q}_{\star i-1}^L$,
	we have
	$\bar{\varPsi}_{\! \star i}^L \, \mathcal{U}^\varepsilon_i \QQ v \in \mathscr{Q}_{\star i}^L$.
\end{lemma}
\begin{proof}
	Let us consider a function $v \in W_i^L$ such that
	$\bar{\varPsi}_{\! \star i-1}^L \, v = \Vec{\kappa} \otimes \Vec{\mu} \otimes \Vec{\nu}$
	with
	\begin{equation}\nonumber
		\Vec{\kappa} \in \mathscr{S}_{i-1}^L
		\, ,
		\quad
		\Vec{\mu} \in
		\mathcal{P}_{\! \! \# (n+1-i) \QQ p_L}^{L,d}
		\, ,
		\quad
		\Vec{\nu} \in
		\bigotimes_{j=i+1}^n \mathcal{P}_{\! \! \# (n+1-j) \QQ p_L}^{L,d}
	\end{equation}
	and show that
	$\bar{\varPsi}_{\! \star i}^L \, \mathcal{U}^\varepsilon_i \QQ v \in \mathscr{Q}_{\star i}^L$.
	Due to the linearity and tensor-product structure of $\mathscr{S}_{i-1}^L$,
	defined by~\eqref{eq:ui-approx-SpaceAvg},
	this is sufficient to verify the claim.

	Applying Definition~\ref{def:unfolding-folding} and Lemma~\ref{lem:WL-avg}
	to $v$,
	we obtain
	\begin{multline}\label{eqn:avg-tensor-product}
		\Par{ \bar{\varPsi}_{\! \star i}^L \, \mathcal{U}^\varepsilon_i \QQ v}_{
			(j-1)2^L + j_i, \, j_{i+1},\ldots, j_n
		}
		=
		\Par{ \mathcal{U}^\varepsilon_i \QQ v}
		\Par[3]{
			\frac{j-1}{2^{\lambda_i}}
			+
			\frac{j_i-\frac12}{2^{\lambda_i+L}}
			,
			\,
			\frac{j_{i+1}-\frac12}{2^L},\ldots, \,\frac{j_n-\frac12}{2^L}
		}
		\\
		=
		\tilde{\Vec{\kappa}}_j
		\,
		\Vec{\mu}_{j_i}
		\,
		\Vec{\nu}_{j_{i+1},\ldots,j_n}
	\end{multline}
	for all
	$
		j\in \mathcal{J}^{\lambda_i,d}
	$
	and
	$
		j_i,j_{i+1},\ldots,j_n\in \mathcal{J}^{L,d}
	$,
	where
	$\tilde{\Vec{\kappa}} = \tilde{\bm{M}}_i^L \Vec{\kappa}$
	is the coefficient tensor of the component of $v$
	with respect to the first variable averaged over scale $\varepsilon_i$:
	\[
		\tilde{\Vec{\kappa}}_j
		=
		2^{-d(L+\lambda_{i-1}-\lambda_i)}
		\!\!\!\!\!\!\!\!\!\!\!\!\!\!
		\sum_{j'\in \mathcal{J}^{L+\lambda_{i-1}-\lambda_i,d}}
		\!\!\!\!\!\!\!\!\!\!
		\Vec{\kappa}_{(j-1) \, 2^{L+\lambda_{i-1}-\lambda_i} + j' }
	\]
	for every
	$
		j\in \mathcal{J}^{\lambda_i,d}
	$.
	With this notation, relation~\eqref{eqn:avg-tensor-product}
	implies
	\begin{equation}\nonumber
		\bar{\varPsi}_{\! \star i}^{L} \, \mathcal{U}^\varepsilon_i \, v
		=
		\Par{\tilde{\bm{M}}_i^L \Vec{\kappa}}
		\otimes
		\Vec{\mu}
		\otimes
		\Vec{\nu}
		\, .
	\end{equation}
	Since, according to~\eqref{eq:ui-approx-SpaceAvgX},
	$\Par{\tilde{\bm{M}}_i^L \Vec{\kappa}} \otimes \Vec{\mu} \in \mathscr{S}_i^L$,
	the claimed inclusion
	$\bar{\varPsi}_{\! \star i}^L \, \mathcal{U}^\varepsilon_i \QQ v \in \mathscr{Q}_{\star i}^L$
	follows immediately from~\eqref{eq:ui-approx-SpaceAvg}.
\end{proof}
%

\def\cprime{$'$}

\end{document}